\documentclass[10pt]{amsart}
\usepackage{latexsym,amsmath,amssymb,graphicx,yhmath}
\usepackage{float}
\usepackage[bf, small]{caption}
\usepackage[all,web]{xy}
\usepackage{color}
\usepackage{hyperref}

\usepackage{xcolor}

\def\xyellowspace{%
  \sbox0{\colorbox{yellow}{\strut\ }}
  \dimen0=\wd0\relax
  \hskip0pt\cleaders\box0\hskip\dimen0\hskip0pt}

\begingroup
\catcode`\ =\active%
\gdef\makeyellowspace{\let \xyellowspace\catcode`\ =\active}%
\endgroup

\def\?#1{\colorbox{yellow}{\strut#1}}

\DeclareFontFamily{OT1}{rsfs10}{}
\DeclareFontShape{OT1}{rsfs10}{m}{n}{ <-> rsfs10 }{}
\DeclareMathAlphabet{\mathscript}{OT1}{rsfs10}{m}{n}

\DeclareMathOperator{\im}{Im}       
\DeclareMathOperator{\id}{id}       
\DeclareMathOperator{\Spec}{Spec}   
\DeclareMathOperator{\Hom}{Hom}     
\DeclareMathOperator{\Tors}{Tors}    
\DeclareMathOperator{\Pic}{Pic}     
\DeclareMathOperator{\Cl}{Cl}       
\DeclareMathOperator{\Cox}{Cox}     
\DeclareMathOperator{\Aut}{Aut}     
\DeclareMathOperator{\rk}{rk}       
\DeclareMathOperator{\Sing}{Sing}   
\DeclareMathOperator{\Mov}{Mov}     
\DeclareMathOperator{\Nef}{Nef}     
\DeclareMathOperator{\Eff}{Eff}     
\DeclareMathOperator{\Relint}{Relint}  
\DeclareMathOperator{\Int}{Int}     
\DeclareMathOperator{\diag}{diag}   
\DeclareMathOperator{\conv}{Conv}   
\DeclareMathOperator{\lcm}{lcm}     
\DeclareMathOperator{\HNF}{HNF}     

\makeatletter
\def\widebreve{\mathpalette\wide@breve}
\def\wide@breve#1#2{\sbox\z@{$#1#2$}%
     \mathop{\vbox{\m@th\ialign{##\crcr
\kern0.08em\brevefill#1{0.8\wd\z@}\crcr\noalign{\nointerlineskip}%
                    $\hss#1#2\hss$\crcr}}}\limits}
\def\brevefill#1#2{$\m@th\sbox\tw@{$#1($}%
  \hss\resizebox{#2}{\wd\tw@}{\rotatebox[origin=c]{90}{\upshape(}}\hss$}
\makeatletter

\title[Extended duality of toric varieties and Mirror Symmetry]{An extension of polar duality of toric varieties and its consequences in Mirror Symmetry}

\author[M. Rossi]{Michele Rossi}

\date{\today}

\address{Dipartimento di Matematica, Universit\`a di Torino,
via Carlo Alberto 10, 10123 Torino} \email{michele.rossi@unito.it}

\thanks{The author was partially supported by the I.N.D.A.M. as a member of the G.N.S.A.G.A.\\
 Author's ORCID:0000-0001-6191-2087}

\def \wrt{with respect to }

\def \a{\alpha }
\def \b{\beta }
\def \d{\delta }
\def \l{\lambda }
\def\ll{\boldsymbol{\lambda}}
\def \s{\sigma }
\def \D{\Delta }
\def \Ga{\Gamma }
\def \Si{\Sigma }
\def \g{\gamma}
\def \vf{\varphi}
\def \ve{\varepsilon}
\def \ét{\'{e}tale}
\def \aa{\mathbf{a}}
\def \bb{\mathbf{b}}
\def \cc{\mathbf{c}}
\def \e{\mathbf{e}}
\def \q{\mathbf{q}}

\def \uu{\mathbf{u}}
\def \v{\mathbf{v}}
\def \n{\mathbf{n}}
\def \m{\mathbf{m}}
\def \w{\mathbf{w}}
\def \tt{\mathbf{t}}
\def \z{\mathbf{z}}
\def \x{\mathbf{x}}
\def \y{\mathbf{y}}
\def \1{\mathbf{1}}
\def \0{\mathbf{0}}

\def\P{{\mathbb{P}}}

\def\p2{\mathbb{P}^2}
\def\p3{\mathbb{P}^3}
\def\p4{\mathbb{P}^4}

\def\cv#1{\wideparen{#1}}

\def\cO{\mathcal{O}}
\def\cY{\mathcal{Y}}

\def\rk{\operatorname{rk}}

\def\GL{\operatorname{GL}}

\def\Mat{\operatorname{Mat}}

\def\Z{\mathbb{Z}}
\def\FF{\mathbb{F}}

\def\C{\mathbb{C}}

\def\R{\mathbb{R}}
\def\M{\mathbf{M}}
\def\Q{\mathbb{Q}}
\def\N{\mathbb{N}}
\def\NN{\nabla}
\def\T{\mathbb{T}}

\def\L{\Lambda}

\def\XX{\mathbb{X}}

\def\SF{\mathcal{SF}}

\def\I{\mathcal{I}}

\def\Ga{\Gamma}

\def\Weil{\mathcal{W}_T}

\def\U1{\mathfrak{U}^{(1)}}
\def\hv{h_{\text{\rm T}}}

\theoremstyle{plain}
\newtheorem{theorem}{Theorem}[section]
\newtheorem{proposition}[theorem]{Proposition}

\newtheorem{thm-def}[theorem]{Theorem--Definition}
\newtheorem{corollary}[theorem]{Corollary}
\newtheorem{conjecture}[theorem]{Conjecture}
\newtheorem{lemma}[theorem]{Lemma}

\newtheorem*{a-proposition}{Proposition}

\theoremstyle{remark}
\newtheorem{remark}[theorem]{Remark}
\newtheorem{remarks}[theorem]{Remarks}
\newtheorem{example}[theorem]{Example}

\theoremstyle{definition}
\newtheorem{definition}[theorem]{Definition}
\newtheorem{algorithm}[theorem]{Algorithm}

\newtheorem*{step I}{Step I}
\newtheorem*{step II}{Step II}
\newtheorem*{step III}{Step III}
\newtheorem*{step IV}{Step IV}

\newtheorem*{acknowledgements}{Acknowledgements}

\newcommand{\halfline}{\vskip6pt}
\newcommand{\cy}{Ca\-la\-bi-Yau }
\newcommand{\ka}{K\"{a}hler }

\begin{document}


\begin{abstract} The present paper is dedicated to illustrating an extension of polar duality between Fano toric varieties to a more general duality, called \emph{framed} duality, so giving rise to a powerful and unified method of producing mirror partners of hypersurfaces and complete intersections in toric varieties, of any Kodaira dimension. In particular, the class of projective hypersurfaces and their mirror partners are studied in detail. Moreover, many connections with known Landau-Ginzburg mirror models, Homological Mirror Symmetry and Intrinsic Mirror Symmetry,  are discussed.
\end{abstract}

\maketitle

\tableofcontents

\section*{Introduction}

Polar duality between reflexive polytopes gives the well known Batyrev duality between Fano toric varieties, inducing a mirror symmetry between generic sections of their anti-canonical divisors \cite{Batyrev94}. Borisov and Batyrev extended this duality to complete intersections described by a nef partition of the anti-canonical divisor of a Fano toric variety \cite{Borisov}, \cite{BB96}. By thinking of Batyrev duality as a duality between toric varieties \emph{framed} by their anti-canonical divisor, the present paper is devoted to show how \emph{deforming} the Batyrev-Borisov duality by allowing a more general framing, in principle just given by an effective torus invariant Weil divisor (see Definitions~\ref{def:ftv} and \ref{def:wftv}). In general, such a deformed correspondence, here called \emph{framed} duality ($f$-duality), between framed and weakly framed toric varieties, is not involutive, but imposing some further conditions on the framing gives back an involutive duality, here called a \emph{calibrated $f$-process}, incorporating the classical duality between Fano toric varieties as a very particular case. Roughly speaking, $f$-duality behaves as follows. Assume $Y$ be a hypersurface of a complete toric variety $X$ and let $D_Y$ be an effective torus invariant Weil divisor (the framing) of $X$ such that $Y\in|D_Y|$, that is $Y\sim D_Y$: actually we are considering the family of hypersurfaces described by the linear system $|D_Y|$. A suitable combinatorial procedure $(X,D_Y)\rightsquigarrow (\XX,D^\vee)$ gives back a dual toric variety $\XX$ and a dual framing $D^\vee$ whose linear system $|D^\vee|$ describes an $f$-mirror family of $|D_Y|$ as a family of hypersurfaces of $\XX$. Calibration means that the same combinatorial procedure applied to $(\XX,D^\vee)$ gives back $(X,D_Y)$ as an $f$-dual framed toric variety, that is $(X,D_Y)\leftrightsquigarrow (\XX,D^\vee)$ is an involutive process. In particular, if $Y$ is a \cy hypersurface of a Fano toric variety $X$, then the $f$-process is involutive giving precisely the Batyrev duality. Theorems~\ref{thm:Deltatriviale} and \ref{thm:Krawitz} will give equivalent conditions to calibration, but it is difficult to understand how general these conditions are. Certainly there exist framed toric varieties (ftv) do not admitting any calibrated $f$-process even if the framing is given by the anti-canonical divisor (see Example~\ref{ex:anticanonico}).

The main goal ot this paper is explaining how $f$-duality is able to describe, in a unified way, an enormous number of $f$-mirror symmetric pairs of not necessarily Calabi-Yau hypersurfaces and complete intersections in toric varieties, sensibly improving the current knowledge of mirror partners of non \cy varieties (see e.g. \cite{Givental96}, \cite{Hori-Vafa}, \cite{Seidel}, \cite{Efimov}, \cite{KKOY}, \cite{Krawitz}, \cite{HSSW}, \cite{GKR}). Here the word ``mirror'' has to be considered in a broader sense, outside the \cy setup: the reader is referred to discussion developed in \S~\ref{ssez:mirrortest},~\ref{ssez:MWeb}.

 For instance, as a very particular but interesting case, a generic projective hypersurface in $\P^n$ of degree $d\geq n+1$, can be thought of a framing of $\P^n$, so admitting (at least) one $f$-mirror dual partner given by an hypersurface in a suitable finite quotient of a weighted projective $n$-space, whose weights are essentially assigned by the framing itself (see \S\ref{sez:ipersuperfici}). This construction turns out to be a nice extention to higher degrees of the pioneering description of a mirror partner of the quintic threefold given by Greene and Plesser \cite{GP}. Moreover, for lower degrees, a generic projective hypersurface in $\P^n$ of degree $d\leq n$ can be thought of a weak framing of $\P^n$, whose associated $f$-mirror dual partner can no more be a complete variety, proposing a rational re-parametrization of Landau-Ginzburg (LG) mirror models proposed by Givental \cite{Givental96}, \cite{Givental-ICM} (see \S\ref{ssez:NegKod}).

 Furthermore, $f$-duality turns out to extend, in a unified construction, many known dualities between Calabi-Yau hypersurfaces and complete intersections in toric varieties: this is the case of the Berglund-H\"{u}bsch duality \cite{Berglund-Hubsch} and, more ge\-ne\-rally, of the recent Artebani-Comparin-Guilbot duality \cite{ACG} (see \S\ref{ssez:BHK} and \ref{ssez:ACG}).

 More in general, $f$-duality opens up to a lot of stimulating connections with many expects of the current status of art of research in mirror symmetry. Here is a list of interesting applications developed in the following.

\halfline \noindent \textbf{Givental\,\&\,Hori-Vafa mirroring: two sides of a same coin.} $f$-duality suggests a sui\-ta\-ble re-parameterization of the Landau-Ginzburg mirror model of a complete intersection in a toric variety, proposed by Hori and Vafa in 2000 \cite{Hori-Vafa}, so getting two interesting consequences: the first one is that $f$-duality turns out to exhibit a suitable compactification of this re-parameterized LG mirror model, extending to higher degrees the Hori and Vafa construction for \cy projective hypersurfaces (see \S\ref{ssez:HoriVafa}); the second one is that Hori-Vafa and Givental mirroring processes turn out to be two sides of a same coin. The latter has been recently observed by Clarke (see \cite{Clarke}, in particular \S~7) and $f$-duality gives a further confirmation. More deeply, $f$-duality seems providing a combinatorial translation of the duality proposed by Clarke in terms of an exchanging between linear data associated with a toric variety \cite[\S~3]{Clarke}.

Since most of the currently proposed LG mirror models for varieties of general type are modelled on the so called \emph{Hori-Vafa recipe}, the previous observation introduces possible re-pa\-ra\-me\-te\-ri\-za\-tions of all these LG models, which should be compared with the known ones from the point of view of Homological Mirror Symmetry (HMS) (see considerations given in \S\ref{ssez:iperellittica}).

\halfline \noindent \textbf{Landau-Ginzburg/Hypersurface correspondence.} By observing a natural Landau-Ginzburg/Hypersurface correspondence, extending the LG/CY correspondence studied by \cite{Chiodo-Ruan} (see \S\ref{sssez:LG/Hyp}), a framing turns out to be the hypersurface counterpart of the ``anchoring at infinity'' of the compactified LG models proposed by Katzarkov, Kontsevich and Pantev \cite{KKP} (see \ref{ssez:KKP-compct}), so opening the door to conceivable connections with the log-geometry of the Gross-Siebert Intrinsic Mirror Symmetry \cite{GS-IMS}.

 Moreover, $f$-duality explains quite well why, passing from a framing to a weak framing, that is, losing positivity properties, translates in losing completeness properties of the associated mirror partner, so well justifying a description of mirror symmetry in terms of a duality between associated LG models (see Remark~\ref{rem:negKod}).

\halfline \noindent \textbf{Multiple mirror phenomenon and the Mirror Web.}
 A further important remark, is that, since $f$-duality is a duality between framed toric varieties, that is, between pairs given by a complete toric variety and a sufficiently positive torus invariant Weil divisor, multiple mirror partners can, in principle, be assigned by a changing of framing in the same linear equivalence class (see \S\ref{ssez:MWeb}). This means that, one should think of mirror duality more in terms of a connection between nodes in a web (the Mirror Web) rather than a phenomenon connecting pairs of mirror partners, that is, a symmetry, as done for \cy varieties. Notice that the multiple mirror phenomenon is a well known one, also for \cy varieties, after e.g. the R{\o}dland example \cite{Rodland} (see Remark~\ref{rem:mult.mirr} and references therein).

\halfline \noindent \textbf{Beyond the toric setup.}
 Following the lines given by Batyrev for \cy varieties in \cite{Batyrev02}, a conjectural approach, to extending $f$-mirror symmetry beyond a toric embedding, is sketched in \S\ref{ssez:Tdegenarazione}, by means of toric degeneration and geometric transitions.

\halfline \noindent \textbf{Mirror theorems}
 As Batyrev-Borisov duality, $f$-duality is just a construction to propose candidate mirror partners. After that, one has to prove they are effectively mirror partners, by checking various instances of mirror symmetry. Beyond the \cy setup, understanding which are those mirror symmetric instances is a bit more involved (see \S\ref{ssez:mirrortest}). Probably, the deepest way of checking mirror symmetry is the one proposed by Kontsevich's HMS. But this seems to be a very difficult approach and we defer it to future works. In this paper, a large section is dedicated to check several matching of (stringy) Hodge numbers in the case of projective hypersurfaces of non-negative Kodaira dimension (see \S\ref{sez:ipersuperfici}). One side of this check (we call $A$-side) turns out to be easily computable (Theorem~\ref{thm:m*=k}). The other side of this check (so called $B$-side) is sensibly more intricate. We will state main results in Theorem~\ref{thm:B-mirror}, deferring their proof to \cite{R-fpCI}, where analogous computation are developed in the broader context of toric complete intersections.

 Anyway, this is a case by case checking, quickly becoming essentially impossible for more general hypersurfaces and complete intersections in toric varieties, due to the wild singularities $f$-duality produces. According with Chiodo and Ruan \cite{Chiodo-Ruan}, it is generally believed that considering suitably associated LG models may sensibly simplify singularities and give rise to alternative way of checking mirror symmetry. The already mentioned LG/Hypersurface correspondence, presented in \S\ref{ssez:K-dualita}, allows one to drawing an alternative conjectural approach to checking mirror symmetry (see Remark~\ref{rem:K-dualita}).

 \halfline
 This paper is organized as follows.
 \S\ref{sez:preliminari} is devoted to introduce notation on toric varieties, their divisors, hypersurfaces and associated stratifications. \S\ref{sez:dualita-ftv} is  dedicated to the definition of framed toric varieties (ftv) and framed duality. Then \S\ref{sez:dualita-hyp} is devoted to present mirror symmetric consequences of $f$-duality for hypersurfaces in complete toric varieties. In \S\ref{sez:framingPn} an important class of framed toric varieties admitting a calibrated $f$-process is presented, namely projective spaces endowed with suitable framings. In the following \S\ref{sez:ipersuperfici}, all these considerations are applied to the important class of examples given by hypersurfaces in $\P^n$ of degree $d\ge n+1$. Then in \S\ref{sez:CI}, $f$-duality is extended to complete intersections subvarieties in complete toric varieties. \S\ref{sez:wftv} is devoted to introduce the concept of a weak framing and weakly framed toric varieties (wftv) and $f$-mirror symmetry of negative Kodaira dimension hypersurfaces: in particular, $f$-mirrors of hypersurfaces in $\P^n$ of degree $d\le n$ are considered and showed matching the Givental LG mirror models.  Finally in \S\ref{sez:open} many further considerations and open problems are collected, ending up with studying the $f$-mirror model of the general hyperelliptic curve, to propose a comparison with an example which has been extensively studied in the literature \cite{KKOY}, \cite{Seidel}, \cite{Efimov}.

 \begin{acknowledgements}
   It is a pleasure to thank M.~Artebani for several clarifications about many aspects treated in \cite{ACG} and T.~H\"ubsch for his interested comments and interesting suggestions, giving rise to perspectives in \S\ref{ssez:Hubsch}. Many thanks also to S.~Filippini, for useful conversation during her last visit in Turin, and to G.~Bini for his considerations. Last but not least, many thanks to the unknown referee for his constructive encouragement and suggestions, greatly improving the final presentation of this work.

Many computations and proofs' prototypes have been partially performed by means of several Maple routines, mostly of them jointly written with L.~Terracini, and some of them based on the Maple package \texttt{Convex} \cite{Convex}.
 \end{acknowledgements}

\section{Preliminaries and notation on toric varieties}\label{sez:preliminari}

A \emph{$n$--dimensional toric variety} is an algebraic normal variety $X$ containing the \emph{torus} $T:=(\C^*)^n$ as a Zariski open subset such that the natural multiplicative self--action of the torus can be extended to an action $T\times X\rightarrow X$.

Let us quickly recall the classical approach to toric varieties by means of \emph{cones} and \emph{fans}. For proofs and details the interested reader is referred to the extensive treatments \cite{Danilov}, \cite{Fulton}, \cite{Oda} and the recent and quite comprehensive \cite{CLS}.

\noindent As usual $M$ denotes the \emph{group of characters} $\chi : T \to \C^*$ of $T$ and $N$ the \emph{group of 1--parameter subgroups} $\lambda : \C^* \to T$. It follows that $M$ and $N$ are $n$--dimensional dual lattices via the pairing
\begin{equation*}
\begin{array}{ccc}
M\times N & \longrightarrow & \Hom(\C^*,\C^*)\cong\C^*\\
 \left( \chi,\lambda \right) & \longmapsto
& \chi\circ\lambda
\end{array}
\end{equation*}
which translates into the standard paring $\langle u,v\rangle=\sum u_i v_i$ under the identifications $M\cong\Z^n\cong N$ obtained by setting $\chi(\tt)=\tt^{\uu}:=\prod t_i^{u_i}$ and $\lambda(t)=t^{\v}:=(t^{v_1},\ldots,t^{v_n})$.

\subsection{Cones and affine toric varieties}\label{ssez:TV}

Define $N_{\R}:=N\otimes \R$ and $M_{\R}:=M\otimes\R\cong \Hom(N,\Z)\otimes\R \cong \Hom(N_{\R},\R)$.

\noindent A \emph{convex polyhedral cone} (or simply a \emph{cone}) $\sigma$ is the subset of $N_{\R}$ defined by
\begin{equation*}
    \sigma = \langle \v_1,\ldots,\v_s\rangle:=\{ r_1 \v_1 + \dots + r_s \v_s \in N_{\R} \mid r_i\in\R_{\geq 0} \}
\end{equation*}
Vectors $\v_1,\ldots,\v_s\in N_{\R}$ are said \emph{to generate} $\sigma$; $\v_i$ is called a \emph{primitive} generator if it generates the semigroup $\langle\v_i\rangle\cap N$. A cone $\s=\langle \v_1,\ldots,\v_s\rangle$ is called \emph{rational} if $\v_1,\ldots,\v_s\in N$, \emph{simplicial} if $\v_1,\ldots,\v_s$ are $\R$--linear independent and \emph{non-singular} if primitive generators $\v_1,\ldots,\v_s$ can be extended to giving a basis of the lattice $N$.

\noindent A cone $\s$ is called \emph{strongly convex} or \emph{pointed} if it does not contain a linear subspace of positive dimension of $N_{\R}$.

\noindent The \emph{dual cone $\s^{\vee}$ of $\s$} is the subset of $M_{\R}$ defined by
\begin{equation*}
    \sigma^{\vee} = \{ \uu \in M_{\R} \mid \forall\ \v \in \sigma \quad \langle \uu, \v \rangle \ge 0 \}
\end{equation*}
A \emph{face $\tau$ of $\s$} (denoted by $\tau <\s$) is the subset defined by
\begin{equation*}
    \tau = \sigma \cap \uu^{\bot} = \{\v \in \sigma \mid \langle \uu, \v \rangle = 0 \}
\end{equation*}
for some $\uu\in \sigma ^{\vee}$. Observe that also $\tau$ is a cone.

\noindent A \emph{facet} $\tau$ of a cone $\s$ is a codimension 1 face, denoted by $\tau<^1\s$.

\noindent Gordon's Lemma ensures that the semigroup $S_{\s}:=\s^{\vee}\cap M$ is \emph{finitely generated}. Then also the associated $\C$--algebra $A_{\s}:=\C[S_{\s}]$ is finitely generated. A choice of $m$ generators gives a presentation of $A_{\s}$
\begin{equation*}
    A_{\s}\cong \C[X_1,\dots,X_m]/I_{\s}
\end{equation*}
Then $U_{\s}:=\Spec(A_{\s})\subset\C^m$
is an \emph{affine toric variety}. Since a closed point $x\in U_{\s}$ is an evaluation of elements in $\C[S_{\s}]$ satisfying the relations generating $I_{\s}$, then it can be identified with a semigroup morphism $x:S_{\s}\rightarrow\C$ assigned by thinking of $\C$ as a multiplicative semigroup. In particular the \emph{characteristic morphism}
\begin{equation}\label{caratteristico}
\begin{array}{cccc}
x_{\s}&:\s^{\vee}\cap M & \longrightarrow & \C\\
      &      \uu & \longmapsto & \left\{\begin{array}{cc}
                                         1 & \text{if $\uu\in\s^{\bot}$} \\
                                         0 & \text{otherwise}
                                       \end{array}
      \right.
\end{array}
\end{equation}
which is well defined since $\s^{\bot}<\s^{\vee}$, defines a \emph{characteristic point} $x_{\s}\in U_{\s}$ whose torus orbit $O_{\s}$ turns out to be a $(n-\dim(\s))$--dimensional torus embedded in $U_{\s}$.

\subsection{Fans and toric varieties}\label{ssez:fans&tv}

A \emph{fan} $\Si$ is a finite set of cones $\s\subset N_{\R}$ such that
\begin{enumerate}
  \item for any cone $\s\in\Si$ and for any face $\tau<\s$ then $\tau\in\Si$,
  \item for any $\s,\tau\in\Si$ then $\s\cap\tau<\s$ and $\s\cap\tau<\tau$.
\end{enumerate}
For every $i$ with $0\leq i\leq n$ denote by $\Si(i)\subset \Si$ the subset of $i$--dimensional cones, called the \emph{$i$--skeleton of $\Si$}.

\noindent A fan $\Si$ is called \emph{simplicial} if every cone $\s\in\Si$ is rational and simplicial, and is called \emph{non-singular} if every such cone is non-singular. The \emph{support} of a fan $\Si$ is the subset $|\Si|\subset N_{\R}$ obtained as the union of all of its cones i.e.
\begin{equation*}
    |\Si|:= \bigcup_{\s\in\Si} \s \subset N_{\R}\ .
\end{equation*}
If $|\Si|=N_{\R}$ then $\Si$ will be called \emph{complete}.

Since for any face $\tau <\s$ the semigroup $S_{\s}$ turns out to be a sub-semigroup of $S_{\tau}$, there is an induced immersion $U_{\tau}\hookrightarrow U_{\s}$ between the associated affine toric varieties which embeds $U_{\tau}$ as a principal open subset of $U_{\s}$. Given a fan $\Si$ one can construct \emph{an associated toric variety $X(\Si)$} by patching all the affine toric varieties $\{U_{\s}\ |\ \s\in\Si \}$ along the principal open subsets associated with any common face. Moreover \emph{for every toric variety $X$ there exists a fan $\Si$ such that $X\cong X(\Si)$}. It turns out that:
\begin{itemize}
  \item \emph{$X(\Si)$ is non-singular if and only if the fan $\Si$ is non-singular,}
  \item \emph{$X(\Si)$ is complete if and only if the fan $\Si$ is complete.}
\end{itemize}
Let $\v_\rho$ be a primitive generator of the ray $\rho\in \Si(1)$. Up to an identification $N\cong\Z^n$, where $n:=\dim X$, and setting $m:=|\Si(1)|$
$$V=\left(\v_\rho\,|\,\rho\in\Si(1)\right)=(\v_1\,\cdots\,\v_m)$$
gives a $n\times m$ integer matrix called \emph{a fan matrix of $\Si$}. Notice that $\Si$ determines $V$ up to the choice of a basis of $N$ and of a permutation of columns (i.e. generators $\v_\rho$), that is, $V$ and $V'$ are \emph{equivalent} fan matrices if
\begin{equation}\label{M-equivalenza}
  \exists\,A\in\GL(n,\Z)\,,\ \exists\,B\in\mathfrak{S}_m\leq \GL(m,\Z)\quad V'=A\cdot V\cdot B
\end{equation}

\subsection{Divisors on Toric varieties}

Let $\mathcal{W}(X)$ denote the group of Weil divisors of a toric variety $X=X(\Si)$. Then its subgroup of \emph{torus--invariant Weil divisors} is given by
\begin{equation*}
    \mathcal{W}_T(X)=\left\langle D_{\rho } \mid \rho \in \Sigma (1)\right\rangle_{\Z} =
    \bigoplus_{\rho \in \Sigma (1)}\Z\cdot D_{\rho }
\end{equation*}
where $D_{\rho}=\overline{\T\cdot x_\rho}$, being $\T\cong\Hom(N,\C^*)$ the acting torus and $x_\rho$ the distinguished point of $\rho$, as defined in (\ref{caratteristico}). Let $\mathcal{P}(X)\subset\mathcal{W}(X)$ be the subgroup of \emph{principal divisors} and $\v_{\rho}$ be the generator of the monoid $\rho\cap N$. Then the morphism
\begin{equation}\label{div}
\begin{array}{llll}
div : & M & \longrightarrow & \mathcal{P}(X)\cap \mathcal{W}_{T}(X)=:
\mathcal{P}_{T}(X) \\
& \uu & \longmapsto & div(\uu):=\sum_{\rho \in \Sigma (1)}\langle \uu,\v_{\rho }\rangle
D_{\rho }
\end{array}
\end{equation}
is surjective.
 Let $V=(\v_1,\ldots,\v_{n+r})$ be a fan matrix of $\Si$, with respect to a chosen identification $N\cong\Z^n$. Then the transposed matrix $V^T$ is a representative matrix of the $\Z$-linear morphism $div$ defined in (\ref{div}), with respect to the basis $\{D_1,\ldots,D_{n+r}\}$ of $\mathcal{W}_T(X)$.

Let $\Pic(X)$ be the group of line bundles modulo isomorphism. It is well known that for an \emph{irreducible} variety $X$ the map $D\mapsto\mathcal{O}_X(D)$ induces an isomorphism $\mathcal{C}(X)/\mathcal{P}(X)\cong\Pic(X)$, where $\mathcal{C}(X)\subset\mathcal{W}(X)$ denotes the subgroup of Cartier divisors. The divisor class group is defined as the group of Weil divisors modulo rational (hence linear) equivalence, i.e. $\Cl(X):=\mathcal{W}(X)/\mathcal{P}(X)$. Then the inclusion $\mathcal{C}(X)\subset\mathcal{W}(X)$ passes through the quotient giving an immersion $\Pic(X)\hookrightarrow \Cl(X)$.

 A toric variety $X=X(\Si)$ is called \emph{non-degenerate} if the support $|\Si|$ spans $N_{\R}$\,: in particular this means that it cannot admit torus factors, or, equivalently, that $H^0(X,\cO_X^*)\cong\C^*$. Then, the cardinality of the 1-skeleton is given by
 $$|\Si(1)|=n+r$$
 where $r:=\rk \Pic(X)\geq 1$ is the \emph{Picard number of} $X$, also called \emph{the rank of $X$}, in the following.

\begin{definition}\cite[Def.~3.10]{RT-LA&GD}\label{def:Fmatrice} An \emph{$F$--matrix} is a $n\times (n+r)$ matrix  $V$ with integer entries, satisfying the conditions:
\begin{itemize}
\item[a)] $\rk(V)=n$;
\item[b)] $V$ is \emph{$F$--complete} i.e. $\langle V\rangle=N_{\R}\cong\R^n$ \cite[Def.~3.4]{RT-LA&GD};
\item[c)] all the columns of $V$ are non zero;
\item[d)] if ${\bf  v}$ is a column of $V$, then $V$ does not contain another column of the form $\lambda  {\bf  v}$ where $\lambda>0$ is a real number.
\end{itemize}
A $F$--matrix $V$ is called \emph{reduced} if every column of $V$ is composed by coprime entries \cite[Def.~3.13]{RT-LA&GD}.
\end{definition}
For instance, a fan matrix of a complete toric variety $X(\Si)$ is always a reduced $F$--matrix.

\subsubsection{Notation}\label{sssez:SF} Given a reduced $F$-matrix $V$, in the following $\SF(V)$ will denote the set of all complete and simplicial fans whose 1-skeleton  is given by all the rays generated by the columns of $V$. Moreover,
$$\P\SF(V)\subset\SF(V)$$
will denote the subset of those fans whose associated toric variety $X(\Si)$ is projective.

\subsection{Polytopes of divisors and associated fans and varieties}\label{ssez:politopi}

A \emph{polytope} $\D\subset M_{\R}$ is the convex hull of a finite set $S$ of points, that is $\D=\conv(S)$.

\noindent If $S\subseteq M$ then $\D$ is called a \emph{lattice polytope}.
When $\D$ is a full dimensional polytope its presentation as an intersection of closed half-spaces has an especially nice form, because each facet $\Phi<^1\D$ has a unique supporting affine hyperplane. We denote such an hyperplane and the corresponding closed half-space as
\begin{equation*}
  H_\Phi=\{\m\in\M_\R\,|\,\langle\m,\n_\phi\rangle=-a_\Phi\}\quad,\quad H^+_\Phi=\{\m\in\M_\R\,|\,\langle\m,\n_\phi\rangle\geq-a_\Phi\}
\end{equation*}
where $(\n_\Phi,a_\Phi)\in N_\R\times \R$ is unique up to multiplication by  a positive real number. We call $\n_\Phi$ an \emph{inward pointing normal} vector of the facet $\Phi$. It follows that
\begin{equation}\label{ss-intersezione}
  \D=\bigcap_{\Phi<^1\D} H^+_\Phi= \{\m\in M_\R\,|\,\forall\,\Phi<^1\D\quad\langle\m,\n_\Phi\rangle\geq-a_\Phi\}
\end{equation}
The relative interior of $\D$ will be denoted by $\Relint{\D}$, or simply $\Int\D$ when $\D$ is full dimensional.

In the following \emph{we will consider full dimensional polytopes only, unless otherwise advised}.

The \emph{polar polytope} $\D^*$ of a polytope $\D\subseteq M_\R$ containing the origin $\0\in M$ as an interior point, that is $\0\in\Int{\D}$, is defined as follows
\begin{equation}\label{polare}
  \D^*:=\{\n\in N_\R\,|\,\forall\,\m\in\D\quad\langle\n,\m\rangle\geq -1\}\subseteq N_\R
\end{equation}
It is a full dimensional polytope in $N_\R$ with $\0\in\Int{\D^*}$ and $(\D^*)^*=\D$.
In particular, if $\D$ admits the presentation given in (\ref{ss-intersezione}) then
\begin{equation}\label{polare2}
  \D^*=\conv(\{a_\Phi^{-1}\n_\Phi\,|\,\forall\,\Phi<^1\D\})\subseteq N_\R
\end{equation}
(see \cite[Exer.~2.2.1]{CLS}). Clearly, in general, $\D^*$ is not a lattice polytope in $N$, even if $\D$ is a lattice polytope in $M$.
A lattice polytope $\D$ is called \emph{reflexive} if $\0\in\Int\D$ and $\D^*$ is still a lattice polytope. By \cite[Thm.~4.1.6]{Batyrev94}
\begin{equation*}
  \D\ \text{is reflexive}\ \Longleftrightarrow\ \Int\D\cap M=\{\0\}
\end{equation*}
Given a divisor $D=\sum_{\rho\in\Si(1)}a_\rho D_\rho\in\Weil(X(\Si))$, the following polyhedron
\begin{equation}\label{div-politopo}
  \D_D:=\{\m\in M_\R\,|\,\forall\,\rho\in\Si(1)\quad\langle\m,\v_\rho\rangle\geq -a_\rho\}=\{\m\in M_\R\,|\,V^T\cdot\m \geq -\aa\}
\end{equation}
is called the \emph{polyhedron associated to $D$}, where $V=\left(\v_\rho\right)_{\rho\in\Si(1)}$ is a fan matrix of $X$ and $\aa=\left(a_\rho\right)_{\rho\in\Si(1)}$ is the column vector of coefficients of $D$.
In general it is not a polytope, but just a polyhedron as intersection of a finitely many closed half spaces.

\begin{proposition}[Prop.~4.3.8~(b) and \S6.1 in \cite{CLS}]\label{prop:gg}
  If $X(\Si)$ is complete then, for any $D$ in $\Weil(X)$, the associated polyhedron $\D_D$ is a polytope. Moreover:
  \begin{enumerate}
    \item $D$ is \emph{basepoint free}, that is $\cO_X(D)$ is generated by global section, if and only if $\D_D=\conv(\{\m_\s\in M\,|\,\s\in\Si(n)\})$,
    \item $D$ is ample if and only if $\D_D=\conv(\{\m_\s\in M\,|\,\s\in\Si(n)\})$ and $\s\neq\s'$ implies $\m_\s\neq\m_{\s'}$.
  \end{enumerate}
\end{proposition}

Recall that a Weil divisor $D$  is \emph{semi-ample} if a positive multiple $kD$, $k\in\N$, is basepoint free (hence Cartier). In particular, if $X(\Si)$ is complete and $D$ semi-ample, then
$$\D_{kD}=k\D_D=\conv(\{\m_\s\in M\,|\,\s\in\Si(n)\})$$
is a lattice polytope.

\begin{proposition}[Thm.~6.3.10 in \cite{CLS}]\label{prop:semiampio}
  Let $|\Si|$ be convex of full dimension. Then $D$ is semi-ample if and only if $kD$ is numerically effective (nef), for some $k\in\N$, that is $kD$ is Cartier and $kD\cdot C\geq 0$, for any complete curve $C\subset X$.
\end{proposition}

Starting from a lattice polytope $\D$ one can construct a \emph{projective} toric variety $\P_\D$ as follows. For any nonempty face $\phi< \D$ consider the dual cone $\s_\phi^\vee\subseteq N_\R$ of the cone
\begin{equation*}
    \s_\phi:=\{r(\m-\m')\ |\ \m\in\D\ ,\ \m'\in \phi\ ,\ r\in\R_{\geq 0}\}\subseteq M_{\R}
\end{equation*}
Then $\Si_{\D}^\perp:=\{\s_\phi^\vee\ |\ \phi< \D\}$ turns out to be a fan, called the \emph{normal fan} of the polytope $\D$, and $\P_\D$ is the associated toric variety. It is projective as there exists an ample divisor $H$ of $\P_{\D}$ whose associated polytope is precisely $\D$.

A further toric variety $\XX_\D$ can be associated with a lattice polytope $\D$ such that $\0\in\D$. Namely, for every facet $\Phi<^1\D$ such that $\0\not\in\Relint\Phi$, consider the cone \emph{projecting $\Phi$ from the origin}, that is
\begin{equation}\label{cono su facciata}
  \s_\Phi:=\{r\m\,|\,\m\in\Phi\ ,\ r\in\R_{\geq 0}\}\subseteq M_\R
\end{equation}
Then $\Si_\D:=\{\tau\,|\,\exists\,\Phi<^1\D:\ \tau<\s_\Phi\}$ turns out to be a fan, called the \emph{fan over the polytope $\D$}, and $\XX_\D$ is the associated toric variety. If $\0\in\Int\D$, $\XX_\D$ is complete as the support $|\Si_\D|$ is the whole $M_\R$ (clearly for $\XX_\D$, the role of the dual lattices $M,N$ is reversed with respect to $\P_\D$). This is a direct consequence of the following

\begin{proposition}\label{prop:Fmatricedipolitopo}
  Given an identification $M\cong\Z^n$ and a lattice polytope
  \begin{equation*}
    \D:=\conv(\{\m_i\in M\,|\,i=1,\ldots,m\})
  \end{equation*}
  let $V_\D=\left(
           \begin{array}{ccc}
             \v_1 & \cdots & \v_m \\
           \end{array}
         \right)$ be the $n\times m$ integer matrix defined by the generators $\v_i$ of the semi-groups $\langle\m_i\rangle\cap M$, for any $1\leq i\leq m$. Then $V_\D$ is a reduced $F$-matrix if and only if $\0\in\Int\D$. In particular, if $\0\in\Int\D$ then $V_\D$ is a fan matrix of $\XX_\D$ and the latter turns out to be a complete toric variety.
\end{proposition}

\begin{proof}
  Assume $V_\D$ is an $F$-matrix. Then $V_\D$ is clearly reduced as all the $\v_i$'s are primitive. Moreover, choosing the first column $\v_1$ of $V_\D$, the opposite vector $-\v_1$ belongs to the cone $\langle\v_2,\ldots,\v_n\rangle$, by \cite[Prop.~3.5]{RT-LA&GD}. Then
  \begin{equation*}
    \0=\v_1-\v_1=\v_1 +\sum_{j=2}^m \lambda_j\v_j=\sum_{i=1}^m \mu_i\m_i\
  \end{equation*}
  where
  \begin{equation*}
   \mu_1={\|\v_1\|\over\|\m_1\|}>0\ ,\quad \forall\,j\geq 2\quad \mu_j=\lambda_j{\|\v_j\|\over\|\m_j\|}>0
  \end{equation*}
  One can then conclude that $\0\in\Int\conv(\{\m_k\}_{k=1}^m)=\Int\D$ by setting  $\mu:=\sum_k\mu_k$ and writing $\0=\sum_k (\mu_k/\mu)\,\m_k$\,.

  Viceversa, assume $\0\in\Int\D=\Int\conv(\{\m_k\}_{k=1}^m)$. Conditions (a), (c) and (d) in Definition~\ref{def:Fmatrice} are clearly satisfied. To show that $V_\D$ is $F$-complete, for any vector $\v\in M_\R$ consider the polytope
  \begin{equation*}
    \D':=\conv(\m_1\,\ldots,\m_m,-\v)\subseteq M_\R
  \end{equation*}
  Since $\D\subseteq\D'$, one has that $\0\in\Int\D'$, meaning that
  \begin{equation*}
    \exists\,\mu_1>0,\ldots,\mu_m>0,\mu>0\,:\quad \sum_k\mu_k +\mu=1\ ,\quad\0=\sum_k\mu_k\m_k-\mu\v
  \end{equation*}
  meaning that
  \begin{equation*}
    \v=\sum_k\l_k\v_k\ ,\quad\text{with}\quad \forall\,k\ \l_k={\mu_k\,\|\m_k\|\over\mu\,\|\v_k\|}>0\ \Longrightarrow\ \v\in\langle V_\D\rangle
  \end{equation*}
\end{proof}

\begin{remark}\label{rem:reflexive}
  If $\D\subseteq M_\R$ is a reflexive polytope, then
  $$\P_\D\cong\XX_{\D^*}\quad\text{and}\quad\XX_\D\cong\P_{\D^*}$$
  These isomorphisms are induced by identity morphisms of lattices $N$ and $M$, respectively.
\end{remark}

\begin{corollary}\label{cor:smallres}
  Let $\D$ be a lattice polytope such that $\0\in\Int(\D)$ and $V_\D$ be the fan matrix of $\XX_\D$, as constructed in the previous Proposition~\ref{prop:Fmatricedipolitopo}. Then, for any $\Si\in\SF(V_\D)$ which is a refinement of $\Si_\D$, the associated toric variety $X(\Si)$ is a $\Q$-factorial small resolution of $\XX_\D$.
\end{corollary}

\begin{proof}
  In fact, every refinement $\Si\in\SF(V)$  of $\Si_\D$ is obtained by a simplicial subdivision of cones in $\Si_\D$. In particular, the induced birational resolution $X(\Si)\longrightarrow\XX_\D$ is small, as $\Si(1)=\Si_\D(1)=\{\langle\v_1\rangle,\ldots,\langle\v_m\rangle\}$\,.
\end{proof}

\subsection{Cones of divisors}\label{ssez:coni&div} Let $X(\Si)$ be a complete toric variety. Then there is a short exact sequence
\begin{equation}\label{complete deg sequence}
  \xymatrix@1{0\ar[r]& M \ar[r]^-{div}_-{V^T} & *!U(.45){\bigoplus\limits_{\rho \in \Sigma (1)} \Z \cdot D_{\rho}}
\ar[r]^-d_-Q & \Cl (X) \ar[r] & 0 }
\end{equation}
(see e.g. \cite[\S3.4]{Fulton}, \cite[Prop.~4.2.5]{CLS}). The representative matrices,
$V$ and $Q$, of the $\Z$-linear morphisms $div$ and $d$, respectively, gives a fan matrix and \emph{weight matrix}, respectively, of $X$. Since $X$ is complete, $V$ is a reduced $F$-matrix and $Q$ is a \emph{Gale dual matrix} of $V$, which \emph{can be assumed to be positive} \cite[Thm.~3.18]{RT-LA&GD}. This means that:
\begin{itemize}
  \item[(*)] \emph{the image $\im(d)=d(\Weil(X))$ of the degree morphism in (\ref{complete deg sequence}), can be assumed contained in  the positive orthant $\R^r_+$ of $\R^r\cong\Cl(X)\otimes\R$, being $r$ the Picard number of $X$}.
\end{itemize}
Recall that every divisor of $X$ is linearly equivalent to a torus invariant divisor. This means that, in the isomorphism $\Cl(X)\otimes\R\cong\R^r$ the cone $\overline{\Eff}(X)\subseteq\Cl(X)\otimes\R$, which is the closure of the cone generated by classes of effective divisors, is identified with cone $\langle Q\rangle$ generated by the columns of $Q$, that is
\begin{equation}\label{iso-coni}
  \overline{\Eff}(X)\cong\langle Q\rangle
\end{equation}
Let us now introduce the following:

\subsubsection{Notation}\label{sssez:notazione} Let $A$ be a $d\times m$ matrix. For any subset $I\subseteq\{1,\ldots,m\}$ we will denote by $A_I$ the sub-matrix of $A$ obtained by considering the columns indexed by $I$, only, and by $A^I$ the complementary sub-matrix of $A_I$ in $A$, that is, the one obtained by considering only the columns not indexed by $I$.

Coming back to the situation of a complete toric variety $X(\Si)$, of dimension $n$ and Picard number $r$, to every cone $\s\in\Si$ one can associate a subset $I\subseteq\{1,\ldots,m=n+r\}$ such that
\begin{equation*}
  \s=\langle V_I\rangle\subseteq N_\R
\end{equation*}
Define $\I_\Si:=\{I\in\mathfrak{P}(\{1,\ldots,m\})\,|\,\langle V_I\rangle\in\Si\}$\,, that is, $\Si=\{\langle V_I\rangle\,|\,I\in\I_\Si\}$\,. Then set
\begin{equation*}
  \Mov(Q):=\bigcap_{i=1}^m \langle Q^{\{i\}}\rangle\quad,\quad \Nef(\I):=\bigcap_{I\in\I} \langle Q^I\rangle\quad \text{(for any}\ \I\subseteq\mathfrak{P}(\{1,\ldots,m\})\,)
  \end{equation*}
Recall that the cone $\overline{\Mov}(X)$, which is the closure of the one generated by classes of \emph{movable} divisors, and the cone $\Nef(X)$, generated by classes of nef divisors, are both sub-cones of the effective cone $\overline{\Eff}(X)$. Then, the isomorphism (\ref{iso-coni}) descends to give isomorphisms
\begin{equation*}
  \overline{\Mov}(X)\cong\Mov(Q)\quad,\quad \Nef(X)\cong\Nef(\I_\Si)
\end{equation*}
More precisely, recalling notation~\ref{sssez:SF}, we get the following

\begin{proposition}\emph{\cite[Thm. 15.1.10(c)]{CLS}}\label{prop:nef} If $V=\left(
                                                                              \begin{array}{ccc}
                                                                                \v_1 & \ldots & \v_{n+r} \\
                                                                              \end{array}
                                                                            \right)$
is an $F$--matrix then, for every fan $\Si\in\P\SF(V)$
there is a na\-tu\-ral isomorphism
$$\Pic(X(\Si))\otimes\R\cong\Cl(X)\otimes\R\cong \R^r$$
taking the cones
\begin{equation*}
    \Nef(X(\Si))\subseteq\overline{\Mov}(X(\Si))\subseteq\overline{\Eff}(X(\Si))\subseteq\R^r
\end{equation*}
to the sub-cones of the positive orthant
\begin{equation*}
    \Nef(\I_\Si)\subseteq\Mov(Q)\subseteq\langle Q\rangle\subseteq\R^r_+
\end{equation*}
In particular, if $d:\mathcal{W}_T(X(\Si))\to\Cl(X(\Si))$ is the degree morphism, then a Weil divisor $D$ on $X(\Si)$ admits a nef (ample) positive multiple if and only if its class $[D]=d(D)\in\Nef(\I_\Si)$ (\,$d(D)\in\Relint\Nef(\I_\Si)$, resp.).
 \end{proposition}

 The following is a useful application to $\Q$-factorial small resolutions of a complete toric variety $X(\Si)$.

 \begin{proposition}\label{prop:risoluzioni}

   Let $X(\Si)$ be a complete toric variety and $V$ be a fan matrix of $X$. Then, for any $\Xi\in\SF(V)$ giving a refinement of $\Si$, the identity morphism of the common lattice $N$, namely $\id_N:N\longrightarrow N$, induces a fans' morphism from $\Xi$ to $\Si$. The induced morphism of toric varieties
   $$\vf:Y(\Xi)\longrightarrow X(\Si)$$
   is a $\Q$-factorial small resolution. Then, there is the following isomorphism of divisorial exact sequences
   \begin{equation*}
     \xymatrix{0\ar[r]& M \ar[d]^-{\id_M}_-{I_n}\ar[r]^-{div}_-{V^T} & \Weil(X)\ar[d]^-{\vf^*}_-{I_m}\ar[r]^-d_-Q & \Cl (X)\ar[d]^-{\overline{\vf}^*}_-{I_r} \ar[r] & 0\\
     0\ar[r]& M \ar[r]^-{div}_-{V^T} & \Weil(Y)\ar[r]^-d_-Q & \Cl (Y) \ar[r] & 0}
   \end{equation*}
   where $\vf^*$ is the pull-back of Weil divisors and $\overline{\vf}^*$ is the induced morphism on divisor classes. By composing $\overline{\vf}^*$ with isomorphisms $\Pic(X)_\R\cong\R^r$ and $\Pic(Y)_\R\cong\R^r$, defined in the previous Proposition~\ref{prop:nef}, there follows the following identification of nested divisorial sub-cones of the positive orthant $\R^r_+$
   \begin{equation*}
     \xymatrix{\Nef(\I_{\Si})\cong\Nef(X)\ar@{^(->}[dd]\ar@<-1ex>@{^(->}[rd]&\\
     &\left(\Mov(Q)\cong\overline{\Mov}(X)\stackrel{\overline{\vf}^*_\R}{\cong}\overline{\Mov}(Y)\right)
     \,\subseteq\, \left(\langle Q\rangle\cong\overline{\Eff}(X)\stackrel{\overline{\vf}^*_\R}{\cong}\overline{\Eff}(Y)\right)\\
     \Nef(\I_{\Xi})\cong\Nef(Y)\ar@{^(->}[ur]&
}
   \end{equation*}
   Moreover, for every Weil divisor $D=\sum_{\rho\in\Si(1)}a_\rho D_\rho\in\Weil(X)$, there follows the identification of associated polytopes
   $$\{\m\in M_\R\,|\,V^T\cdot\m \geq -\aa\}=\D_D=\id_M(\D_D)=\D_{\vf^*(D)}$$
   where $\aa=(a_\rho)_{\rho\in\Si(1)}$\,.
   In particular $D$ is semi-ample if and only if its class $[D]$ belongs to $\Nef(X)$. Then there exists a positive integer $k\in\N$ such that
   \begin{equation*}
     \conv(\{\m_\s\in M\,|\,\s\in\Si(n)\})=k\D_D\stackrel{\id_M}{=}k\D_{\vf^*(D)}
     =\conv(\{\m_I\in M\,|\,I\in\I_{\Xi(n)}\})
   \end{equation*}
where
\begin{equation}\label{mI}
\m_I:=-k(V_I^T)^{-1}\cdot \aa_I
\end{equation}
and $\aa_I$ is the sub-vector of $\aa$ whose entries are indexed by $I$.
   In particular, this means that, for every $\s\in\Si(n)$, $\m_\s=\m_I$ for any $I\in\I_{\Xi(n)}$ such that $\langle V_I\rangle\subseteq \s$\,.
 \end{proposition}

 \begin{proof}
 Results on divisorial cones are direct consequences of the previous Proposition \ref{prop:nef}. In particular\footnote{Moreover, after \cite[Prop.~1.11]{Hu-Keel}, $\Nef(X)\cong\vf^*(\Nef(X))$ lives on the boundary of $\Nef(Y)$: actually Hu-Keel proved this fact when $X$ is projective, but this hypothesis is unnecessary (see e.g. \cite[Thm.~3.7]{R-wMDS} and references therein).}
\begin{equation*}
   \Nef(\I_\Si)=\bigcap_{I\in\I_\Si} \langle Q^I\rangle\subseteq\bigcap_{J\in\I_\Xi} \langle Q^J\rangle = \Nef(\I_\Xi)
 \end{equation*}
 because $\Xi$ is a refinement of $\Si$, so implying that
 \begin{equation*}
   \forall\,I\in\I_\Si\ \exists\,J\in\I_\Xi\,: \langle V_J\rangle \subseteq \langle V_I\rangle\ \Longleftrightarrow\ \langle Q^I\rangle \subseteq \langle Q^J\rangle\
 \end{equation*}
 Results on divisorial polytopes come, on the one hand, directly from the definition given in (\ref{div-politopo}) and, on the other hand, from Proposition~\ref{prop:semiampio}, giving that $[D]\in\Nef(X)$, and Proposition~\ref{prop:gg}. In fact, a positive multiple $kD$ is base-point free, that is, $k\D_D=\conv(\{\m_\s\in M\,|\,\s\in\Si(n)\})$. Since $\Nef(X)\subseteq\Nef(Y)$, $[\vf^*(D)]=\overline{\vf}^*([D])\in\Nef(Y)$ so giving that $\vf^*(D)$ is semi-ample, too, still by Proposition~\ref{prop:nef}. Therefore
 \begin{equation*}
   k\D_D=k\D_{\vf^*(D)}=\conv(\{\m_\tau\in M\,|\,\tau\in\Xi(n)\})=\conv(\{\m_I\in M\,|\,I\in\I_{\Xi(n)}\})
 \end{equation*}
where $\m_\tau=\m_I$ whenever $\tau=\langle V_I\rangle$\, and $\m_I$ is defined as in (\ref{mI}).
 \end{proof}

\begin{remark} Part of the statement in Proposition~\ref{prop:risoluzioni} admits a streamlined interpretation by using the secondary fan (see \cite[\S~14.4]{CLS}). In fact, one can say that if $X$ is not $\Q$-factorial then $\Nef(X)$ is a boundary cell of a full dimensional chamber $\Nef(Y)$ of the secondary fan supported by the pseudo-effective cone $\overline{\Eff}(Y)\cong\overline{\Eff}(X)$.
\end{remark}

 \subsection{Non-degenerate hypersurfaces and their stratification}\label{ssez:regular} Given a toric variety $X=X(\Si)$, let $\T\subseteq X$ be the maximal acting torus on $X$.
 Consider a Laurent polynomial
 \begin{equation*}
 f=\sum_{\substack{m\in M\\\text{finite}}} c_m\chi^m\ ,\quad c_m\in\C^*
 \end{equation*}
Denote by $Z_f\subseteq\T$ the zero-locus of $f$ in $\T$ and let $Y_f$ be its closure in $X$.

 \begin{definition}[see Def.~3.1.4 in \cite{Batyrev94} and Def.~4.13 in \cite{BC}]\label{def:Si-regolare}
   A Laurent polynomial $f$, and the associated hypersurfaces $Z_f\subseteq\T$ and $Y_f\subseteq X(\Si)$, are called \emph{non-degenerate} \wrt $\Si$ (or, equivalently, \emph{$\Si$-regular}) if, for every $\s\in\Si$, the associated $\s$-stratum $Y_{f,\s}:=Y_f\cap \T\cdot x_\s$ is empty or a smooth subvariety of codimension 1 in the torus orbit $\T\cdot x_\s$\,.
In other words, non-degenerate means that $Y_f$ admits only \emph{transversal intersections} with all the torus orbits $\T\cdot x_\s$, $\s\in\Si$\,.
 \end{definition}

 \section{A duality between framed toric varieties}\label{sez:dualita-ftv}

 Let us start the present section by introducing the main character of this paper.

 \begin{definition}[Framed toric variety (ftv)]\label{def:ftv}
   A \emph{framed toric variety} is a couple $(X,D)$ where:
   \begin{itemize}
     \item $X$ is a complete toric variety, with $\dim(X)=n$ and $\rk(\Pic(X))=r$,
     \item $D=\sum_{\rho\in\Si(1)}a_\rho D_\rho=\sum_{i=1}^m a_i D_i\in\Weil(X)$, with $m=n+r$, is a \emph{strictly effective} (that is $a_i>0$, for every $i$), torus invariant Weil divisor, called a \emph{framing} of $X$.
   \end{itemize}
   A \emph{morphism of framed toric varieties} $f:(X,D)\longrightarrow(X',D')$ is a morphism of underlying toric varieties $f:X\longrightarrow X'$ inducing a well defined pull-back morphism on torus invariant Weil divisors $f^*:\Weil(X')\longrightarrow\Weil(X)$ such that $f^*D'=D$. If $f$ is an isomorphism of toric varieties, then it gives an \emph{isomorphism of framed toric varieties} $f:(X,D)\cong(X',D')$\,. This construction defines a category \textbf{ftv} of framed toric varieties.
 \end{definition}

 \subsection{Framed duality}
 Given a ftv
 $$\left(X(\Si),D_\aa=\sum_{\rho\in\Si(1)} a_\rho D_\rho\right)=:(X,\aa)$$
 consider the polytope associated with $D_\aa$
 \begin{equation}\label{Delta_a}
   \D_{\aa}:=\D_{D_\aa}=\{\m\in M_\R\,|\,V^T\cdot\m\geq -\aa\}
 \end{equation}
 being $V$ a fan matrix of $X$\,.
 Since ${D_\aa}$ is strictly effective, that is $-\aa<\0$\,, certainly $\0\in\Int\D_\aa$, meaning that $\0\in\Int (k\D_\aa)$, for any positive integer $k\in\N$.

 On the other hand, define the \emph{integer part} of a polytope $\D\subseteq M_\R$ as
 \begin{equation*}
   [\D]:=\conv(\{\m\in M\cap\D\})
 \end{equation*}
 Clearly, if $\D$ is a lattice polytope then $[\D]=\D$\,.

 \begin{definition}[$f$-polytope]\label{def:Deltapolytope}
   The \emph{framing polytope} (\emph{$f$-polytope}) of a ftv $(X,\aa)$ is the lattice polytope $\D(X,\aa)\subseteq M_\R$ so defined:
   \begin{equation}\label{k0}
     \D(X,\aa):=[k_0\D_\aa]\quad,\quad k_0:=\min\{k\in\N\,|\,\0\in\Int[k\D_\aa]\}
   \end{equation}
 \end{definition}

 \begin{remark}
   For what observed above, $k_0$ is well defined. Notice that $k_0$ may be bigger than 1\,: in fact it may happen that $\0$ is not an interior point of the integer part $[\D_\aa]$\,, as the following Example~\ref{ex:k0>1} shows. On the other hand $k_0=1$ when $\D_\aa$ is a lattice polytope: in this case $\D(X,\aa)=\D_\aa$\,. The geometric importance of having $\0$ as an interior point of the framing polytope $\Int\D(X,\aa)$ is explained in the next Remark~\ref{rem:negKod}.
 \end{remark}

 \begin{example}\label{ex:k0>1}
   Consider the ftv given by $(X,\aa)=(\P(1,2,5),(2,1,1))$. Then
   \begin{equation*}
     \D_\aa=\conv\left(
                   \begin{array}{ccc}
                     -2 & -2 & 7 \\
                     3/2 & {3/ 5} & -3 \\
                   \end{array}
                 \right)\ \Longrightarrow\ [\D_\aa]=\conv\left(
                                                           \begin{array}{cccc}
                                                             -1 & -2 & 2 & 7 \\
                                                             1 & 1 & -1 & -3 \\
                                                           \end{array}
                                                         \right)
   \end{equation*}
   Then $\0\not\in\Int[\D_\aa]$. But $\0\in\Int[2\D_\aa]$, so giving $k_0=2$ in the previous Definition~\ref{def:Deltapolytope}.
 \end{example}

\begin{remark} Notice that, given a complete toric variety $X$ there can be different choices of the framing $\aa$ giving the same framing polytope $\D(X,\aa)$, as shown by the following Example~\ref{ex:due framing}. Anyway, only one of these choices can give rise to an involutive duality, that is, a \emph{calibrated $f$-process}, as shown by Example~\ref{ex:due framing 3}.
\end{remark}

\begin{example}\label{ex:due framing}
  Consider the fan matrix $V=\left(
                   \begin{array}{ccc}
                     5 & -2 & -1\\
                     -2 & 5 & -1 \\
                   \end{array}
                 \right)$.
There is a unique simplicial and complete fan in $\SF(V)$ giving rise to the $\Q$-factorial complete toric variety $X$ described by the quotient of the weighted projective space $\P(1,1,3)$
by the following action of $\Z/7\Z$
\begin{equation}\label{azione2}
  \Z/7\Z\times\P(1,1,3)\ni(\overline{\ve},[x_1:x_2:x_3])\mapsto[\mu x_1:x_2:\mu^{-1}x_3]\in\P(1,1,3)
\end{equation}
where $\mu=\exp(2\ve\pi i/7)$ (see also the next Lemma~\ref{lem:Ga}). Then consider the two framing given by $\aa=(3,3,1)$ and $\aa'=(4,4,1)$. The associated polytopes are given by
\begin{equation*}
  \D_\aa=\conv\left(
                   \begin{array}{ccc}
                     8/7 & -1 & -1/7 \\
                     -1/7 & -1 & 8/7 \\
                   \end{array}
                 \right)\,,\quad
  \D_{\aa'}=\conv\left(
                   \begin{array}{ccc}
                     9/7 & -4/3 & -2/7 \\
                     -2/7 & -4/3 & 9/7 \\
                   \end{array}
                 \right)
\end{equation*}
Then
\begin{equation*}
  \D(X,\aa)=[\D_\aa]=\conv\left(
                   \begin{array}{ccc}
                     1 & 0 & -1 \\
                     0 & 1 & -1 \\
                   \end{array}
                 \right)=[\D_{\aa'}]=\D(X,\aa')
\end{equation*}
\end{example}

 \begin{remark}\label{rem:anticanonico}
   Assume $\aa=\1$, that is $D_\aa=-K_X$ is the anti-canonical divisor of the complete toric variety $X$. Then, if $\D_\1$ is a reflexive polytope, the above construction gives
   \begin{equation*}
     \D(X,\aa)=\D_\aa=\D_\1=\D_{-K_X}
   \end{equation*}
 \end{remark}

 \subsubsection{The $f$-dual ftv and its small $\Q$-factorial resolutions}\label{sssez:deltadual}
 Associated with the construction of the lattice polytope $\D(X,\aa)$ there is the complete toric variety
 \begin{equation*}
   \XX_\aa:=\XX_{\D(X,\aa)}
 \end{equation*}
 given by the fan $\Si_\aa:=\Si_{\D(X,\aa)}$ over the lattice polytope $\D(X,\aa)$\,.
Let $\Lambda_\aa$ be the fan matrix of $\XX_\aa$ constructed in Proposition~\ref{prop:Fmatricedipolitopo}: it is  a $n\times m'$ integer matrix. Given a fan matrix $V$ of $X$, which is a $n\times m$ integer matrix, define
\begin{equation*}
  M_\aa:=V^T\cdot\L_\aa\in\mathbf{M}(m\times m';\Z)
\end{equation*}
Let $\bb=(b_j)_{j=1}^{m'}$ be \emph{the minimum strictly positive column vector} such that
\begin{equation}\label{b}
  M_\aa^T+B\geq \0\quad\text{where}\quad B:=\underbrace{\left(\,\bb\ \cdots\ \bb\,\right)}_{m\ \text{times}}\,\in \mathbf{M}(m'\times m;\N)
\end{equation}

\begin{definition}\label{def:dual-ftv}
  Calling $D'_\bb:=\sum_{j=1}^{m'}b_jD'_j$, where $D'_1,\ldots,D'_{m'}$ are the torus invariant prime divisors generating $\Weil(\XX_\aa)$\,, then $(\XX_\aa,\bb):=(\XX_\aa,D'_\bb)$ is a ftv, called \emph{the framed dual ($f$-dual) of $(X,\aa)$}.
\end{definition}

\begin{remark}\label{rem:batyrev}
  Recalling previous Remarks~\ref{rem:reflexive} and \ref{rem:anticanonico}, if $\aa=\1$ and $\D_\1$ is a reflexive polytope, then \emph{$f$-duality is Batyrev's duality} between Fano toric varieties, as defined in \cite{Batyrev94}.
\end{remark}

\begin{example}[Example~\ref{ex:due framing} continued]\label{ex:due framing 2} Consider the two ftv $(X,\aa)$ and $(X,\aa')$ given in Example~\ref{ex:due framing}. They admit the same framing polytope $\D(X,\aa)=\D(X,\aa')$, that is, the same $f$-dual ftv $(\XX_\aa,\bb)=(\XX_{\aa'},\bb)$ with $\XX_\aa=\XX_{\aa'}=\P^2$ and $\bb=(2,2,3)$ determined by condition (\ref{b}), as
\begin{equation*}
  M_\aa^T= \L_\aa^T\cdot V= \left(
                              \begin{array}{cc}
                                1 & 0 \\
                                0 & 1 \\
                                -1 & -1 \\
                              \end{array}
                            \right)\cdot \left(
                   \begin{array}{ccc}
                     5 & -2 & -1\\
                     -2 & 5 & -1 \\
                   \end{array}
                 \right)= \left(
                            \begin{array}{ccc}
                              5 & -2 & -1 \\
                              -2 & 5 & -1 \\
                              -3 & -3 & 2 \\
                            \end{array}
                          \right)
\end{equation*}
Anyway, next Example~\ref{ex:due framing 3} will show that only one of the two choices $\aa$ and $\aa'$ gives actually rise to an involutive duality between framed toric varieties.

\end{example}

\begin{remark}\label{rem:negKod}
   Hypothesis of strict effectiveness, given in Definition~\ref{def:ftv} for a framing $\aa$, is needed to get $\0$ as an interior point of $\D(X,\aa)$ and, consequently, completeness of the $f$-dual toric variety $\XX_\aa$ (recall Proposition~\ref{prop:Fmatricedipolitopo}). Dropping that hypothesis leads to an asymmetric duality, as $\XX_\aa$ can no more be complete: this fact is quite reminiscent of the Givental's LG mirror model construction \cite{Givental96}. We will discuss this aspect in \S\ref{ssez:NegKod}, introducing the concept of a \emph{weak framing}. In a sense, dropping strict effectiveness is the key to understand when to look for a LG mirror model rather than for a complete mirror partner.
 \end{remark}

The following statement is a direct application of the Propostion~\ref{prop:risoluzioni}.

\begin{corollary}\label{cor:risoluzioni}
For every fan $\Xi\in\SF(\L_\aa)$ such that $\Xi$ refines $\Si_\aa$ there is a well defined birational morphism $\vf:Y(\Xi)\longrightarrow \XX_\aa$ which is a $\Q$-factorial small resolution.
In particular, for any such $\Xi$, $(Y(\Xi),\vf^*D'_\bb)$ is a $\Q$-factorial ftv.
\end{corollary}

\subsubsection{$f$-process as double $f$-duality}\label{ssez:Deltaprocesso} By definition, we call \emph{$f$-process the double application of $f$-duality}. This gives rise to a ftv $(\XX_\bb, \cc):=(\XX_\bb, D''_\cc)$ where:
\begin{itemize}
  \item calling
  \begin{equation}\label{rel1}
    \D_{\bb}=\{\n\in N_\R\,|\,\L_\aa^T\cdot\n \geq -\bb\}=\{\n\in N_\R\,|\,\forall\,1\leq j\leq m'\quad\langle\n,\ll_j\rangle \geq -b_j\}
  \end{equation}
  $\XX_\bb$ is the complete toric variety associated with the fan $\Si_\bb:=\Si_{\D(\XX_\aa,\bb)}$ over the lattice polytope
      \begin{equation}\label{k1}
        \D(\XX_\aa,\bb):=[k_1\D_\bb]\subseteq N_\R\quad,\quad k_1:=\min\{k\in\N\,|\,\0\in\Int[k\D_\bb]\}
      \end{equation}
  \item $D''_\cc=\sum_{l=1}^{m''} c_l D''_l$, where $D''_1,\ldots,D''_{m''}$ are the torus invariant prime divisors generating $\Weil(\XX_\bb)$ and $\cc=(c_l)_{l=1}^{m''}$ is the minimum strictly positive column vector such that
\begin{equation}\label{c}
  M_{\aa,\bb}^T+C\geq 0\,, \ \text{where}\ M_{\aa,\bb}:=\L^T_\aa\cdot\L_\bb\,,\ C:=\underbrace{\left(\,\cc\ \cdots\ \cc\,\right)}_{m'\ \text{times}}\,\in \mathbf{M}(m''\times m';\N)
\end{equation}
being $\L_\bb$ the fan matrix of $\XX_\bb$, defined by the primitive generators as\-so\-cia\-ted with the vertices of $\D_\bb$, as in Proposition~\ref{prop:Fmatricedipolitopo}.
\end{itemize}

\begin{remark}
   Let $\NN:=\conv(V)=\conv(\{\v_1,\ldots,\v_m\})$ be the lattice polytope associated with the fan matrix $V=(\v_1\ \cdots\ \v_m)$ of $X$. Then
  \begin{equation}\label{inclusione}
    \NN\subseteq [\D_\bb]\subseteq\D(\XX_\aa,\bb)
  \end{equation}
  as, for any column $\v_i$ of $V$, relation (\ref{b}) gives that
  \begin{equation*}
    \L_\aa^T\cdot\v_i\geq -\bb\ \stackrel{(\ref{rel1})}{\Longrightarrow}\ \v_i\in\D_\bb\cap N\subseteq [\D_\bb]
  \end{equation*}
\end{remark}

\begin{definition}[Calibrated $f$-process]\label{def:Deltaproc banale}
  A $f$-process
  \begin{equation}\label{Deltaproc}
    (X,\aa)\stackrel{f-dual}{\rightsquigarrow}(\XX_\aa,\bb)\stackrel{f-dual}{\rightsquigarrow}(\XX_\bb,\cc)
  \end{equation}
  is called \emph{calibrated} if there exist $\Xi\in\SF(V)$ and $\Xi'\in\SF(\L_\bb)$, refining $\Si$ and $\Si_\bb$, respectively, and an isomorphism of toric varieties $f:Y(\Xi)\stackrel{\cong}{\longrightarrow}Y'(\Xi')$ such that, calling $\vf:Y(\Xi)\longrightarrow X(\Si)$
  and $\vf':Y'(\Xi')\longrightarrow \XX_\bb(\Si_\bb)$
  the $\Q$-factorial resolutions associated with the choice of $\Xi$ and $\Xi'$, respectively, one has
  $$\vf^*D_\aa=(\vf'\circ f)^*D''_\cc$$
   In particular, there is an induced birational isomorphism in codimension 1, say $\check{f}:X\dashrightarrow \XX_\bb$\,, fitting in the following commutative diagram
  \begin{equation*}
    \xymatrix{Y\ar[d]^-\vf\ar[r]^-f_-\cong&Y'\ar[d]^-{\vf'}\\
                X\ar@{-->}[r]^-{\check{f}}&\XX_\bb}
  \end{equation*}
\end{definition}

\begin{remark}\label{rem:ftviso}
  Notice that, in the notation of the previous Definition~\ref{def:Deltaproc banale}, \emph{both $(Y,\vf^*D_\aa)$ and $(Y',(\vf')^*D''_\cc)$ are still framed toric varieties}. In fact, the birational transform $\vf^*D_\aa=\sum_i a_i\vf^*D_i$ is still a strictly effective divisor, as
  $$\Weil(Y)=\bigoplus_{i=1}^{n+r}\Z\cdot\vf^*(D_i)$$
  being $\vf$ a \emph{small} birational contraction.  Analogously for $(Y',(\vf')^*D''_\cc)$.

  \noindent Consequently, the condition of being calibrated can be restated by asking that \emph{$f:(Y,\vf^*D_\aa)\stackrel{\cong}{\longrightarrow}(Y',(\vf')^*D''_\cc)$ is a ftv isomorphism}.
\end{remark}

\begin{example}\label{ex:D}
  To fix ideas, consider the following example, that will be a running example throughout the present paper. Actually, it is the easiest case of the big class of framing given by projective hypersurfaces of general type, extensively studied in the next \S\ref{sez:ipersuperfici}.

  Consider the ftv $(X,\aa)=(\P^2,(1,1,2))$. A fan matrix of $X$ is given by
\begin{equation*}
  V=\left(
      \begin{array}{ccc}
        1 & 0 & -1 \\
        0 & 1 & -1 \\
      \end{array}
    \right)
\end{equation*}
Consequently, the polytope $\D_\aa =\D_{D_\aa}$ is given by
\begin{equation*}
  \D_\aa = \conv(\L_\aa)\ ,\quad\text{with}\ \L_\aa=\left(
                                                      \begin{array}{ccc}
                                                        3 & -1 & -1 \\
                                                        -1 & 3 & -1 \\
                                                      \end{array}
                                                    \right)
\end{equation*}
(see Fig.~\ref{Fig1}).
$\XX_\aa$ is the unique complete and $\Q$-factorial toric variety whose fan matrix is given by $\L_\aa$. It is a quotient of the weighted projective space $\P(\aa)=\P(1,1,2)$
by the action of $\Z/4\Z$ given by sending
\begin{equation}\label{azione2}
  \Z/4\Z\times\P(1,1,2)\ni(\overline{\ve},[x_1:x_2:x_3])\mapsto[\mu x_1:x_2:\mu^{-1}x_3]\in\P(1,1,2)
\end{equation}
being $\mu=\exp(\ve\pi i/2)$ (see also the next Lemma~\ref{lem:Ga}).
Dually, observing that
  \begin{equation*}
    \L_\aa^T\cdot V = \left( \begin {array}{ccc} 3&-1&-2\\
                                                -1&3&-2\\
                                                -1&-1&2\end {array} \right)
  \end{equation*}
  the framing of $\XX_\aa$ is given by the minimum positive vector $\bb$ such that
  \begin{equation*}
    \L_\aa^T\cdot V + \left(
                      \begin{array}{ccc}
                        \bb & \bb & \bb \\
                      \end{array}
                    \right)\geq \0\ \Longrightarrow\ \bb=\left(
                                                          \begin{array}{c}
                                                            2 \\
                                                            2 \\
                                                            1 \\
                                                          \end{array}
                                                        \right)
  \end{equation*}
    \begin{figure}
\begin{center}
\includegraphics[width=11truecm]{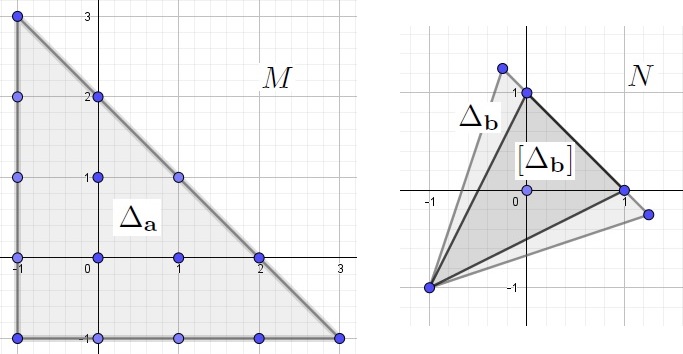}
\caption{\label{Fig1} Example \ref{ex:D}: polytopes $\D_\aa\subset M_\R$ and $[\D_\bb]\subseteq\D_\bb\subset N_\R$. }
\end{center}
\end{figure}
  Then
  \begin{equation*}
    \D_\bb=\conv\left( \begin {array}{ccc} 5/4&-1/4&-1\\
                                            -1/4&5/4&-1\end {array} \right)\ \Longrightarrow\ [\D_\bb]=\conv(V)
  \end{equation*}
  so giving that the $f$-process $(X,\aa)\leftrightsquigarrow (\XX_\aa,\bb)$ is calibrated.
\end{example}

The following result gives an algebraic criterion to having a calibrated $f$-process. Theorem~\ref{thm:Krawitz} will give a more geometric version of this criterion, in terms of Krawitz duality of the associated Landau-Ginzburg models.

\begin{theorem}\label{thm:Deltatriviale} Let $V=(\v_1\ \cdots\ \v_m)$\,, $\L_\aa=(\ll_1\ \cdots\ \ll_{m'})$ and $\L_\bb$ be the fan matrices of $X$\,, $\XX_\aa$ and $\XX_\bb$\,, respectively, constructed above. Then, up to identifying lattices $M$ (hence $N$) of $X$ and $\XX_\bb$, the $f$-process \emph{(\ref{Deltaproc})} is calibrated if and only if
  \begin{eqnarray}\label{min}
    V&=&\L_\bb\quad\text{(up to a permutation of columns)}\\
    \nonumber
    \min_{1\leq j\leq m'}\langle\v_i,\ll_j\rangle&=&-a_i\quad\text{(for all $i$ with $1\leq i\leq m$\,)}
  \end{eqnarray}
   In particular, recalling (\ref{k0}) and (\ref{k1}), $k_0=1=k_1$, that is,
   \begin{equation*}
     \D(X,\aa)=[\D_\aa]\quad\text{and}\quad \D(\XX_\aa,\bb)=[\D_\bb]
   \end{equation*}
\end{theorem}

\begin{proof}
  If (\ref{Deltaproc}) is a calibrated $f$-process then there exist $\Xi\in\SF(V)$ and $\Xi'\in\SF(\L_\bb)$, refining $\Si$ and $\Si_\bb$, respectively, and a ftv isomorphism $$f:(Y(\Xi),\vf^*D_\aa)\stackrel{\cong}{\longrightarrow}(Y'(\Xi'),(\vf')^*D''_\cc)$$
  as described in Definition~\ref{def:Deltaproc banale}. In particular, this means that $Y$ and $Y'$ admit equivalent fan matrices, as defined in relation (\ref{M-equivalenza}), that is
  \begin{equation}\label{A,B}
    \exists\,A\in\GL(n,\Z)\,,\ \exists\,B\in\mathfrak{S}_m\leq \GL(m,\Z)\quad \L_\bb=A\cdot V\cdot B
  \end{equation}
  Therefore $m''=m$ and inclusion (\ref{inclusione}) implies that $\0\in\Int [\D_\bb]$. Hence, $k_1=1$ in (\ref{k1}) and $\D(\XX_\aa,\bb)=[\D_\bb]$.
   Calling $M$ and $M'$ the characters' lattices of acting tori on $Y$ and $Y'$, respectively, and recalling Proposition~\ref{prop:risoluzioni}, condition (\ref{A,B}) comes from the following commutative diagram between associated divisorial short exact sequences
   \begin{equation}\label{div-diagram-MM'}
     \xymatrix{0\ar[r]& M \ar[r]^-{div}_-{V^T} & \Weil(Y)\ar[d]^-{(f^*)^{-1}}_-{B^T}\ar[r]& \Cl (Y)\ar[d]^-{(\overline{f}^*)^{-1}} \ar[r] & 0\\
     0\ar[r]& M'\ar[u]^-{A^T} \ar[r]^-{div}_-{\L_\bb^T} & \Weil(Y')\ar[r] & \Cl (Y') \ar[r] & 0}
   \end{equation}
This actually means that, up to a change of bases in $M$ and $M'$, matrices $A,B$ in (\ref{A,B}) and (\ref{div-diagram-MM'}) can be chosen as $A=I_n$ and $B=I_m$, so that $\L_\bb=V$.
Therefore, via $f,\vf,\vf'$ lattices $N$ and $M$ of $X,Y,Y'$ and $\XX_\bb$ can be identified as above, so giving an identification
   \begin{equation}\label{Weil=}
     \Weil(X)\cong\Weil(Y)\cong\Weil(Y')\cong\Weil(\XX_\bb)
   \end{equation}
   under which, generators $D_i$ are identified with generators $D''_i$, and the ftv isomorphism $f$ gives $\aa=\cc$. Definition (\ref{c}) of $\cc$ with $V=\L_\bb$ imply that
   \begin{equation}\label{a=c}
   \forall\,1\leq i\leq m\quad a_i=c_i=\max\left(\{1\}\cup\{-\langle\v_i,\ll_j\rangle\,|\,1\leq j\leq m'\}\right)
   \end{equation}
   Since $\L_\aa$ is a reduced $F$-matrix, Proposition~\ref{prop:Fmatricedipolitopo} gives that
   \begin{eqnarray*}
   \0\in\conv(\L_\aa)&\Longrightarrow&\0=\sum_{j=1}^{m'} x_j\ll_j\ \text{with $x_j\geq 0$ and $\sum_jx_j=1$} \\
   &\Longrightarrow&\forall\,i\quad 0=\sum_jx_j\langle\v_i,\ll_j\rangle\\
   &\Longrightarrow&\forall\,i\ \exists\,j\,:\quad \langle\v_i,\ll_j\rangle<0 \\
   &\Longrightarrow&\forall\,i\quad \max\left(\{-\langle\v_i,\ll_j\rangle\,|\,1\leq j\leq m'\}\right)\geq 1\\
   &\Longrightarrow&\forall\,i\quad -a_i=\min\left(\{\langle\v_i,\ll_j\rangle\,|\,1\leq j\leq m'\}\right)
   \end{eqnarray*}

   For the converse, assume $\L_\bb=V$, up to a permutation of columns. Then $\SF(V)=\SF(\L_\bb)$ and, for any choice $\Xi\in\SF(V)$ there exists $\Xi'\in\SF(\L_\bb)$ and an isomorphism of toric varieties $f:Y(\Xi)\cong Y'(\Xi')$\,. We can then identify lattices $N$ and $M$ of $Y$ and $Y'$\,. Moreover, via the $\Q$-factorial small resolutions $\vf:Y\longrightarrow X$\,, $\vf':Y'\longrightarrow \XX_\bb$, we can also identify lattices $N$ and $M$ of $X,Y,Y'$ and $\XX_\bb$, as above, so getting identifications (\ref{Weil=}) for torus invariant Weil divisors. Then, the $f$-process (\ref{Deltaproc}) is calibrated if  $\aa=\cc$\,. The latter is guaranteed by the second condition in (\ref{min}), as
\begin{equation}\label{c=a}
  \forall\,1\leq i\leq m\quad -c_i:=\min_{j}\langle\v_i,\ll_j\rangle=-a_i\ \Longrightarrow\ \cc=\aa
\end{equation}
Reasoning as above, one then has $k_1=1$.

\noindent Moreover, still recalling Proposition~\ref{prop:Fmatricedipolitopo}, relations (\ref{a=c}),(\ref{c=a}) give
$$\0\in\Int(\conv(\L_\aa))\subseteq\{\m\in M_\R\,|\,V^T\cdot \m\geq -\aa\}=\D_\aa\ \Longrightarrow\ k_0=1$$
\end{proof}

\begin{corollary}
  Assume $\D_\aa$ be a lattice polytope with primitive vertices, that is
  $$\D_\aa=\conv(\L_\aa)$$
  Then the $f$-process (\ref{Deltaproc}) is calibrated if and only if $V=\L_\bb$, up to an identification of lattices $M$ (hence $N$) of $X$ and $\XX_\bb$ and a permutation of columns.
\end{corollary}

\begin{proof}
  In fact, the second condition in (\ref{min}), in Theorem~\ref{thm:Deltatriviale}, is immediately attained, as columns $\L_\aa$ are given by vertices of $\D_\aa$\,.
\end{proof}

\begin{example}[Examples~\ref{ex:due framing} and \ref{ex:due framing 2} continued]\label{ex:due framing 3}
  Consider the two ftv $(X,\aa)$ and $(X,\aa')$ admitting the same $f$-dual ftv $(\P^2,(2,2,3))$. Notice that the polytope associated with $\bb=(2,2,3)$ on $\P^2$ is given by
\begin{equation*}
  \D_\bb=\conv\left(
                   \begin{array}{ccc}
                     5 & -2 & -2\\
                     -2 & 5 & -2 \\
                   \end{array}
                 \right)\ \Longrightarrow \L_\bb=\left(
                   \begin{array}{ccc}
                     5 & -2 & -1\\
                     -2 & 5 & -1 \\
                   \end{array}
                 \right)=V
\end{equation*}
Then
\begin{equation*}
  M_{\aa',\bb}^T=M_{\aa,\bb}^T= V^T\cdot\L_\aa= M_\aa^T\ \Longrightarrow\ \cc=(3,3,1)\left\{\begin{array}{c}
                                                                                           = \aa \\
                                                                                           \neq \aa'
                                                                                         \end{array}
\right.
\end{equation*}
This means that the $f$-process $\left(X,(3,3,1)\right)\stackrel{f-dual}{\leftrightsquigarrow}\left(\P^3,(2,2,3)\right)$ is calibrated, on the contrary of the $f$-process $\left(X,(4,4,1)\right)\stackrel{f-dual}{\rightsquigarrow}\left(\P^3,(2,2,3)\right)\stackrel{f-dual}
{\rightsquigarrow}\left(X,(3,3,1)\right)$\,.
\end{example}

  By the previous results, a calibrated $f$-process is the key ingredient to introduce an involutive duality between framed toric varieties, largely extending the classical Batyrev duality between Fano toric varieties: in fact the latter can be thought of the particular case of an $f$-duality associated with an ample anti-canonical framing (see the following \ref{ssez:Bat-dualita}).

\begin{remark}\label{rem:anticanonico}
  Although the expectation would be that, in general, a complete toric variety should admit a sufficiently positive line bundle with a strictly effective section giving rise to a calibrated $f$-duality, it is not true that every sufficiently positive line bundle on a complete toric variety admits a section  whose associated $f$-process is calibrated, neither for the anti-canonical one, as the following example shows.
\end{remark}

\begin{example}\label{ex:anticanonico}
  Let $X$ be the $\Q$-factorial complete toric surface of Picard number 2 given by the unique fan in $\SF(V)$ with
\begin{equation*}
  V:=\left(
       \begin{array}{cccc}
         1 & 2 & 0 & -1 \\
         0 & 3 & 1 & -2 \\
       \end{array}
     \right)
\end{equation*}
$X$ is a $\Q$-Fano projective surface. The reader can check that the anti-canonical line bundle $\omega_X^{-1}$ does not admit any section giving rise to a calibrated $f$-process. This fact still holds for the ample multiple $\omega_X^{-3}$.
Nevertheless, the framing $\aa:=(3,3,3,2)$ gives rise to a calibrated $f$-process
\begin{equation*}
  (X,\aa)\leftrightsquigarrow(\XX_\aa,\bb)
\end{equation*}
where $\XX_\aa$ is the $\Q$-factorial complete toric variety given by the unique fan in $\SF(\L_\aa)$ and
\begin{equation*}
  \L_\aa:=\left(
            \begin{array}{ccccc}
               -1 & -3 & -3 & 8 & 1 \\
            1 & 2 & 1 & -3 & -1 \\
            \end{array}
          \right)\ ,\quad\bb:=\left(1,3,3,3,1\right)
\end{equation*}
\end{example}

\section{A duality between hypersurfaces in toric varieties}\label{sez:dualita-hyp}

Consider the general setting presented in \S~\ref{sez:dualita-ftv} and an hypersurface $Y$ in a complete toric variety $X$. Assume that:
\begin{enumerate}
  \item there exists a divisor $D_\aa\in\Weil(X)$ such that $Y$ is a generic element in the linear system $|D_\aa|:=d^{-1}\left([D_\aa]\right)$\,, where $d$ is the class morphism in (\ref{complete deg sequence}),
  \item $(X,D_\aa)$ is a ftv satisfying conditions (\ref{min}) in Theorem~\ref{thm:Deltatriviale}\,, that is the $f$-process $$(X,\aa)\stackrel{f-dual}{\rightsquigarrow}(\XX_\aa,\bb)\stackrel{f-dual}{\rightsquigarrow}(\XX_\bb,\cc)$$ is calibrated.
\end{enumerate}

\begin{definition}\label{def:mirror}
  A generic element $Y^\vee\in|D'_\bb|:=d^{-1}\left([D'_\bb]\right)$ is called \emph{a $f$-mirror partner of $Y\in|D_\aa|$}.
\end{definition}

\begin{remark}\label{rem:famiglie}
  One can explicitly describe the defining polynomials of both $Y$ and $Y^\vee$ in the Cox rings of $X$ and $\XX_\aa$, respectively. Namely:
  \begin{itemize}
    \item[(a)] the polytope $\D(X,\aa)$ is the Newton polytope of $Y\in|D_\aa|$; call $\overline{\L}_\aa$ a matrix whose columns are given by all the lattice points in $\D(X,\aa)$: it is well defined up to a permutation of columns; setting $l:=|\D(X,\aa)\cap M|$, then $\overline{\L}_\aa$ is a $n\times l$ integer matrix; define
        \begin{equation*}
          \overline{M}_\aa:= V^T\cdot \overline{\L}_\aa\quad\text{and}\quad \overline{A}:=\underbrace{\left(\,\aa\ \cdots\ \aa\,\right)}_{l\ \text{times}}\,\in \mathbf{M}(m\times l;\N)\,;
        \end{equation*}
        then the polynomial of $Y$ is given by
        \begin{equation}\label{f-WT}
          f=\sum_{j=1}^l c_j\x^{\m_j} \in \Cox(X)\cong\C[x_1,\ldots,x_m]
        \end{equation}
        where $\m_j=(m_{i,j})$ is the $j$-th column of $\overline{M}_\aa+\overline{A}$ and $\x^{\m_j}:=\prod_{i=1}^m x_i^{m_{i,j}}$;
    \item[(b)] the polytope $\D(\XX_\aa,\bb)$ is the Newton polytope of $Y^\vee\in|D'_\bb|$; call $\overline{\L}_\bb$ a matrix whose columns are given by all the lattice points in $\D(\XX_\aa,\bb)$; setting $l':=|\D(\XX_\aa,\bb)\cap N|$, then $\overline{\L}_\bb$ is a $n\times l'$ integer matrix; define
        \begin{equation*}
          \overline{M}_{\aa,\bb}:= \L_\aa^T\cdot \overline{\L}_\bb\quad\text{and}\quad \overline{B}:=\underbrace{\left(\,\bb\ \cdots\ \bb\,\right)}_{l'\ \text{times}}\,\in \mathbf{M}(m'\times l';\N)\,;
        \end{equation*}
        then the polynomial of $Y^\vee$ is given by
        \begin{equation}\label{fdual}
          f^\vee=\sum_{j=1}^l c_j\x^{\n_j} \in \Cox(\XX_\aa)\cong\C[x_1,\ldots,x_{m'}]
        \end{equation}
        where $\n_j=(n_{i,j})$ is the $j$-th column of $\overline{M}_{\aa,\bb}+\overline{B}$ and $\x^{\n_j}:=\prod_{i=1}^{m'} x_i^{n_{i,j}}$.
  \end{itemize}
  Notice that both $f$ and $f^\vee$ are homogeneous polynomials, with respect to degrees induced by class groups. In fact, columns of both $\overline{M}_\aa$ and $\overline{M}_{\aa,\bb}$ determine trivial divisors, up to linear equivalence. Then
  $$\deg(f)=[D_\aa]\in\Cl(X)\quad\text{and}\quad\deg(f^\vee)=[D'_\bb]\in\Cl(\XX_\aa)$$
\end{remark}

\begin{remark}\label{rem:semimple}
  Given a ftv $(X,D_\aa)$, the divisor $D_\aa$ is not necessarily semi-ample. Anyway, if $(X,D_\aa)$ admits a calibrated $f$-process then it can be shown that there exists a small partial resolution $\phi:\widehat{X}\longrightarrow X$ such that the pulled back framing $\phi^*(D_\aa)$ is semi-ample and, moreover, it admits a positive multiple $h\phi^*(D_\aa)$ which is an ample divisor, showing that $\widehat{X}$ is projective. A detailed proof of this fact will be given in the forthcoming paper \cite{R-fpCI} in the more general context of toric complete intersections.
\end{remark}

\begin{example}[Example \ref{ex:D} continued]\label{ex:D2} Let us consider the calibrated $f$-process described in Example~\ref{ex:D}.   As explained in Remark~\ref{rem:famiglie}
  \begin{equation*}
    V^T\cdot \L_\aa+\overline{A} = \left( \begin {array}{ccccccccccccccc} 0&1&0&2&1&0&3&2&1&0&4&3&2&1&0\\
    4&3&3&2&2&2&1&1&1&1&0&0&0&0&0\\
    0&0&1&0&1&2&0&1&2&3&0&1&2&3&4\end {array} \right)
  \end{equation*}
  and the family $\mathcal{Y}_\aa$ of plane quartics has general element given by the zero-locus of the polynomial
      \begin{eqnarray*}
        f_\aa&=&c_1x_2^4+c_2x_1x_3+c_3x_2^3x_3+ c_4x_1^2x_2^2+c_5x_1x_2^2x_3\\
        &&c_6x_2^2x_3^2+c_7x_1^3x_2+c_8x_1^2x_2x_3
        +c_9x_1x_2x_3^2+c_{10}x_2x_3^3\\
        &&+c_{11}x_1^4+c_{12}x_1^3x_3+c_{13}x_1^2x_3^2
        +c_{14}x_1x_3^3+c_{15}x_3^4
      \end{eqnarray*}
  Dually
  \begin{equation*}
    \L_\aa^T\cdot \overline{V}+\overline{B} = \left( \begin {array}{cccc} 1&5&2&0\\
                                                            5&1&2&0\\
                                                            0&0&1&3\end {array} \right)
  \end{equation*}
  meaning that the general element of the dual family $\mathcal{Y}_\bb$ of $\mathcal{Y}_\aa$ is a quotient, by the $\Z/4\Z$-action described in (\ref{azione2}), of the zero-locus in $\P(1,1,2)$ of the weighted homogeneous polynomial
  \begin{equation*}
    f_\bb=c_1x_1x_2^5+c_2x_1^5x_2+c_3x_1^2x_2^2x_3+c_4x_3^3
  \end{equation*}
\end{example}

\subsection{Generalizing Batyrev's duality}\label{ssez:Bat-dualita}
\noindent Definition~\ref{def:mirror} is clearly motivated by the case when $X$ is a Fano toric variety and $\aa=\1$, that is $D_\aa=-K_X$. In fact, recalling Remark~\ref{rem:batyrev}, in this case $f$-duality gives precisely the Batyrev's polar duality, inducing the well known \emph{mirror symmetry} $Y\leftrightsquigarrow Y^\vee$, being $Y$ and $Y^\vee$ both \cy varieties, up to suitable \emph{crepant} resolutions of singularities.

\subsection{Topological mirror test and Hodge diamond symmetry}\label{ssez:mirrortest} Let $Y$ be a generic hypersurface in a toric variety $X$ of degree $[D_\aa]\in\Cl(X)$.  If $Y$ is quasi-smooth and $X$ is $\Q$-factorial and complete, then there is a well defined concept of (coarse) moduli space $\mathcal{M}_Y$ (see e.g. \cite[\S13]{BC} and the recent \cite{Bunnet}). In this case, define $m_Y$ to be the dimension of the tangent space to $\mathcal{M}_{Y}$ at $[Y]$. By \cite[Prop.~13.7]{BC} one has
  \begin{equation*}
    m_Y=\dim\P\left(H^0(X,\cO_X(D_\aa)\right) -\dim\left(\Aut(X)\right)
  \end{equation*}
For a $\Q$-factorial and complete toric variety $X$, $\Aut(X)$ is an affine algebraic group of dimension
\begin{equation}\label{aut}
  \dim(\Aut(X))= \dim(X)+\sum_{\Theta} l^*(\Theta)
\end{equation}
where $\Theta$ ranges on the facets of the anti-canonical polytope $\D_{-K_X}=\D_\1$ \cite[Prop.~3.6.1 and 3.6.2]{CoxKatz}\footnote{Actually in \cite[Prop.~3.6.2]{CoxKatz} authors assume $X$ to be Gorenstein. Under this assumption $\D_{-K_X}$ is a lattice polytope, making easier to understand relation (\ref{aut}). Anyway, by the previous result  stated in \cite[Prop.~3.6.1]{CoxKatz}, the Gorenstein assumption may be dropped in getting (\ref{aut}).}, \cite[\S4]{Cox} and $l^*(\Theta)$ denotes \emph{the number of lattice points in the relative interior of the polytope $\Theta$}. Moreover,
\begin{equation}\label{h0}
  h^0(X,\cO_X(D_{\aa}))=l(\D_{\aa})
\end{equation}
so giving
\begin{equation}\label{mY}
  m_Y= l(\D_\aa)-1-n-\sum_{\Theta<^1\D_{-K_X}} l^*(\Theta)
\end{equation}
Unfortunately, conditions (1) and (2) opening the present \S\ref{sez:dualita-hyp} are not sufficient to guaranteeing quasi-smoothness neither of $Y$ nor of a $f$-mirror $Y^\vee$ of $Y$. Therefore, in the following, \emph{numbers $m_Y$ and $m_{Y^\vee}$ of complex moduli of $Y$ and $Y^\vee$}, respectively, will be combinatorially defined as the right term in (\ref{mY}). Namely
\begin{eqnarray}\label{mYY}
\nonumber
  m_Y &:=& l(\D_\aa)-1-n-\sum_{\Theta<^1\D_{-K_X}} l^*(\Theta) \\
  m_{Y^\vee} &:=& l(\D_\bb) - 1- n-  \sum_{\Theta<^1\D_{-K_{\XX_\aa}}} l^*(\Theta)
\end{eqnarray}

On the other hand, given a suitable resolution $\widehat{Y}\longrightarrow Y$, it is well defined the \emph{\ka moduli space} of $\widehat{Y}$ as the quotient of the \emph{complexified \ka cone} under the action of the automorphism group $\Aut(\widehat{Y})$ \cite[\S6.2]{CoxKatz}. Define
 \begin{equation*}
   k_{\widehat{Y}}\quad\text{be the dimension of the \ka moduli space of $\widehat{Y}$}
 \end{equation*}
 also called \emph{the number of \ka moduli of $\widehat{Y}$}. If $\widehat{Y}$ is a smooth hypersurface in a complete toric variety then, by the weak Lefschetz Theorem, its \ka cone, that is $\Nef(\widehat{Y})$, has dimension given by $h^{1,1}(\widehat{Y})$, and $\Aut(\widehat{Y})$ turns out to acting as a finite group (apply an argument similar to that given in \cite[\S6.2.3]{CoxKatz}, there proposed for a \cy toric hypersurface), so that $k_{\widehat{Y}}=h^{1,1}(\widehat{Y})$.

\begin{definition}\label{def:A,B-mirror}
Assume $n=\dim X\geq 4$.  Then we will say that:
\begin{itemize}
  \item[$(i)$] the ordered couple $(Y,Y^\vee)$ \emph{satisfies the A-side topological mirror test} if there exists a (partial) resolution of singularities $\widehat{Y}\longrightarrow Y$ such that $\widehat{Y}$ is (quasi-)smooth and
      \begin{equation*}
        k_{\widehat{Y}}= m_{Y^\vee}
      \end{equation*}
      In this case, we will also say that \emph{$Y^\vee$ is an $A$-mirror of $Y$};
  \item[$(ii)$] the ordered couple $(Y,Y^\vee)$ \emph{satisfies the B-side topological mirror test} if there exists a (partial) resolutions of singularities $\widehat{Y}^\vee\longrightarrow Y^\vee$ such that $\widehat{Y}^\vee$ is (quasi-)smooth and
      \begin{equation*}
        k_{\widehat{Y}^\vee}=m_Y
      \end{equation*}
      Then, we will also say that \emph{$Y^\vee$ is a $B$-mirror of $Y$};
  \item[$(iii)$] the ordered couple $(Y,Y^\vee)$ \emph{satisfies the Hodge diamond $A$-symmetry} if
  \begin{equation*}
    h^1(\widehat{\Omega}_{\widehat{Y}})=:h^{1,1}(\widehat{Y})=h^{n-2,1}(\widehat{Y}^\vee):=
    h^{n-2}(\widehat{\Omega}_{\widehat{Y}})
  \end{equation*}
where $\widehat{\Omega}:=i_*(\Omega)$ is the sheaf of Zariski differentials and $i$ is the inclusion of the smooth locus;
  \item[$(iv)$] the ordered couple $(Y,Y^\vee)$ \emph{satisfies the Hodge diamond $B$-symmetry} if
  \begin{equation*}
    h^{n-2,1}(\widehat{Y})=h^{1,1}(\widehat{Y}^\vee)
  \end{equation*}
\end{itemize}
Moreover, if both $(i)$ and $(ii)$ are satisfied we will say that the $f$-mirror partner $Y^\vee$ of $Y$ is actually a \emph{topological mirror} partner of $Y$ (and viceversa), and if both $(iii)$ and $(iv)$ are satisfied we will say that the $f$-mirror partner $Y^\vee$ of $Y$ is an \emph{Hodge mirror} partner of $Y$ (and viceversa).
\end{definition}

\begin{remark}
  Again, the above nomenclature is clearly inspired by the Calabi-Yau/Fano toric case for the toric hypersurface $Y\subset X$. The interested reader is referred to \cite[\S 6.1.2]{CoxKatz} and therein references, for a definition of $m_Y$: see in particular \cite[Prop.~6.1.3]{CoxKatz}. Due to the well known Bogomolov-Tian-Todorov-Ran Theorem and the \cy condition, if $Y$ is a \cy hypersurface in a Fano toric variety $X$, then
  \begin{equation*}
    m_Y=h^1(\widehat{Y},\mathcal{T}_{\widehat{Y}})=h^{2,1}(\widehat{Y})
  \end{equation*}
  meaning that, in the \cy case, $(i)\Leftrightarrow (iii)$ and $(ii)\Leftrightarrow (iv)$ and being a topological mirror partner is equivalent to being a Hodge mirror partner.
\end{remark}

\subsection{Mirror Web vs Mirror Symmetry}\label{ssez:MWeb}

Let $Y$ be a hypersurface in a toric variety $X$, both satisfying above conditions (1) and (2) opening the present \S \ref{sez:dualita-hyp}. Notice that the divisor $D_\aa\in\Weil(X)$ satisfying condition (1), may not be unique. Assume there exist two distinct divisors $D_{\aa_1}\sim D_{\aa_2}$ such that $Y\in|D_{\aa_1}|=|D_{\aa_2}|$ and $(X,D_{\aa_i})$ is a ftv, for both $i=1,2$. Then, $f$-duality may assign two distinct mirror partners $Y^\vee_i\in|D'_{\bb_i}|$, $i=1,2$, which, a priori, may be even non-isomorphic: observe that, in general, $D'_{\bb_1},D'_{\bb_2}$ are divisors living in distinct toric varieties $\XX_{\aa_1}$ and $\XX_{\aa_2}$, respectively.

 Such a phenomenon does not occur in the Calabi-Yau/Fano toric case, as there is a unique strictly effective divisor in the anti-canonical class of $X$, given by $D_\1\in[-K_X]$. In general, it makes then more sense to speak about a concept of \emph{mirror web $\mathcal{MW}$ of toric hypersurfaces} rather then about \emph{mirror symmetry}. More precisely:
 \begin{itemize}
   \item hypersurfaces connected by means of a calibrated $f$-process give rise to what will be called the \emph{$f$-mirror web} $f\mathcal{MW}$,
   \item hypersurfaces of dimension $\geq 3$ and connected by means of a calibrated $f$-process originating a topological mirror pair, give rise to the sub-web $T\mathcal{MW}\subset f\mathcal{MW}$,
   \item hypersurfaces of dimension $\geq 3$ and connected by means of a calibrated $f$-process originating a Hodge mirror pair, give rise to the sub-web $H\mathcal{MW}\subset f\mathcal{MW}$.
 \end{itemize}
 For \cy hypersurfaces $H\mathcal{MW}=T\mathcal{MW}= f\mathcal{MW}$.

 \begin{remark}\label{rem:mult.mirr}
   A similar phenomenon of multiple mirror partners is not a new one. As observed by Chiodo and Ruan \cite[Rem.~1]{Chiodo-Ruan}, examples of multiple mirrors can be easily obtained in the context of Berglund-H\"{u}bsch-Kravitz (BHK) duality between \cy hypersurfaces of (suitable quotients of) weighted projective spaces. But probably, the deepest known example of multiple mirrors is the R{\o}dland one \cite{Rodland}, then further studied and generalized by Borisov, Cald\u{a}r\u{a}ru and Libgober \cite{Borisov-C},\cite{Borisov-L} and Kuznetsov \cite{Kuznetsov}.   Moreover, this fact is well known from the point of view of Homological Mirror Symmetry, where the construction of Landau-Ginzburg (LG) mirror models is not in general expected  to producing unique mirror partners (see e.g. considerations following Def.~2.2 in \cite{KKOY} and the next Example~\ref{ssez:iperellittica}).
 \end{remark}

 \subsection{Generalizing Artebani-Comparin-Guilbot (ACG) duality}\label{ssez:ACG} In \cite{ACG} M.~Artebani, P.~Comparin and R.~Guilbot presented a way of extending Batyrev's dua\-li\-ty of families of anti-canonical hypersurfaces in Fano toric varieties, to suitable sub-families whose associated Newton polytope is \emph{canonical}, that is, a lattice sub-polytope of the anti-canonical polytope admitting the origin as a unique interior point \cite[\S 2]{ACG}.  As observed in \S\ref{ssez:Bat-dualita}, $f$-duality is an extension of Batyrev's duality. Then $f$-duality applies to give an extension of ACG-duality beyond the realm of \cy hypersurfaces.

 Namely, set the following assumptions:
 \begin{enumerate}
   \item let $(X,D_\aa)$ be a ftv admitting a calibrated $f$-process,
   \item let $\D$ be a lattice sub-polytope of $\D(X,\aa)$ containing the origin as an interior point: thinking of $\D$ as a Newton polytope, it describes a sub-family $\mathcal{Y}_\D\subseteq\mathcal{Y}_\aa$ of the family of hypersurfaces in $X$ of degree $[D_\aa]\in \Cl(X)$;
   \item consider the toric variety $\XX_\D$, which is complete by Proposition~\ref{prop:Fmatricedipolitopo}, and the framing $D_\v$, assigned by the minimum strictly positive column vector $\v$ such that
       \begin{equation*}
         V_\D^T\cdot V + \underbrace{\left(\,\v\ \cdots\ \v\,\right)}_{m_\D\ \text{times}}\geq \0
       \end{equation*}
       where $V$ and $V_\D$ are fan matrices of $X$ and $\XX_\D$, respectively, and $m_\D=\rk\Weil(\XX_\D)$; assume the ftv $(\XX_\D,D_\v)$ admitting a calibrated $f$-process;
   \item finally assume that $\D(\XX_\aa,\bb)=[\D_\bb]$ (recall the last assertion in Theorem~\ref{thm:Deltatriviale}) is a lattice sub-polytope of $\D(\XX_\D,\v)=[\D_\v]$, where
       \begin{equation*}
         \D_\v:=\{\n\in N_\R\,|\,V_\D^T\cdot\n\geq -\v\}
       \end{equation*}
 \end{enumerate}
Then, thinking of $[\D_\bb]$ as a Newton polytope, it describes a sub-family $\mathcal{Y}_\bb\subseteq\mathcal{Y}_\v$ of the family of hypersurfaces in $\XX_\D$ of degree $[D_\v]\in\Cl(\XX_\D)$\,.

\begin{definition}
  The family $\mathcal{Y}_\bb$ is called a  \emph{$f$-ACG dual family} of the family $\mathcal{Y}_\D$\,.
\end{definition}

\begin{remarks}
  \begin{enumerate}
    \item By construction, $\mathcal{Y}_\D$  is a $f$-ACG dual family of $\mathcal{Y}_\bb$\,.
    \item If $\D=\D_\aa$, then $f$-ACG duality reduces to $f$-duality exhibiting $\mathcal{Y}_\bb$ as a $f$-dual family of $\mathcal{Y}_\aa$.
    \item If $\aa=\1$, that is $D_\aa=-K_X$, than $\bb=\1$, too, and $(\D,\D_\aa)$ turns out to be a \emph{good pair} in the sense of \cite[Def.1.4]{ACG}. In particular, assumptions (1), (3) and (4) follow immediately, and $f$-ACG duality reduces to giving just ACG duality between families $\mathcal{Y}_\D$ and $\mathcal{Y}_\bb$ of \cy varieties \cite[Thm.~1]{ACG}.
  \end{enumerate}
\end{remarks}

\begin{example}\label{ex:D-ACG}
To better understand the level of generalization introduced by $f$-duality, the present example should be compared with \cite[Ex.~3.3]{ACG}.

Consider the ftv $(X,\aa)=(\P^2,(1,1,2))$ given in Example~\ref{ex:D} and notation there introduced. We are then looking for a suitable sub-family of plane quartics admitting a $f$-ACG dual family.
Consider the sub-polytope $\D\subseteq\D_\aa$ given by
\begin{equation*}
  \D=\conv(V_\D)\ ,\quad\text{with}\ V_\D=\left(
                                            \begin{array}{cccc}
                                              1 & -1 & 1 & -1 \\
                                              -1 & 1 & 1 & -1 \\
                                            \end{array}
                                          \right)
\end{equation*}
\begin{figure}
\begin{center}
\includegraphics[width=12truecm]{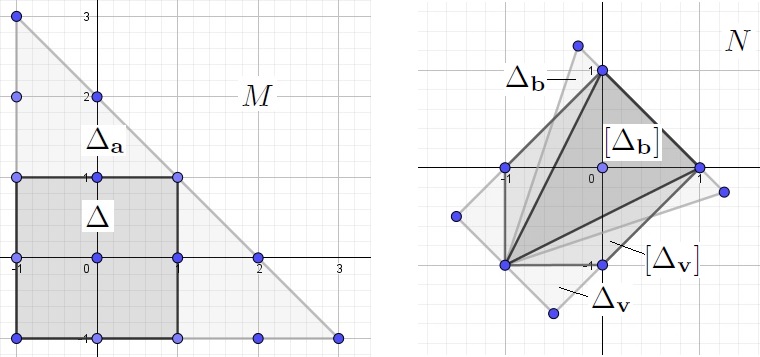}
\caption{\label{Fig2} Example \ref{ex:D-ACG}: Newton polytopes $\D\subseteq\D_\aa\subset M_\R$ and $[\D_\bb]~\subseteq~[\D_\v]~\subset~N_\R$.
Notice that $\D_\bb\nsubseteq\D_\v$.}
\end{center}
\end{figure}
One can then easily check that (see Fig.~\ref{Fig2}):
\begin{itemize}
  \item as observed in Example \ref{ex:D}, a $f$-dual ftv $(\XX_\aa,\bb)$ of $(X,\aa)$ is given by choosing $\bb:=(2,2,1)$, where $\XX_\aa$ is a quotient of the weighted projective space $\P(\aa)=\P(1,1,2)$ by the action of $\Z/4\Z$ described in (\ref{azione2}); in particular, the $f$-process $(X,\aa)\leftrightsquigarrow (\XX_\aa,\bb)$ is calibrated, satisfying assumption (1) above;
  \item clearly the origin of $M$ is an interior point of $\D$, so giving assumption (2);
  \item recalling Remark~\ref{rem:famiglie} and Example~\ref{ex:D2}, and observing that
  \begin{equation*}
    V^T\cdot \overline{V}_\D+\overline{A} = \left( \begin {array}{ccccccccc} 2&1&0&2&1&0&2&1&0\\
    2&2&2&1&1&1&0&0&0\\
    0&1&2&1&2&3&2&3&4\end {array} \right)
  \end{equation*}
  the sub-family $\mathcal{Y}_\D\subseteq\mathcal{Y}_\aa$ of plane quartics has general element given by the zero-locus of the polynomial
      \begin{eqnarray*}
        f_\D&=& c_4x_1^2x_2^2+c_5x_1x_2^2x_3+c_6x_2^2x_3^2+c_8x_1^2x_2x_3\\
        &&+c_9x_1x_2x_3^2+c_{10}x_2x_3^3+c_{13}x_1^2x_3^2+c_{14}x_1x_3^3+c_{15}x_3^4
      \end{eqnarray*}
  \item the toric variety $\XX_\D$, which is the unique complete and $\Q$-factorial one, whose fan matrix is given by $V_\D$, is a quotient of $\P^1\times\P^1$ by the action of $\Z/2Z$ defined by sending
      \begin{equation*}
        (\eta,([x_1:x_2],[y_1:y_2])\mapsto ([\nu x_1:\nu^{-1}x_2],[y_1:y_2])\ ,\quad\text{where}\ \nu=\exp(\eta\pi i)
      \end{equation*}
  \item observing that
  \begin{equation*}
    V_\D^T\cdot V = \left( \begin {array}{ccc} 1&-1&0\\
                                                -1&1&0\\
                                                1&1&-2\\
                                               -1&-1&2\end {array} \right)
  \end{equation*}
  the framing $\v$ of $\XX_\D$ is given by the minimum positive vector such that
  \begin{equation*}
    V_\D^T\cdot V + \left(
                      \begin{array}{ccc}
                        \v & \v & \v \\
                      \end{array}
                    \right)\geq \0\ \Longrightarrow\ \v=\left(
                                                          \begin{array}{c}
                                                            1 \\
                                                            1 \\
                                                            2 \\
                                                            1 \\
                                                          \end{array}
                                                        \right)
  \end{equation*}
  \item lattice polytopes $[\D_\bb]$ and $[\D_\v]$ are given by
  \begin{equation*}
    [\D_\bb]=\conv(V)\subseteq[\D_\v]=\conv(\L_\v)\,,\ \text{with}\ \L_\v=\left(
                                          \begin{array}{ccccc}
                                            1 & -1 & 0 & 0 & -1 \\
                                            0 & 0 & 1 & -1 & -1 \\
                                          \end{array}
                                        \right)
  \end{equation*}
  so guaranteeing assumption (4);
  \item assumption (3), that is, $(\XX_\D,\v)$ is admitting a calibrated $f$-process, is checked by observing that
      \begin{equation*}
        \L_\v^T\cdot V_\D= \left( \begin {array}{cccc} 1&-1&1&-1\\
                                                        -1&1&-1&1\\
                                                        -1&1&1&-1\\
                                                        1&-1&-1&1\\
                                                        0&0&-2&2\end {array} \right)\ \Longrightarrow\ \w=\left(
                                                    \begin{array}{c}
                                                      1 \\
                                                      1 \\
                                                      1 \\
                                                      1 \\
                                                      2 \\
                                                    \end{array}
                                                  \right)
      \end{equation*}
      and noticing that
  \begin{equation*}
    [\D_\w]=[\{\n\in N_\R\,|\,\L_\v^T\cdot \n\geq -\w\}]=\D
  \end{equation*}
\end{itemize}
Therefore, there is a well defined $f$-ACG dual family $\mathcal{Y}_\D^\vee=\mathcal{Y}_{[\D_\bb]}$ of the family $\mathcal{Y}_\D$, described by $[\D_\bb]$ as a Newton polytope of hypersurfaces inside the family $\mathcal{Y}_\v$ of hypersurfaces of degree $[D_\v]\in\Cl(\XX_\D)$. Namely,
\begin{equation*}
  V_\D^T\cdot \overline{V}+\left(
                             \begin{array}{cccc}
                               \v & \v & \v & \v \\
                             \end{array}
                           \right)
  =\left( \begin {array}{cccc} 0&2&1&1\\
                                                        2&0&1&1\\
                                                        3&3&2&0\\
                                                        0&0&1&3\end {array} \right)
\end{equation*}
so giving that the general element of $\mathcal{Y}_{[\D_\bb]}$ is a quotient, by the $\Z/2\Z$-action described above, of the zero-locus of the polynomial
\begin{equation*}
  f_{[\D_\bb]}=c_1x_2^2y_1^3+c_2x_1^2y_1^3+c_3x_1x_2y_1^2y_2+c_4x_1x_2y_2^3
\end{equation*}
\end{example}

\begin{remark}
  Notice that, given assumptions from (1) to (4) above, it is not true, in general, that $\D_\bb$ is a sub-polytope of $\D_\v$. In fact, in the previous Example~\ref{ex:D-ACG}
  \begin{equation*}
    \D_\bb=\conv\left( \begin {array}{ccc} 5/4&-1/4&-1\\
                                            -1/4&5/4&-1\end {array} \right)\ \nsubseteq\ \D_\v=\left( \begin {array}{cccc} 1&0&-1/2&-3/2\\
                                                                                0&1&-3/2&-1/2\end {array} \right)
  \end{equation*}
  (see the right part of Fig.~\ref{Fig2}).
\end{remark}

\subsection{Generalizing Berglund-H\"{u}bsch-Krawitz (BHK) duality}\label{ssez:BHK}

In 1993, physicist Berglund and H\"{u}bsch \cite{Berglund-Hubsch} presented a first generalization of the mirror symmetric Greene-Plesser construction \cite{GP}. Their construction appeared just before the Batyrev's one \cite{Batyrev94} and it is, in a sense, ``orthogonal'' to the latter. The intersection between the two is just the Greene-Plesser example of the quintic threefold. The Berglund-H\"{u}bsch construction was later refined by Krawitz \cite{Krawitz}. For this reason, it is often quoted as the \emph{Berglund-H\"{u}bsch-Krawitz (BHK) duality}. In \cite[\S 4]{ACG} Artebani, Comparin and Guilbot showed how their new ACG-duality ge\-ne\-ra\-lizes BHK-duality from \cy hypersurfaces of (a quotient of) a weighted projective space, whose polynomial is of Delsarte type, that is, same number of monomials and variables, to \cy hypersurfaces of (a quotient of) a $\Q$-Fano toric variety \cite[\S 4.2]{ACG} \footnote{Artebani, Comparin and Guilbot asked for $\Q$-Fano toric varieties with \emph{torsion free class group}, when presenting their generalization. Actually this hypothesis is unnecessary, as it was confirmed to me by Artebani (private communication).}.

In the previous \S \ref{ssez:ACG}, we introduced the $f$-ACG duality, which is a generalization of ACG-duality to suitable subfamilies of hypersurfaces in toric varieties. Clearly the same approach gives a generalization of BHK-duality, which will be called $f$-BHK duality. Namely,
\begin{itemize}
  \item if a lattice sub-polytope $\D\subseteq [\D_\aa]$ satisfies assumption from (1) to (4) in \S\ref{ssez:ACG}, then sub-families $\mathcal{Y}_{\D(0)}$ and $\mathcal{Y}_{\D(0)}^\vee=\mathcal{Y}_{[\D_\bb](0)}$, generated by vertices of the lattice polytopes $\D$ and $[\D_\bb]$, respectively, will be called \emph{$f$-BHK dual families}.
\end{itemize}

\begin{example}\label{ex:D-BHK}
  Consider the previous Example~\ref{ex:D-ACG} and the sub-family $\mathcal{Y}_{\D(0)}\subset \mathcal{Y}_{\D}$, whose general element is the zero-locus of the polynomial
  \begin{equation*}
    f_{\D(0)}=c_4x_1^2x_2^2+c_6x_2^2x_3^2+c_{13}x_1^2x_3^2+c_{15}x_3^4
  \end{equation*}
  Then its $f$-BHK dual family is given by the sub-family $\mathcal{Y}_{[\D_\bb](0)}\subset\mathcal{Y}_{[\D_\bb]}$, whose general element is a quotient, by the $\Z/2\Z$-action described above, of the zero-locus of the polynomial
  \begin{equation*}
    f_{[\D_\bb(0)]}=c_1x_2^2y_1^3+c_2x_1^2y_1^3+c_4x_1x_2y_2^3
  \end{equation*}
\end{example}

\begin{remark}
  In the previous Example~\ref{ex:D-BHK}, both $f_{[\D_\bb](0)}$ and  $f_{\D(0)}$ are no more polynomials of Delsarte type. Notice that this fact may also occur in the ACG generalization of BHK-duality (see \cite[Ex.~4.12]{ACG}).
\end{remark}

\subsection{Framed duality versus Krawitz duality}\label{ssez:K-dualita}

Given a pair of framed toric varieties linked by a calibrated $f$-process
$$(X,\aa)\stackrel{\text{$f$-process}}{\leftrightsquigarrow}(\XX_\aa,\bb)$$
generic hypersurfaces $Y\in|D_\aa|$ and $Y^\vee\in|D_\bb|$ may be very singular, making quite difficult finding suitable resolutions $\widehat{Y}$ and $\widehat{Y}^\vee$ and compute all the needed Hodge numbers to check the various instances of mirror symmetry as explained in \S\ref{ssez:mirrortest}. According with Chiodo and Ruan \cite{Chiodo-Ruan}, it is generally believed that considering suitably associated Landau-Ginzburg (LG) models may sensibly simplify singularities and giving rise to alternative way of checking mirror symmetry.

In the present section, a sort of a \emph{LG/Hypersurface correspondence} is presented, as an extension of the LG/CY correspondence, studied by Chiodo and Ruan \cite{Chiodo-Ruan}, in the case of Delsarte \cy hypersurfaces, and also by Chiodo, Kalashnikov and Veniani in the recent \cite{CKV}, beyond the \cy setting. As it will be observed in the next \S\ref{sez:ipersuperfici}, in the case of projective hypersurfaces, the associated LG models turn out to be even smooth. A similar LG/Hypersurface correspondence, translates the mirror duality at a level of LG models. The latter has been described, for hypersurfaces of Delsarte type in weighted projective spaces, by Krawitz \cite{Krawitz} by means of an extension of Berglund-H\"{u}bsch duality without any \cy condition. ACG extension of BHK-duality and, furthermore, considerations given in the previous \S\ref{ssez:ACG} and \S\ref{ssez:BHK}, allows us \emph{to think of f-duality in terms of a generalized Krawitz duality}, as stated in the following Theorem~\ref{thm:Krawitz}, so getting a geometric equivalent condition to the existence of a calibrated $f$-process. Compare also with the more recent \cite{HSSW}, where He, Si, Shen and Webb give an interesting improvement of Krawitz duality.

\subsubsection{A LG/Hypersurface correspondence}\label{sssez:LG/Hyp} Given a ftv $(X,\aa)$ and a generic hypersurface $Y\in|D_\aa|$, let $\T\cong(\C^*)^n$ be the acting torus on $X$. Consider the torus hypersurface $Z:=\T\cap Y$. Recalling Remark~\ref{rem:famiglie}(a), $Y$ is the zero locus of the polynomial $f$ in (\ref{f-WT}), generated by the columns of the matrix $\overline{M}_\aa+\overline{A}$. Consider the Laurent polynomial
\begin{equation*}
  f_\aa:={f\over \x^\aa}\in\C[\x,\x^{-1}]
\end{equation*}
generated by the columns of the matrix $\overline{M}_\aa$. Notice that, in $\T$ both $f$ and $f_\aa$ admit the same zero-locus $Z$, that is,
\begin{equation*}
  Z=\T\cap f^{-1}(0)=\T\cap f_\aa^{-1}(0)
\end{equation*}
In particular, $f_\aa$ defines a function $f_\aa:\T\longrightarrow\C$, so giving a LG model $(\T,f_\aa)$ admitting a Laurent superpotential.

On the other hand, following Remark~\ref{rem:famiglie}(b), let $\T_\aa\cong(\C^*)^n$ be the acting torus on $\XX_\aa$ and $Z^\vee:=\T_\aa\cap Y^\vee$. Consider the Laurent polynomial
\begin{equation*}
  f^\vee_\bb:={f^\vee\over \x^\bb}\in\C[\x,\x^{-1}]
\end{equation*}
where $f^\vee$ is the polynomial given in (\ref{fdual}), generated by the columns of the matrix $\overline{M}_{\aa,\bb}+\overline{B}$\,. In particular, $f^\vee_\bb$ defines a function $f^\vee_\bb:\T_\aa\longrightarrow\C$, so giving a LG model $(\T_\aa,f^\vee_\bb)$ with a Laurent superpotential.

\begin{definition}[K-duality of Laurent LG models]
  Let $(\T,f)$ and $(\T',f')$ be two \emph{Laurent LG models}, that is, $\T$ and $\T'$ are algebraic tori and $f,f'$  Laurent superpotentials.  Then they are said  \emph{Krawitz dual (K-dual)} if
  \begin{equation*}
    M^T=M'\quad\text{(up to a permutation of columns)}
  \end{equation*}
  being $M, M'$ \emph{vertex matrices of the Newton polytopes} of $f$ and $f'$, respectively, defined up to a permutation on vertices.
\end{definition}

\begin{theorem}\label{thm:Krawitz}
  If a ftv $(X,\aa)$ admits a calibrated $f$-process
$$(X,\aa)\leftrightsquigarrow(\XX_\aa,\bb)$$
   then the associated Landau-Ginzburg models $(\T,f_\aa)$ and $(\T_\aa,f^\vee_\bb)$ are K-dual, that is,
\begin{equation}\label{K-dual}
  M_{\aa,\bb}=M_\aa^T
\end{equation}
This gives rise to the following commutative diagram of LG/Hypersurfaces correspondences and mirror dualities
  \begin{equation*}
    \xymatrix{Y\ar@{<~>}[d]_-{\text{LG/Hyp}}\ar@{<~>}[rr]^-{\text{$f$-MS}}&&Y^\vee
    \ar@{<~>}[d]^-{\text{LG/Hyp}}\\
    (\T,f_\aa)\ar@{<~>}[rr]^-{\text{K-duality}}&&(\T_\aa,f^\vee_\bb)}
  \end{equation*}
Viceversa, if condition (\ref{K-dual}) and the second condition displayed in (\ref{min}) are satisfied then one gets a calibrated $f$-process.
\end{theorem}

\begin{proof} Assume, at first, that $(X,\aa)$ admits a calibrated $f$-process. Then,
  Theorem~\ref{thm:Deltatriviale} implies that one can assume $\L_\bb=V$, up to a change of ge\-ne\-ra\-tors in lattices $M$ and a permutation of columns. Then
  \begin{equation*}
    M_{\aa,\bb}= \L_\aa^T\cdot\L_\bb= \L_\aa^T\cdot V = M_{\aa}^T
  \end{equation*}
On the other hand, assume $M_{\aa,\bb}=M_\aa^T$, with
\begin{equation*}
  M_\aa= V^T\cdot \L_\aa\ ,\quad M_{\aa,\bb}=\L_\aa^T\cdot\L_\bb
\end{equation*}
being $V,\L_\aa,\L_\bb$ be fan matrices of $X,\XX_\aa,\XX_\bb$, respectively. Then
\begin{equation}\label{eqzmat}
  \L_\aa^T\cdot V= \L_\aa^T\cdot\L_\bb
\end{equation}
Being $\L_\aa$ a fan matrix, there exist two invertible matrices $A$ and $U$ such that
\begin{equation*}
  \L_\aa^T=\widehat{\L}_\aa^T\cdot A\ ,\quad \widehat{\L}_\aa^T=U\cdot\left(
                                                                        \begin{array}{c}
                                                                          I_n \\
                                                                          \0 \\
                                                                        \end{array}
                                                                      \right)
\end{equation*}
\cite[Prop.~3.1(3), Prop.~2.6(2)]{RT-LA&GD}. Then, multiplying on the left equation (\ref{eqzmat}) by $U^{-1}$, one gets
\begin{equation*}
  \left(
                                                                        \begin{array}{c}
                                                                          I_n \\
                                                                          \0 \\
                                                                        \end{array}
                                                                      \right)\cdot A\cdot V=\left(
                                                                        \begin{array}{c}
                                                                          I_n \\
                                                                          \0 \\
                                                                        \end{array}
                                                                      \right)\cdot A\cdot\L_\bb\ \Longleftrightarrow\ A\cdot V=A\cdot \L_\bb
\end{equation*}
Multiplying the latter on the left by $A^{-1}$, one finally gets $V=\L_\bb$. Then $\XX_\bb=X$ and, calling $\cc$ the f-dual framing of $(\XX_\aa,\bb)$ on the f-dual toric variety $\XX_\bb$, one gets $\cc=\aa$ by the same argument given in (\ref{c=a}). Then $(\XX_\bb,\cc)=(X,\aa)$, so proving that $(X,\aa)$ admits a calibrated $f$-process.
\end{proof}

\begin{remark}
  Example~\ref{ex:due framing 3} shows that the second condition displayed in (\ref{min}) is needed to get the necessary condition in Theorem~\ref{thm:Krawitz}.
\end{remark}

\begin{remark}\label{rem:K-dualita}
The previous Theorem~\ref{thm:Krawitz} leads to an alternative conjectural approach, of checking mirror symmetry for an $f$-mirror pair $(Y,Y^\vee)$, following the lines described in \cite{Chiodo-Ruan}. Namely, Krawitz established a Mirror Theorem for LG mo\-dels corresponding to \emph{quasi-homogeneous and non-degenerate} weighted hypersurfaces defined by Delsarte polynomials and linked by Berglund-H\"{u}bsch duality \cite[Thm.~1.1]{Krawitz}: the Krawitz mirror map is constructed by means of a bi-graded isomorphism between suitable graded vector spaces associated with the involved superpotentials. See also \cite{HSSW} for an interesting improvement of this Mirror Theorem for LG models. Then, Chiodo and Ruan proved, under the further \cy condition, that those graded vector spaces are related with the cohomology of suitable resolutions $\widehat{Y}$ and $\widehat{Y}^\vee$ of $Y$ and $Y^\vee$, respectively \cite[Thm.~16, Cor.~17]{Chiodo-Ruan}. Superpotentials involved in the statement of Theorem~\ref{thm:Krawitz}, can be assumed quasi-homogeneous, by considering $f$ and $f^\vee$ rather than  $f_\aa$ and $f^\vee_\bb$, respectively. But in general they cannot be assumed neither non-degenerate nor Delsarte, so imposing a deep revision of the Krawitz construction. Moreover, the lack of any CY condition imposes a deeper understanding of relations between Chen-Ruan cohomology and the usual cohomology of $Y$ and $Y^\vee$ (in this sense,  consider also \cite{CKV}, for a slight relaxation of the CY condition).
\end{remark}

\subsection{KKP-compactification of associated LG models and log geometry}\label{ssez:KKP-compct}

Landau-Ginzburg models associated with an $f$-mirror pair $(Y,Y^\vee)$ as in \S\ref{sssez:LG/Hyp}, admit a \emph{compactification} in the sense of Katzarkov-Kontsevich-Pantev \cite[Def.~2.4]{KKP}, exhibiting a log geometry which is that of a log \cy defined by Gross and Siebert \cite[Def.~1.10]{GS-IMS}, where the simple normal crossings divisor $D$ is replaced by the framing of the considered ftv.

Namely, under notation introduced in the previous \S\ref{sssez:LG/Hyp}, $K$-dual superpotential functions $f_\aa:\T\longrightarrow\C$ and $f^\vee_\bb:\T_\aa\longrightarrow\C$ admit the following properifications
\begin{equation*}
  \xymatrix{\T\ar[d]^-{f_\aa}\ar@{^(->}[r]&X\ar[d]^-{\overline{f}_\aa:=[f:\x^\aa]}\\
            \C\ar@{^(->}[r]&\P^1}\quad \xymatrix{\ar@{<~>}[rr]<-18pt>^-{\text{$K$-duality}}&&}\quad
   \xymatrix{\T_\aa\ar[d]^-{f^\vee_\bb}\ar@{^(->}[r]&\XX_\aa\ar[d]^-{\overline{f}^\vee_\bb:=[f^\vee:\x^\bb]}\\
            \C\ar@{^(->}[r]&\P^1}
\end{equation*}
Notice that:
\begin{enumerate}
  \item $\overline{f}_\aa^{-1}([0:1])=Y\subset X$ and $\overline{f}_\aa^{-1}([1:0])=D_\aa\subset X$
  \item $(\overline{f}^\vee_\bb)^{-1}([0:1])=Y^\vee\subset \XX_\aa$ and $(\overline{f}^\vee_\bb)^{-1}([1:0])=D'_\bb\subset \XX_\aa$
  \item families $\mathcal{Y}_\aa=\{Y\in|D_\aa|\}$ and $\mathcal{Y^\vee}_\bb=\{Y^\vee\in|D'_\bb|\}$ give rise to corresponding families of LG models $\{(\T,f_\aa)\}$ and $\{(\T_\aa,f^\vee_\bb)\}$, respectively, whose variation turns out to be ``anchored at infinity'' by their framing; when (suitable resolutions of) $Y$ and $Y^\vee$ are \cy varieties, then these families of LG models are precisely those considered in \cite{KKP}, meaning that, in this case, their spaces of ``anchored'' versal deformations are smooth \cite[Thm.~A]{KKP};
  \item recalling the Gross-Siebert definition of a log \cy pair \cite[Def.~1.10]{GS-IMS}, one has
      $$K_X+D_\aa\sim D_{\aa-\1}\quad\text{and}\quad K_{\XX_\aa}+D'_\bb\sim D_{\bb-\1}$$
      so giving effective divisors supported on $\bigcup_iD_i$ and $\bigcup_jD'_j$, respectively; by this point of view, framed toric varieties $(X,\aa)$ and $(\XX_\aa,\bb)$ may be understood as \emph{log pairs}, no more \cy as $D_\aa$ and $D'_{\bb}$ have only normal crossings; this gives an hint about how thinking of the $f$-duality in the context of \emph{Intrinsic Mirror Symmetry} \cite{GS-IMS}.
\end{enumerate}

\section{Framing $\P^n$ and associated dual partners}\label{sez:framingPn}

A projective space $\P^n$ is a smooth and complete toric variety associated with the fan matrix
\begin{equation}\label{V}
  V=\left(
      \begin{array}{ccc}
        I_n & | &-\mathbf{1} \\
      \end{array}
    \right)=\left(
              \begin{array}{cccc}
                \e_1 & \cdots & \e_n & -\mathbf{1}\\
              \end{array}
            \right)
    \in\mathbf{M}(n,n+1;\Z)
\end{equation}
and the unique fan $\Si\in\SF(V)$, given by all the faces of the $n+1$, maximal, $n$-dimensional cones, generated by every choice of $n$ columns of $V$. For this complete toric variety, it turns out that \emph{condition (\ref{min}) in Theorem~\ref{thm:Deltatriviale} is satisfied for a sufficiently large number of framings}: this is the content of the following result. Notice that the following construction depends on the divisor chosen as a framing, that is, a selected section of a line bundle, and not on the line bundle itself, just as the Greene-Plesser and, more in general, the Berglund-H\"{u}bsch duality. In fact, the reader can check that different choices of sections in the same line bundle can give rise to opposite phenomena: e.g. for projective plane sixtics, the ftv $(\P^2,(1,1,4))$ admits a calibrated process, as explained in the next theorem, but the ftv $(\P^2,(2,2,2))$ does not.

\begin{theorem}\label{thm:dualita}
  Let $D_\aa=\sum_{i=0}^{n+1}a_iD_i$ be a strictly effective divisor of $\P^n$. Then $(\P^n,D_\aa)$ is a ftv.

\noindent For every $i=1,\ldots,n+1$\,, define $d_i:=\gcd(\{a_j\,|\,j\neq i\})$
and assume that
\begin{equation}\label{convenzione}
  a_1\leq a_2\leq \cdots \leq a_{n+1}\quad\text{and}\quad \gcd(a_1,\ldots,a_{n+1})=1
\end{equation}
Let $\NN'\subset N_\R\cong\R^n$ be the polytope given by the convex hull of suitable multiples of the standard basis, as follows
\begin{equation*}
     \NN'=\conv\left(\0,\e_1,{a_n\over a_{n-1}}\e_2,\ldots,{a_n\over a_1}\e_n\right)\quad(\text{notation as in \emph{(\ref{V})}}\,)
   \end{equation*}
    Then, the $f$-process associated with the ftv $(\P^n,D_\aa)$ is calibrated if and only if the following conditions hold:
    \begin{itemize}
      \item[(a)] $\conv(\NN'\cap N)=\conv\left(\{\0\}\cup\left\{\left[{a_n\over a_{n-i+1}}\right]\e_i\,|\,\forall\,1\leq i\leq n \right\}\right)$\,,
      \item[(b)] $\exists\,i,j\in\{1,\ldots,n+1\}:\ i\neq j\,,\  d_i=d_j=1$\,.
    \end{itemize}
  In this case, the associated $f$-dual ftv is given by $(\XX_\aa,D'_\bb)$ with
  \begin{eqnarray*}
    \XX_\aa &\cong& \P(\q)/G_\aa\quad\text{where $\q$ is the reduced weight vector of}\ \aa
                                                                             \\
    D'_\bb &=& \sum_{i=1}^{n+1}b_iD'_i\quad\text{where}\ b_i=\left\{\begin{array}{cc}
                                                               a_{n+1}/d_i & \text{for}\ i\leq n \\
                                                               a_{n}/ d_{n+1} & \text{for}\ i= n+1
                                                             \end{array}\right.
  \end{eqnarray*}
  being $D'_1,\ldots,D'_{n+1}$ the torus invariant prime divisors generating $\Weil(\XX_\aa)$ and $G_\aa$ a finite abelian group of order
  \begin{equation}\label{Ga-order}
    |G_\aa|=\left(\sum_{i=1}^{n+1}a_i\right)^{n-1}
  \end{equation}
  whose action on the weighted projective space $\P(\q)$ is represented by a \emph{torsion matrix} $\Ga$ as follows: by setting
  $$G_\aa\cong\Z/\tau_1\Z\oplus\cdots\oplus\Z/\tau_s\Z$$
  with $\tau_1|\tau_2|\cdots|\tau_s$, the action is given by
  \begin{eqnarray*}
     &\xymatrix{\left(\bigoplus_{k=1}^s\Z/\tau_k\Z\right)\times\P(\q)\ar[rrr]^-{\Ga=([\g_{k,j}]_{\tau_k})}&&&\P(\q)}\hskip1.5truecm&  \\
     &\xymatrix{\left(([\ve_1]_{\tau_1},\ldots,[\ve_s]_{\tau_s}),[x_1:\ldots :x_{n+1}]\right)\ar@{|->}[r]&\left[\left(\prod_{k=1}^s\exp\left({2\pi i \g_{k,j}\ve_k\over\tau_k}\right)\right)x_j\right]_{j=1}^{n+1} }&
  \end{eqnarray*}
  where $$\Ga=\left(
              \begin{array}{ccc}
                [\g_{1,1}]_{\tau_1} & \cdots & [\g_{1,n+1}]_{\tau_1} \\
                \vdots &  & \vdots \\
                 {[}\gamma_{s,1}]_{\tau_s}  & \cdots & [\g_{s,n+1}]_{\tau_s} \\
              \end{array}
            \right)
  $$
  is represented by $(\g_{k,j})\in \mathbf{M}(s,n+1;\Z)$ constructed by means of the next Algorithm~\ref{algoritmoG}.\\
  In particular, if $\aa$  is a reduced weight vector then $\XX_\aa\cong\P(\aa)/G_\aa$ and
  $$\bb=\left(
          \begin{array}{cccc}
            a_{n+1} &  \cdots & a_{n+1} & a_n \\
          \end{array}
        \right)
  $$
\end{theorem}

\begin{algorithm}\label{algoritmoG}
The torsion matrix $\Ga$, representing the $G_\aa$-action giving $\XX_\aa=\P(\q)/G_\aa$ in the previous Theorem~\ref{thm:dualita}, is defined in display (3) of \cite[Thm.~3.2]{RT-Erratum}. Namely:
\begin{enumerate}
    \item consider a fan matrix $\widetilde{\L}$ of $\P(\q)$ such that $A\cdot\L_\aa=\b\cdot\widetilde{\L}$, with $$A\in\GL_n(\Z)\quad\text{and}\quad\b=\diag\left(\underbrace{1,\ldots,1}_{n-s},\tau_1,\ldots\tau_s\right)$$
    \item consider the following matrix $U_\q\in\GL_{n+1}(\Z)$ sending the transposed weight vector $\q^T$ in Hermite normal form (HNF):
        \begin{equation*}
          U_\q=\left(
                  \begin{array}{c}
                    \uu \\
                    \widetilde{\L} \\
                  \end{array}
                \right)\ \Longrightarrow\ U_\q\cdot\q^T=\left(
                                                          \begin{array}{c}
                                                            1 \\
                                                            0 \\
                                                            \vdots \\
                                                            0 \\
                                                          \end{array}
                                                        \right)
        \end{equation*}
    \item let $^{n+1-s}U_\q$ be the submatrix of $U_\q$ given by the upper $n+1-s$ rows and consider the matrix $W\in\GL_{n+1}(\Z)$ sending the transposed matrix $^{n+1-s}U_\q^T$ in HNF, that is
        \begin{equation*}
          W\cdot\  ^{n+1-s}U_\q^T=\HNF\left(\,^{n+1-s}U_\q^T\right)
        \end{equation*}
        \item consider the submatrices $_s\widetilde{\L}$ and $_sW$ of $\widetilde{\L}$ and $W$, respectively, assigned by the lower $s$ rows and define the following $s\times s$ integer matrix
            \begin{equation*}
              G:=_s\widetilde{\L}\cdot\, _sW^T\in\mathbf{M}(s,s;\Z)
            \end{equation*}
            \item finally, consider $U_G\in\GL_s(\Z)$ sending the transposed matrix $G^T$ in HNF, that is $U_G\cdot G^T=\HNF(G^T)$, and define
                \begin{equation*}
                  (\g_{k,i}):=U_G\cdot\,_sW\in \mathbf{M}(s,n+1;\Z)\ \Longrightarrow\ \Ga:=(\g_{k,i})\mod \boldsymbol{\tau}
                \end{equation*}
  \end{enumerate}
\end{algorithm}

\begin{proof}[Proof of Theorem~\ref{thm:dualita}] The first part of this proof will describe the $f$-dual ftv $(\XX_\aa,\bb)$ under condition (\ref{convenzione}). Then the $f$-process $(\P^n,\aa)\leftrightsquigarrow(\XX_\aa,\bb)$ will be shown to be calibrated if and only if conditions (a) and (b) are satisfied: that is, assuming (\ref{convenzione}), conditions (a) and (b) are equivalent to conditions (\ref{min}) in Theorem~\ref{thm:Deltatriviale}.

  $\P^n$ is a smooth and complete toric variety whose Picard group
  $$\Pic(\P^n)\cong\Cl(\P^n)\cong\Z\cdot h$$
  is generated by the hyperplane class $h=[D_1]=\cdots =[D_{n+1}]$, associated with the torus invariant prime divisors generating $\Weil(\P^n)\cong\bigoplus_{i=1}^{n+1}\Z\cdot D_i$. In particular $h$ is a very ample class, so giving that every strictly effective divisor is necessarily very ample, that is, for every ftv  $(\P^n,D_{\aa})$, $D_\aa$ ia very ample divisor. Recalling Proposition~\ref{prop:gg}~(2) and relation (\ref{mI}) in Proposition~\ref{prop:risoluzioni}, the associated lattice polytope $\D_\aa=\D_{D_\aa}$ is given by
  \begin{eqnarray}\label{DeltaaConv}
   \nonumber
    \D_\aa&=&\conv\left(\left\{-((V^{\{i\}})^T)^{-1}\cdot\aa^{\{i\}}\,|\,i=1,\ldots,n+1\right\}
    \right)\\
      &=&\conv\left(
                          \begin{array}{c}
                            |\aa|- a_1 \\
                            -a_2 \\
                            -a_3\\
                            \vdots \\
                            -a_n \\
                          \end{array}
                                        \begin{array}{c}
                                          -a_1 \\
                                           |\aa|- a_2\\
                                          -a_3\\
                                          \vdots \\
                                          -a_n \\
                                        \end{array}
                                      \cdots
                                        \begin{array}{c}
                                          -a_1 \\
                                          -a_2\\
                                          \vdots \\
                                          -a_{n-1} \\
                                          |\aa|- a_n\\
                                        \end{array}
                                        \begin{array}{c}
                                          -a_1 \\
                                          -a_2\\
                                          \vdots \\
                                          -a_{n-1} \\
                                          -a_n\\
                                        \end{array}
                                      \right)
  \end{eqnarray}
  where we set $|\aa|:=\sum_{i=1}^{n+1}a_i$\,.
  Then the associated reduced $F$-matrix $\L_\aa$ is
  \begin{equation}\label{Lambda_a}
    \L_\aa=\left(
                          \begin{array}{c}
                            {(|\aa|- a_1)/ d_1} \\
                            -{a_2/ d_1} \\
                            \vdots \\
                            -{a_n/ d_1} \\
                          \end{array}
                                        \begin{array}{c}
                                          \\
                                          \cdots\\
                                          \cdots\\
                                           \\
                                        \end{array}
                                        \begin{array}{c}
                                          -a_1/ d_n \\
                                          \vdots \\
                                          -a_{n-1}/ d_n \\
                                          (|\aa|- a_n)/ d_n\\
                                        \end{array}
                                        \begin{array}{c}
                                          -a_1/ d_{n+1} \\
                                          -a_2/ d_{n+1}\\
                                          \vdots \\
                                          -a_n/ d_{n+1}\\
                                        \end{array}
  \right)
  \end{equation}
  so giving
  \begin{equation}\label{LamdaV}
    \L_\aa^T\cdot V=\left(
                          \begin{array}{c}
                            (|\aa|- a_1)/ d_1 \\
                            -{a_1/ d_2} \\
                            \vdots \\
                            -a_{1}/ d_{n}\\
                            -a_{1}/ d_{n+1} \\
                          \end{array}
                                        \begin{array}{c}
                                          -a_2/ d_{1} \\
                                          (|\aa|- a_2)/ d_{2}\\
                                          -a_2/d_3\\
                                          \vdots \\
                                          -a_{2}/ d_{n+1} \\
                                        \end{array}
                                        \begin{array}{c}
                                          \cdots\\
                                          \cdots\\
                                          \cdots\\
                                          \cdots \\
                                          \cdots \\
                                        \end{array}
                                        \begin{array}{c}
                                          -a_{n}/ d_{1} \\
                                          \vdots \\
                                          -a_{n}/d_{n-1}\\
                                          (|\aa|- a_n)/ d_{n}\\
                                          -a_{n}/d_{n+1}\\
                                        \end{array}
                                        \begin{array}{c}
                                          -a_{n+1}/ d_{1} \\
                                          -a_{n+1}/d_2\\
                                          \vdots \\
                                          -a_{n+1}/d_n\\
                                          (|\aa|- a_{n+1})/ d_{n+1}\\
                                        \end{array}\right)
  \end{equation}
 Recalling that $a_1\leq \cdots \leq a_{n+1}$, there follows
 $$\bb=\left(
                                                                          \begin{array}{c}
                                                                            a_{n+1}/d_1 \\
                                                                            \vdots\ \\
                                                                            a_{n+1}/d_n \\
                                                                            a_n/d_{n+1} \\
                                                                          \end{array}
                                                                        \right)
 $$
 Moreover
\begin{equation*}
   \left(
      \begin{array}{ccc}
        d_1a_1 & \cdots & d_{n+1}a_{n+1} \\
      \end{array}
    \right)\cdot \L_\aa^T= \0
 \end{equation*}
 meaning that the reduced weight vector $\q$ of $(d_1a_1,\,\cdots\,,d_{n+1}a_{n+1})$ is a \emph{weight vector} of $\XX_\aa$, in the sense explained in \S\ref{ssez:coni&div}, that is, a representative matrix of the class morphism $d$ in the  short exact sequence (\ref{complete deg sequence}). Hence $\XX_\aa$ is a suitable quotient of the weighted projective space (WPS) $\P(\q)$ by the action of a finite abelian group $G_\aa$\,. The action of $G_\aa$  on $\P(\q)$ is described by item 6 in \cite[Thm.~3.2]{RT-Erratum}, so giving items from (2) to (5) in Algorithm~\ref{algoritmoG}. The isomorphism type of $G_\aa$ can be determined by in item (1) of Algorithm~\ref{algoritmoG}, that is, by looking for a fan matrix $\widetilde{\L}$ of $\P(\q)$ and a switching matrix $\b=\diag\left(\1_{n-s},\tau_1,\ldots\tau_s\right)$, such that $A\cdot\L_\aa=\b\cdot\widetilde{\L}$, for some $A\in\GL_n(\Z)$. Then
 \begin{equation*}
   G_\aa\cong\bigoplus_{i=1}^s\Z/\tau_i\Z
 \end{equation*}
 with $\tau_1|\tau_2|\cdots|\tau_s$. In particular $|G_\aa|=\prod_{i=1}^s\tau_i=\det\b$. Then, to prove (\ref{Ga-order}), notice that, on the one hand Binnet theorem gives
 \begin{equation*}
   \forall\,i=1,\ldots,n+1\quad \left|\det\left(\L_\aa^{\{i\}}\right)\right|=\left|\det\widetilde{\L}^{\{i\}}\right|\cdot\det\b=q_i\det\b
 \end{equation*}
 On the other hand, we claim that, under condition (\ref{convenzione}),
 \begin{equation}\label{pesi}
   \forall\,i=1,\ldots,n+1\quad \left|\det\left(\L_\aa^{\{i\}}\right)\right|=q_i\,|\aa|^{n-1}
 \end{equation}
 so giving $\det\b=|\aa|^{n-1}$ and then (\ref{Ga-order}). In fact, the reduction $\q=(q_1\,\cdots\,q_{n+1})$ is obtained by setting
\begin{equation*}
  q_i:={d_ia_i\over\lcm(\{\d_j\,|\,j\neq i\})}\quad \text{where}\quad \d_j:=\gcd(\{d_ka_k\,|\,k\neq j\})
\end{equation*}
Moreover, condition (\ref{convenzione}) implies that
\begin{equation}\label{lcm=prod}
  \forall\,i=1,\ldots,n+1\quad \lcm(\{\d_j\,|\,j\neq i\})=\prod_{j=1}^{n+1}d_j
\end{equation}
Notice that (\ref{lcm=prod}) implies (\ref{pesi}), as
 \begin{equation*}
   \left|\det\left(\L_\aa^{\{i\}}\right)\right|= {a_i|\aa|^{n-1}\over\prod_{j\neq i}d_j}={d_ia_i|\aa|^{n-1}\over\prod_{j=1}^{n+1}d_j}=q_i\,|\aa|^{n-1}
 \end{equation*}
To show (\ref{lcm=prod}), notice that, for any $i=1,\ldots,n+1$,
\begin{eqnarray*}
  d_i=\gcd(\{a_k\,|\,k\neq i\})\ \Longrightarrow\ \forall\,k\neq i\quad d_i|d_ka_k&\Longrightarrow& d_i|\d_i\\
  \forall\,j,k\neq i\quad d_j|d_ka_k\ \text{as}\ \left\{\begin{array}{cc}
                                                   d_j|d_j & \text{for $k=j$} \\
                                                   d_j|a_k & \text{for $k\neq j$}
                                                 \end{array}\right.&\Longrightarrow& d_j|\d_i\\
                                                 &\Longrightarrow&\lcm(d_1,\ldots,d_{n+1})| \d_i
\end{eqnarray*}
Recall that $\gcd(a_1,\ldots,a_{n+1})=1$ implies that $\gcd(d_j,d_k)=1$, for any $j\neq k$ \cite[Prop.~3]{RT-wps}. Therefore $\lcm(d_1,\ldots,d_{n+1})=\prod_{j=1}^{n+1}d_j$, so giving that
\begin{equation*}
  \forall\,i=1,\ldots,n+1\quad \prod_{j=1}^{n+1}d_j\,|\,\d_i
\end{equation*}
On the other hand, $\d_i=\gcd(\{d_ka_k\,|\,k\neq i\})$. Then $\d_i|d_ka_k$, for any $k\neq i$\,. Recall that $\gcd(d_k,a_k)=1$ \cite[Prop.~3]{RT-wps}. Hence, for any $i=1,\ldots,n+1$,
\begin{equation*}
  \forall\,k\neq i\quad \d_i|d_ka_k\ \Longrightarrow\ \left\{\begin{array}{cccc}
                              \exists\,k:& \d_i|d_k&\Longrightarrow&\d_i|\prod_{j=1}^{n+1}d_j\\
                              \forall\,k\neq i& \d_i|a_k&\Longrightarrow&\d_i|\gcd(\{a_k\,|\,k\neq i\})=d_i\\
                              &&\Longrightarrow& \d_i  |\prod_{j=1}^{n+1}d_j
                            \end{array}\right.
\end{equation*}
\begin{equation*}
  \Longrightarrow\hskip1.3truecm \d_i|\prod_{j=1}^{n+1}d_j\hskip2.5truecm
\end{equation*}
In conclusion, $\d_i=\prod_{j=1}^{n+1}d_j$ for any $i=1,\ldots,n+1$. Then (\ref{lcm=prod}) immediately follows.

Moreover, notice that $\q$ is also the reduced vector of $\aa$. In fact
\begin{equation*}
  \forall\,i\quad q_i={d_ia_i\over\lcm(\{\d_j\,|\,j\neq i\})}={d_ia_i\over\prod_{j=1}^{n+1}d_j}={a_i\over \prod_{j\neq i}d_j}={a_i\over \lcm(\{d_j\,|\,j\neq i\})}
\end{equation*}
Notice that if $\aa$ is already a reduced weight vector, then
 $$d_1=\cdots=d_{n+1}=1\ \Longrightarrow\ \q=\aa\ ,\quad \bb= \left(
                                                                          \begin{array}{c}
                                                                            a_{n+1} \\
                                                                            \vdots\ \\
                                                                            a_{n+1} \\
                                                                            a_n \\
                                                                          \end{array}
                                                                        \right)
                                                                        $$
Therefore: $(\XX_\aa:=\P(\q)/G_\aa,D'_\bb$) is the $f$-dual ftv of $(\P^n,D_\aa)$.

 We are now going to considering the $f$-process associated with $(\P^n,D_\aa)$\,. By the definition of $\D_\bb$ given in (\ref{rel1}), and noticing that $\XX_\aa$ is the toric variety associated with the fan $\Si_{\D_\aa}$, which turns out to be the unique one in $\SF(\L_\aa)$, we get
\begin{equation}\label{Delta_b}
  \D_\bb=\{\n\in N_\R\,|\,\L_\aa^T\cdot\n \geq -\bb\}=\conv(\{\n_i\in N_\R\,|\,1\leq i\leq n+1\}
\end{equation}
where $\n_i=-\left(\left(\L_\aa^{\{i\}}\right)^T\right)^{-1}\cdot\bb^{\{i\}}$, so giving
\begin{eqnarray*}
\forall\,i=1,\ldots,n\quad\n_i&=&-{a_{n+1}-a_n\over |\aa|}\,\mathbf{1}+{a_{n+1}\over a_i}\left(1-{a_{n+1}-a_n\over |\aa|}\right)\,\e_i \\
 \n_{n+1}     &=& -\mathbf{1}\\
\end{eqnarray*}
Notice that, in this expression of $\D_\bb$\,
\begin{itemize}
  \item the dependence on $d_1,\ldots,d_{n+1}$ completely disappeared,
  \item since $a_n\leq a_{n+1}$\,, it follows that
  \begin{equation}\label{disuguaglianze}
    0\leq{a_{n+1}-a_n\over|\aa|}< 1\quad\text{and}\quad 0<{a_{n+1}\over a_i}\left(1-{a_{n+1}-a_n\over |\aa|}\right)\leq {a_{n+1}\over a_i}
  \end{equation}
\end{itemize}
Inequalities in (\ref{disuguaglianze}) imply that, for every $i=1,\ldots,n$, the $j$-th entry of $\n_i$ has to satisfy the relations
\begin{equation}\label{componenti}
  \forall\,j\neq i\quad -1<n_{j,i}\leq 0\quad\text{and}\quad \left\{\begin{array}{cc}
                                                                      n_{i,i} =1 & \text{if $a_i=a_{n+1}$} \\
                                                                       n_{i,i}>1 & \text{if $a_i<a_{n+1}$}
                                                                     \end{array}\right.
\end{equation}
where the inequality $n_{i,i}>1$ is obtained as follows:
\begin{equation*}
  n_{i,i}={a_{n+1}\over a_i}-\left(1+{a_{n+1}\over a_i}\right){a_{n+1}-a_n\over |\aa|}={a_{n+1}|\aa|-(a_i+a_{n+1})(a_{n+1}-a_n)\over a_i|\aa|}
\end{equation*}
Therefore,
\begin{equation*}
  n_{i,i}>1\ \Longleftrightarrow\ (a_{n+1}-a_i)|\aa|>(a_i+a_{n+1})(a_{n+1}-a_n)
\end{equation*}
and the latter follows immediately by hypothesis on $\aa$ and the first condition in (\ref{convenzione}).

\noindent Calling $\NN':=\D_\bb\cap N_{\R}^\geq$, where $N_{\R}^\geq$ represents the positive orthant in the chosen identification $N_{\R}\cong \R^n$\,, (\ref{componenti}) give that
\begin{equation*}
  \NN'=\conv\left(\0,\e_1,{a_n\over a_{n-1}}\e_2,\ldots,{a_n\over a_1}\e_n\right)
\end{equation*}
as one can check by intersecting the hyperplane passing through $\n_1,\ldots,\n_n$ with coordinate axes.
Moreover,
\begin{equation*}
  \D_\bb\cap N=\{-\mathbf{1}\}\cup(\NN'\cap N)
\end{equation*}
and
\begin{equation}\label{[Deltab]}
  \D(\XX_\aa,\bb)=[\D_\bb]=\conv\left(\left\{\e_1,\left[{a_n\over a_{n-1}}\right]\e_2,\ldots,\left[{a_n\over a_{1}}\right]\e_n,-\mathbf{1}\right\}\right)\ \Longleftrightarrow\ \L_\bb=V
\end{equation}
where the last equality has to be understood up to a possible permutation of columns. This means that the first condition (\ref{min}) in Theorem~\ref{thm:Deltatriviale} is equivalent to condition (a) in the statement. Moreover, recalling expression (\ref{LamdaV}) of the transposed matrix of $\L_\bb^T\cdot\L_\aa=V^T\cdot\L_\aa$, the second condition in (\ref{min}) can be attained if and only if condition (b) in the statement is assumed, that is, if and only if at least two of $d_i$'s equal 1. Then Theorem~\ref{thm:Deltatriviale} ensures that the $f$-process associated with $(\P^n,D_\aa)$ is calibrated if and only if conditions (a) and (b) hold.
\end{proof}

\begin{corollary}\label{cor:ipersuperfici}
  Let $Y_d\subseteq\P^n$ be a projective hypersurface of degree $d\geq n+1$. Then there always exists a framing $D_{\aa_0}$ of $\P^n$ such that $Y_d\sim D_{\aa_0}$ and the $f$-process associated with $(\P^n,D_{\aa_0})$ is calibrated.
\end{corollary}

\begin{proof}
 It suffices choosing ${\aa_0}=(\underbrace{1,\ldots,1}_{n\ \text{times}},\d:=d-n)=(\1,\d)$. It clearly satisfies conditions (\ref{convenzione}) and (b) of Theorem~\ref{thm:dualita}. Moreover $$\NN'=\conv\left(\0,\e_1,\ldots,\e_n\right)=\conv(\NN'\cap N)$$
  so giving condition (a), too. Then, thesis follows by theorems~\ref{thm:Deltatriviale} and \ref{thm:dualita}.
\end{proof}

\begin{remark}\label{rem:s-ample} Recalling Remark~\ref{rem:semimple},
  consider the framing $\aa_0=(\1_n,\d)$ of $\P^n$ introduced in the previous Corollary~\ref{cor:ipersuperfici}. The dual ftv is then given by
  $$\left(\XX_{\aa_0},\bb_0\right)=\left(\P(\1_n,\d)/\left(\Z/d\Z\right)^{n-1}\ ,\ (\underbrace{\d,\ldots,\d}_{\text{$n$ times}},1)\right)$$
  (notice that $G_{(\1,\d)}\cong\left(\Z/d\Z\right)^{n-1}$, by the following Lemma~\ref{lem:Ga}). By (\ref{Delta_b}), the polytope $d\D_{\bb_0}$ is a lattice polytope, convex hull of $n+1$ lattice points associated to the maximal cones of the fan $\Si_{\aa_0}$ of $\XX_{\aa_0}$, that is, $\cO_{\XX_{\aa_0}}(d D'_{\bb_0})$ is a globally generated line bundle and $D'_{\bb_0}$ is semi-ample, by Proposition~\ref{prop:gg}.\\
  Since $\XX_{\aa_0}$ has Picard number 1, this is enough to ensure that $dD'_{\bb_0}$ is an ample divisor of $\XX_{\aa_0}$.\\
  Moreover, calling $\pi:\P(\1_n,\d)\twoheadrightarrow\XX_{\aa_0}$ the canonical quotient associated with the $\left(\Z/d\Z\right)^{n-1}$-action, the pull-back $\pi^*(D'_{\bb_0})=\sum_{j}b_j\pi^*(D'_j)$ turns out to be the generator of $\Pic(\P(\1_n,\d))\cong\Z$ and a very ample divisor of the universal 1-covering $\P(\1_n,\d)$ of $\XX_{\aa_0}$, as guaranteed by \cite[Prop.~8]{RT-wps}.
\end{remark}

\section{Mirror partners of hypersurfaces of degree $d\geq n+1$ in $\P^n$}\label{sez:ipersuperfici}

Degree $d$ hypersurfaces in $\P^n$ are parameterized by the projective space $$\P\left(H^0(\P^n,\cO_{\P^n}(d)\right)$$
The action of $\P\GL(n+1)$ on $\P^n$ extends naturally to an action on the parameter space $\P\left(H^0(\cO_{\P^n}(d)\right)$. Recalling (\ref{mYY}), define the \emph{number of complex moduli} of a generic (smooth) hypersurface $Y=Y_d\subset\P^n$,  of degree $d$, to be the following one
\begin{equation}\label{m_d^n}
  m_d^n:=\dim \P\left(H^0(\cO_{\P^n}(d)\right) - \dim \P\GL(n+1) ={n+d\choose d}-(n+1)^2
\end{equation}
which is actually the dimension of the moduli space $\mathcal{M}^n_d$ of degree $d$ hypersurfaces in $\P^n$, well defined after Mumford's GIT \cite{Mumford}, as $\P\GL(n+1)$ is a reductive group.\\
On the other hand, if $n\geq 4$, Weak Lefschetz Theorem implies that the \emph{Picard number} of $Y$ is given by
\begin{equation*}
  k_d^n:=h^{1,1}(Y)=b_2(Y)=b_2(\P^n)=1
\end{equation*}
which is also called the \emph{number of \ka moduli} of $Y$, being $k_d^n$ the dimension of the complexified \ka cone of $Y$ \cite[\S6.2]{CoxKatz}.

\begin{remark}\label{rem:combinatorica}
From the combinatorial point of view, consider the framing of $\P^n$ given in Corollary~\ref{cor:ipersuperfici}, that is $D_{\aa_0}\sim Y_d$ with
\begin{equation*}
  {\aa_0}=(\underbrace{1,\ldots,1}_{n\ \text{times}},d-n)
\end{equation*}
Then
\begin{equation}\label{Deltaa}
  \D_{\aa_0}=\conv(\L_{\aa_0})\quad\text{with}\quad \L_{\aa_0}=\left(
                          \begin{array}{c}
                            d-1 \\
                            -1 \\
                            \vdots \\
                            -1 \\
                          \end{array}
                                        \begin{array}{c}
                                          \\
                                          \cdots\\
                                          \cdots\\
                                           \\
                                        \end{array}
                                        \begin{array}{c}
                                          -1 \\
                                          \vdots \\
                                          -1 \\
                                          d-1\\
                                        \end{array}
                                        \begin{array}{c}
                                          -1 \\
                                          -1\\
                                          \vdots \\
                                          -1\\
                                        \end{array}\right)
\end{equation}
is the Newton polytope associated with the generic polynomial in $H^0(\P^n,\cO_{\P^n}(d))$. Recalling (\ref{aut}) and (\ref{h0}), one has
\begin{equation*}
  m_d^n= l(\D_{\aa_0})-1-n-\sum_{\Theta<^1\D_\1} l^*(\Theta)
\end{equation*}
as the anti-canonical polytope $\D_{-K_{\P^n}}$ is given by the following sub-polytope of $\D_{\aa_0}$
\begin{equation*}
  \D_{-K_{\P^n}}=\D_\1=\conv\left(\begin{array}{c}
                            d-2 \\
                            -1 \\
                            \vdots \\
                            -1 \\
                          \end{array}
                                        \begin{array}{c}
                                          \\
                                          \cdots\\
                                          \cdots\\
                                           \\
                                        \end{array}
                                        \begin{array}{c}
                                          -1 \\
                                          \vdots \\
                                          -1 \\
                                          d-2\\
                                        \end{array}
                                        \begin{array}{c}
                                          -1 \\
                                          -1\\
                                          \vdots \\
                                          -1\\
                                        \end{array}\right)
\end{equation*}
On the other hand, by Remark~\ref{rem:reflexive}, $\P^n=\P_{\D_\1}\cong\XX_\NN$ with $\NN=\D_\1^*=\conv(V)$ and $V$ is the fan matrix given in (\ref{V}). Then \cite[Prop.~4.4.1]{Batyrev94} gives
\begin{equation}\label{k}
  k_d^n=h^{1,1}(\P^n)=l(\NN)-l^*(\NN)-n = l(\conv(V))-1-n =1
\end{equation}
\end{remark}

\subsection{A-side mirroring} The $f$-dual ftv of $(\P^n,D_{\aa_0})$, as given by Theorem~\ref{thm:dualita}, is $(\XX_{\aa_0},D'_{\bb_0})$ with
\begin{eqnarray*}
    \XX_{\aa_0} &\cong& \P(1,\ldots,1,d-n)/G_{\aa_0}\\
    D'_{\bb_0} &=& \sum_{i=1}^{n+1}b_iD'_i\quad\text{where}\ b_i=\left\{\begin{array}{cc}
                                                               d-n & \text{for}\ i\leq n \\
                                                               1 & \text{for}\ i= n+1
                                                             \end{array}\right.
  \end{eqnarray*}

\begin{lemma}\label{lem:Ga}
  $G_{\aa_0}\cong\left(\Z/d\Z\right)^{n-1}$ and its action on $\P({\aa_0})$ can be written as follows
  \begin{eqnarray*}
     &\xymatrix{\left(\Z/d\Z\right)^{n-1}\times\P(1,\ldots,1,d-n)\ar[rr]^-{\Ga}&&\P(1,\ldots,1,d-n)}\hskip3.8truecm&  \\
     &\xymatrix{\left((\overline{\ve}_1,\ldots,\overline{\ve}_{n-1}),[x_1:\ldots :x_{n+1}]\right)\ar@{|->}[r]&\left[\mu_1x_1:\cdots:\mu_{n-1}x_{n-1}:x_{n}:\left(\prod_{j=1}^{n-1}\mu_j\right)^{-1}x_{n+1} \right] }&
  \end{eqnarray*}
  where $\mu_j:=\exp\left({2\pi i\over d}\ve_j\right)$. It can then be represented by the following torsion matrix
  \begin{equation}\label{azione}
  \Ga=\left(
                      \begin{array}{ccc}
                        \overline{I}_{n-1} & \overline{\0}_{n-1} & (d-1)\cdot\overline{\mathbf{1}}_{n-1} \\
                      \end{array}
                    \right)\in\M(n-1,n+1;\Z/d\Z)
  \end{equation}
\end{lemma}

\begin{proof} First of all, we need to compute the torsion coefficients $\tau_1|\cdots|\tau_s$\,. At this purpose we determine a fan matrix $\widetilde{\L}_{\aa_0}$ of the covering wps $\P({\aa_0})$.

\noindent Since ${\aa_0}=(1,\ldots,1,d-n)$, we can choose
\begin{equation}\label{L-hat}
  \widetilde{\L}_{\aa_0}=\left(
                \begin{array}{ccc}
                  I_{n-1} & -\1_{n-1} & \0_{n-1} \\
                  \0_{n-1}^T & d-n & -1 \\
                \end{array}
              \right)\in \M(n, n+1, \Z)
\end{equation}
as ${\aa_0}\cdot \widetilde{\L}_{\aa_0}^T=\0_n^T$\,. As a second step we have to determine  a matrix
\begin{equation*}
  B\in\GL(n,\Q)\cap\M(n,\Z):\quad B\cdot \widetilde{\L}_{\aa_0}=\L_{\aa_0}
\end{equation*}
where $\L_{\aa_0}$ is the fan matrix of $\XX_{\aa_0}$ presented in (\ref{Deltaa}). Such an integer matrix $B$ exists by \cite[Prop.~3.1\,(3)]{RT-LA&GD} (see also \cite[Rem.~2.4]{RT-QUOT}) and is given by
\begin{equation}\label{betamat}
  B=\left(
             \begin{array}{ccccc}
               d-1&-1 & \cdots & -1 & 1 \\
               -1 & d-1 &\cdots & -1 & 1 \\
                \vdots  &         & \ddots &  \vdots  & \vdots       \\
              -1 & \ldots & -1 & d-1 & 1 \\
               -1 & \ldots &-1 & -1 & 1 \\
             \end{array}
           \right)
\end{equation}
Then, torsion coefficients are given by entries different than 1 in the diagonal Smith Form $\b$ of $B$, namely given by
\begin{equation}\label{smith}
  \b=A\cdot B\cdot C =\diag(1,\underbrace{d,\ldots,d}_{\text{$n-1$ times}})
\end{equation}
for suitable matrices $A,C\in\GL_n(\Z)$.
That is, $s=n-1$ and $\tau_1=\cdots=\tau_{n-1}=d$\,. In particular, $G_{\aa_0}\cong (\Z/d\Z)^{n-1}$.\\
 A torsion matrix $\Ga$ is a representative matrix of the torsion part of the class morphism from $\Weil(\XX_{\aa_0})$ to $\Cl(\XX_{\aa_0})$ and it is characterized by properties from (i) to (iv) in the proof of item (6)  in \cite[Thm.~3.2]{RT-Erratum}, that is:
\begin{itemize}
                                               \item[(i)] $\Ga=(\g_{kj})$ with $\g_{kj}\in \Z/d\Z$\,,
                                               \item[(ii)] $\Ga\cdot(\,^1U_{\aa_0})^T\equiv\0_{n-1} \mod d$\,, being $U_{\aa_0}\in\GL_{n+1}(\Z)$ a matrix switching the transposed weight vector $\aa_0^T$ in Hermite normal form,
                                               \item[(iii)] $\Ga\cdot\L_{\aa_0}^T \equiv\0_{n-1,n} \mod d$\,,
                                               \item[(iv)] $\Ga\cdot(\, _{n-1}(C^{-1}\cdot\widetilde{\L}_{\aa_0}))^T\equiv I_{n-1} \mod d$, where $C$ is given in (\ref{smith}), since the bottom $n-1$ rows of $C^{-1}\cdot\widetilde{\L}_{\aa_0}$ are sent by $\Ga$ to a set of generators of $$\Tors(\Cl(\XX_{\aa_0}))\cong(\Z/dZ)^{n-1}$$
                                             \end{itemize}
Assume $\Ga$ is given as in (\ref{azione}). Then (i) and (iii) are clear and (ii) follows by choosing
\begin{equation*}
  \,^1U_{\aa_0}=\left(
                  \begin{array}{ccccc}
                    0 & \cdots & 0 & 1 & 0 \\
                  \end{array}
                \right)
\end{equation*}
 Finally, condition (iv) is verified up to a basis change in $\Weil(\P(\aa_0))$. In fact, suppressing the $n$-th column from $\widetilde{\L}_{\aa_0}$, by \cite[Cor.~3.3]{RT-LA&GD} it follows that
  \begin{equation*}
1=\det\left(\L_{\aa_0}^{\{n\}}\right)\ \Longrightarrow\ \L_{\aa_0}^{\{n\}}\in\GL_n(\Z)
  \end{equation*}
  being 1 the $n$-th entry in $\aa_0$. Then (iv) is satisfied by setting
  \begin{equation*}
    C=\L_{\aa_0}^{\{n\}}\cdot\left(
                               \begin{array}{cc}
                                 \0_{n-1} & 1 \\
                                 I_{n-1} & 0 \\
                               \end{array}
                             \right)^{-1}
  \end{equation*}
  \end{proof}

Let $\D_{\bb_0}:=\D_{D'_{\bb_0}}$ be the polytope associated with $D'_{\bb_0}$. Then $\D(\XX_{\aa_0},{\bb_0})=[\D_{\bb_0}]$ is the Newton polytope of the generic section in $H^0(\XX_{\aa_0},\cO_{\XX_{\aa_0}}(D'_{\bb_0}))$, meaning that
\begin{equation*}
  h^0(\XX_{\aa_0},\cO_{\XX_{\aa_0}}(D'_{\bb_0}))=l(\D_{\bb_0})=l(\D(\XX_{\aa_0},{\bb_0}))
\end{equation*}
Moreover, $D'_{\bb_0}$ turns out to be a semi-ample divisor of $\XX_{\aa_0}$, as observed in Remark~\ref{rem:s-ample}.

\begin{theorem}\label{thm:m*=k} The family of hypersurfaces $Y^\vee\subseteq\XX_{\aa_0}$, obtained as zero-locus of sections in $H^0(\XX_{\aa_0},\cO_{\XX_{\aa_0}}(D'_{\bb_0}))$, depends on a unique complex modulus, that is, $m_{Y^\vee}=1$. If $n\geq 4$ then $m_{Y^\vee}$ equals the number $k_d^n$ of \ka moduli of projective hypersurfaces $Y_d\subseteq\P^n$ of degree $d$, that is
  \begin{equation*}
    m_{Y^\vee}=k_d^n=1
  \end{equation*}
  By Definition~\ref{def:A,B-mirror}, this means that  $Y^\vee$ is an $A$-mirror of $Y$.
\end{theorem}

\begin{proof}
Since ${\bb_0}=(d-n,\ldots,d-n,1)$, relation (\ref{Delta_b}) gives that
\begin{equation*}
  \D_{\bb_0}=\conv\left(\begin{array}{c}
                            {nd+1-n^2\over d} \\
                            -{d-n-1\over d} \\
                            \vdots \\
                            -{d-n-1\over d} \\
                          \end{array}
                                        \begin{array}{c}
                                          \\
                                          \cdots\\
                                          \cdots\\
                                           \\
                                        \end{array}
                                        \begin{array}{c}
                                          -{d-n-1\over d} \\
                                          \vdots \\
                                          -{d-n-1\over d} \\
                                          {nd+1-n^2\over d}\\
                                        \end{array}
                                        \begin{array}{c}
                                          -1 \\
                                          -1\\
                                          \vdots \\
                                          -1\\
                                        \end{array}\right)
\end{equation*}
One can then directly check that
\begin{equation*}
  \D(\XX_{\aa_0},{\bb_0})=[\D_{\bb_0}]=\conv(V)=:\NN
\end{equation*}
as already shown by relation (\ref{[Deltab]}). Therefore $\0\in\Int([\D_{\bb_0}])$, meaning that $k_1=1$ in the first item of \S\ref{ssez:Deltaprocesso}. In particular, one gets
\begin{equation}\label{h0}
  h^0(\XX_{\aa_0},\cO_{\XX_{\aa_0}}(D'_{\bb_0}))=l(\D_{\bb_0})=l(\NN)=n+2
\end{equation}
and a generic section $f\in H^0(\XX_{\aa_0},\cO_{\XX_{\aa_0}}(D'_{\bb_0}))$ can be written as follows
\begin{equation}\label{f}
  f=\left(\sum_{i=1}^n c_i\, x_i^d\cdot \prod_{j=1}^n x_j^{d-n-1}\right)+c_{n+1}\,x_{n+1}^{n+1}+c_{n+2}\,\left(\prod_{j=1}^n x_j^{d-n}\right)\cdot x_{n+1}
\end{equation}
in the Cox ring $\C[x_1,\ldots,x_{n+1}]$ of $\XX_{\aa_0}$\,. Recall now that
$\XX_{\aa_0}\cong\P({\aa_0})/G_{\aa_0}$, with ${\aa_0}=(1,\ldots,1,d-n)$, and consider the automorphism of $\P({\aa_0})$
represented by the diagonal matrix $\d=\diag(\g_1,\ldots,\g_n,\g_{n+1})$, where $\g_1,\ldots,\g_{n+1}$ are solutions  of the following equations
\begin{eqnarray}\label{automorfismo}
\nonumber
  \g_k\hskip1.8truecm=&1&\text{if $c_k=0$, for $1\leq k\leq n+1$} \\
  \g_i^d\left(\prod_{j=1}^n\g_j\right)^{d-n-1} =& c_i &\text{for $1\leq i\leq n$ with $c_i\neq 0$}\\
  \nonumber
  \g_{n+1}^{n+1}\hskip1.6truecm =& c_{n+1} &\text{if $c_{n+1}\neq 0$}
\end{eqnarray}
By the previous Lemma~\ref{lem:Ga}, the action of $G_{\aa_0}$ can be assumed diagonal, meaning that $\d$ commutes with such an action, giving rise to an automorphism of $\XX_{\aa_0}$ making $f$ equivalent to the section
\begin{equation}\label{mirror}
  f'= \left(\sum_{i=1}^n \epsilon_i\, x_i^d\cdot \prod_{j=1}^n x_j^{d-n-1}\right)+\epsilon_{n+1}\,x_{n+1}^{n+1}+\psi\,\left(\prod_{j=1}^n x_j^{d-n}\right)\cdot x_{n+1}
\end{equation}
where $\epsilon_k=\left\{\begin{array}{cc}
                           0 & \text{if $c_k=0$} \\
                           1 & \text{otherwise}
                         \end{array}
\right.$\,.

\noindent Then $\psi\in\C$ turns out to be the unique complex modulus of the family of hypersurfaces $Y^\vee\subseteq\XX_{\aa_0}$ of degree $[D'_{\bb_0}]\in\Cl(\XX_{\aa_0})$\,.
\end{proof}

\begin{remark}\label{rem:moduli} Notice that the general hypersurface $Y^\vee\in|D'_{\bb_0}|$ is not quasi-smooth. Then results by Batyrev and Cox \cite{BC} and Bunnet \cite{Bunnet} cannot be applied to gua\-ran\-tee a good definition of a moduli space $\mathcal{M}_{Y^\vee}$. By the way, observe that the computation performed in Theorem~\ref{thm:m*=k} is consistent with the definition of $m_{Y^\vee}$ given in $(\ref{mYY})$. In fact
\begin{eqnarray*}
    \dim\P\left(H^0(\XX_{\aa_0},\cO_{\XX_{\aa_0}}(D'_{\bb_0}))\right)- \dim\Aut(\XX_{\aa_0})&=&l(\D_{\bb_0})-1-n\\
    &=& l(\NN)-1-n\,=\,1
  \end{eqnarray*}
  The last equality is obtained by recalling (\ref{h0}).  The former follows by (\ref{aut}), just observing that the anti-canonical polytope $\D_{-K_{\XX_{\aa_0}}}$ is given by the following sub-polytope of $\D_{\bb_0}$
  \begin{equation*}
    \D_{-K_{\XX_{\aa_0}}}=\conv\left(
  \begin{array}{cccc}
    \e_1 & \cdots & \e_n & -1/(d-n)\1 \\
  \end{array}
\right)
  \end{equation*}
  whose facets do not contain any lattice point in their relative interior.
\end{remark}

\subsection{According with the Hori-Vafa LG mirror model}\label{ssez:HoriVafa}
In their pivotal, and unpublished, paper \cite{Hori-Vafa}, Hori and Vafa proposed, from a physical point of view, Landau-Ginzburg (LG) mirror models of hypersurfaces and complete intersections in a complete toric variety. Their construction is consistent with the interpretation of Mirror Symmetry as T-duality. In particular, for the projective hypersurface of degree $d=n+1$ in $\P^{n+1}$, a suitable quotient of their LG mirror model still proposes the mirror construction previously given by Greene and Plesser \cite{GP}, for $n=3$, and Batyrev \cite{Batyrev94}, in the general case.

Namely, for the projective, degree $d$, hypersurface $Y=Y_d\subset\P^n$, the Hori-Vafa recipe proposes the LG mirror model $(\L_{d,\psi},w)$, where (see \cite[5.4]{Hori-Vafa} and notation introduced in \cite{KKOY}):
 \begin{itemize}
   \item $\L_{d,\psi}\cong(\C^*)^{n+1}$ is the choice of an irreducible component of the reducible torus hypersurface
   $$\L_{d}:=\left\{\prod_{i=1}^{n+1}x_i^d= \tau y^d\right\}\subset(\C^*)^{n+1}\times\C^*=(\C^*)^{n+2}$$
       being $\psi^{-d}=\tau=e^{-t}\in\C^*$, with $t$ the \ka volume of $Y_d$,
   \item $w_{d,\psi}:\L_{d,\psi}\longrightarrow\C$ is the holomorphic function defined by setting
       $$w_{d,\psi}=w_{d|\L_{d,\psi}}$$
       being $w_d:\C^{n+2}\longrightarrow\C$ defined by
       \begin{equation*}
         w_d(x_1,\ldots,x_{n+1},y)=\sum_{i=1}^{n+1}x_i^{d}+y\,\Longrightarrow\,
         w_{d,\psi}(\x)=\sum_{i=1}^{n+1}x_i^{d}+\psi \prod_{i=1}^{n+1}x_i
       \end{equation*}
       and called the \emph{superpotential} of the LG model.
 \end{itemize}
When $d=n+1$, the superpotential $w_{n+1,\psi}$ turns out to be equivariant \wrt the $\C^*$-action defining $\P^n$ and invariant \wrt the action of $G_\1\cong(\Z/(n+1)\Z)^{n-1}$ described in Lemma~\ref{lem:Ga} and defining $\XX_\1=\P^n/G_\1$, so getting the following picture
 \begin{equation}\label{LGquoziente1}
   \xymatrix{ \{\0\}\ar@{^(->}[r]&\C&\L_{n+1,\psi}\cong(\C^*)^{n+1}\ar@{^(->}[r]\ar[l]_-{w_{n+1,\psi}}
   \ar@{->>}[d]_-{/(\C^*\times G_\1)}&\C^{n+1}\setminus\{\0\}\ar@{->>}[d]_-{/(\C^*\times G_{\1})}\\
   w_{n+1,\psi}^{-1}(0)/(\C^*\times G_\1)\ar[u]\ar@{^(->}[rr]&&\T\ar@{^(->}[r]&\XX_\1}
 \end{equation}
 where $\T$ is the acting torus on $\XX_\1$. Then the Batyrev's mirror $Y^\vee$ of $Y_{n+1}$ is precisely the closure
 \begin{equation*}
   Y^\vee=\overline{w_{n+1,\psi}^{-1}(0)/(\C^*\times G_\1)}\subset\overline{\T}=\XX_\1
 \end{equation*}
 induced by the open embedding $\T\hookrightarrow \XX_\1$.

 \begin{remark}\label{rem:LG-d} Unfortunately, for $d\geq n+2$ the Hori-Vafa LG mirror model does no more admit a similar compactification process, as the superpotential $w_{d,\psi}$ is no more quasi-homogeneous, although we know that \emph{a compact mirror model $Y^\vee_d$ of $Y_d$ should exist}, as defined in Definition~\ref{def:mirror}.
 \end{remark}

 \subsubsection{LG mirror model of the projective hypersurface of degree $d$}\label{sssez:HVinvariante} To bypass troubles observed in Remark~\ref{rem:LG-d}, replace the Hori-Vafa LG mirror model with the LG model $(\widetilde{\L}_{d,\psi},\widetilde{w}_{d,\psi})$ where
 \begin{itemize}
   \item $\widetilde{\L}_{d,\psi}\cong(\C^*)^{n+1}$ is the choice of an irreducible component of the reducible torus hypersurface
   $$\widetilde{\L}_{d}:=\left\{x_{n+1}^{n+1}\cdot\prod_{i=1}^{n}x_i^{(n+1)(d-n)}= \tau y^{n+1}\right\}\subset(\C^*)^{n+1}\times\C^*=(\C^*)^{n+2}$$
       being $\psi^{-(n+1)}=\tau=e^{-t}\in\C^*$, with $t$ the \ka volume of $Y_d$,
   \item  $\widetilde{w}_{d,\psi}:\widetilde{\L}_{d,\psi}\longrightarrow\C$ is the holomorphic function defined by setting
 \begin{equation*}
 \widetilde{w}_{d,\psi}=\widetilde{w}_{d|\L_{\psi}}
 \end{equation*}
       being $\widetilde{w}_d:\C^{n+2}\longrightarrow\C$ defined by
       \begin{eqnarray*}
         \widetilde{w}_d(x_1,\ldots,x_{n+1},y)&=&\left(\sum_{i=1}^n x_i^d\cdot \prod_{j=1}^n x_j^{d-n-1}\right)+x_{n+1}^{n+1}+y\\
         \Longrightarrow\quad
         \widetilde{w}_{d,\psi}(\x)&=&\left(\sum_{i=1}^n x_i^d\cdot \prod_{j=1}^n x_j^{d-n-1}\right)+x_{n+1}^{n+1}+\psi \,x_{n+1}\,\prod_{j=1}^n x_j^{d-n}
       \end{eqnarray*}
 \end{itemize}

In addition to what observed in Remark~\ref{rem:LG-d}, there is a further fact leading to consider the previous modification of the original Hori-Vafa proposal, that is, its consistency with the re-parameterization of the Givental LG mirror model presented in the next Theorem~\ref{thm:d<n+1}, for projective hypersurfaces of degree $d\le n$. This will be treated more in detail in the next Remark~\ref{rem:HVvsGivental}, to which the reader is referred. Here we want just underline that it seems comparable with an analogous modification proposed by Clarke, making the modified Hori-Vafa LG mirror model consistent with the Givental's one  \cite[\S7.2, Rem.~7.4]{Clarke}.

 \begin{conjecture}\label{conj:LGmirror}
   For $d\geq n+1$, a LG mirror model of the projective hypersurface $Y_d\subset\P^n$, of \ka modulus $t=-(n+1)\ln(\psi)$, is given by $((\C^*)^{n+1},\widetilde{w}_{d,\psi})$\,.
 \end{conjecture}

Following Kontsevich \cite{Kontsevich}, proving this Conjecture means showing that the derived categories of coherent sheaves, from the complex point of view, and the Fukaya category of lagrangian structures, from the symplectic point of view, are each other equi\-va\-lent on the two mirror partners involved (Homological Mirror Symmetry, HMS). This is a quite difficult topic. Here, just some evidences will be provided.

First of all, notice that, when $d=n+1$, the LG model $((\C^*)^{n+1},\widetilde{w}_{n+1,\psi})$ is precisely the Hori-Vafa LG model $((\C^*)^{n+1},w_{n+1,\psi})$.

As a second evidence, consider the fact that, under the weighted $\C^*$-action on $(\C^*)^{n+1}$, given by
\begin{equation}\label{azione4}
  \xymatrix{(\l,\x)\ar@{|->}[r]&(\l x_1,\ldots,\l x_n,\l^{d-n}x_{n+1})}
\end{equation}
the superpotential $\widetilde{w}_{d,\psi}$ is equivariant. In Hori-Vafa notation, this means that there is a \emph{gauged linear sigma model} associated with the LG model $((\C^*)^{n+1},\widetilde{w}_{d,\psi})$, whose gauge action is the weighted one presented in (\ref{azione4}). Moreover, $\widetilde{w}_{d,\psi}$ is also equivariant with respect to the action of $G_{\aa_0}\cong (\Z/d\Z)^{n-1}$ described in Lemma~\ref{lem:Ga} and defining $\XX_{\aa_0}=\P(\aa_0)/G_{\aa_0}$, recalling that $\aa_0=(1,\ldots,1,d-n)$. There is then an analogous picture generalizing (\ref{LGquoziente1}) as follows
\begin{equation}\label{LGquoziente}
   \xymatrix{ \{\0\}\ar@{^(->}[r]&\C&\widetilde{\L}_{d,\psi}\cong(\C^*)^{n+1}\ar@{^(->}[r]\ar[l]_-{\widetilde{w}_{d,\psi}}
   \ar@{->>}[d]_-{/(\C^*\times G_{\aa_0})}&\C^{n+1}\setminus\{\0\}\ar@{->>}[d]_-{/(\C^*\times G_{{\aa_0}})}\\
   \widetilde{w}_{d,\psi}^{-1}(0)/(\C^*\times G_{\aa_0})\ar[u]\ar@{^(->}[rr]&&\T\ar@{^(->}[r]&\XX_{\aa_0}}
 \end{equation}
 where $\T$ is the acting torus on $\XX_{\aa_0}$. Then the $f$-mirror $Y^\vee$ of $Y_{d}$, as proposed in Definition~\ref{def:mirror}, is precisely the closure
 \begin{equation*}
   Y^\vee=\overline{\widetilde{w}_{d,\psi}^{-1}(0)/(\C^*\times G_{\aa_0})}\subset\overline{\T}=\XX_{\aa_0}
 \end{equation*}
 induced by the open embedding $\T\hookrightarrow \XX_{\aa_0}$.

 As a final evidence, notice that the picture described by diagram (\ref{LGquoziente}) is strongly related with general LG/Hypersurface correspondence described in \S\ref{sssez:LG/Hyp} and its compactification given in \S\ref{ssez:KKP-compct}. In a sense, the latter turns out to be the quotient of $((\C^*)^{n+1},\widetilde{w}_{d,\psi})$ by the action of $\C^*\times G_{\aa_0}$ defining $\XX_{\aa_0}$ as a Cox quotient.

 \begin{remark}\label{rem:LGmirrors}
   Taking into account what just observed, relating the LG model here presented with those described in \S\ref{sssez:LG/Hyp} and \S\ref{ssez:KKP-compct}, one could argue that the LG model $(\T_{\aa_0},f^\vee_{\bb_0})$ would be a more appropriated LG mirror model for $Y_d\subset\P^n$ than the one proposed in Conjecture~\ref{conj:LGmirror}. On the other hand, one may expect that these two LG mirror models turn out to be equivalent by the HMS point of view. These are all completely open tasks, at least as far as the author's knowledge allows!
 \end{remark}

\begin{remark}\label{rem:HVvsGivental} The modification of the Hori-Vafa LG mirror model here proposed, and in particular the associated LG model $(\T_{\aa_0},f^\vee_{\bb_0})$, is consistent with LG mirror models proposed by Givental in \cite[Thm.~5]{Givental-ICM} for complete intersections in toric varieties. In fact, one can consider the following re-parameterization
\begin{eqnarray}\label{riparametrizzazione1}
   u_i&:=&\begin{cases}\begin{array}{cc}
\displaystyle{\sum_{i=1}^n x_i^d\cdot \prod_{j=1}^n x_j^{d-n-1}\over \psi\,\x^{\bb_0}}& \text{for}\ 1\leq i\leq n\\
\\
                                     \displaystyle{x_{n+1}^{n+1}\over \psi\,\x^{\bb_0}}& \text{for}\ i=n+1
                                   \end{array}
    \end{cases}\\
\nonumber
    q&:=&\psi^{-n-1}
\end{eqnarray}
so that the LG model $(\T_{\aa_0},f^\vee_{\bb_0})$ can be rewritten, up to a rescaling and a translation by 1, as the Givental LG model defined by the  superpotential
\begin{equation*}
    F: \C^{n+1}\longrightarrow \C\ ,\quad F(\uu):=\sum_{i=1}^{n+1} u_i = (1/\psi) f^\vee_{\bb_0}-1
\end{equation*}
restricted to the torus fibration
\begin{equation*}
  \pi:\C^{n+1}\longrightarrow \C\ ,\quad \pi(\uu):=\prod_{i=1}^{n+1} u_i=q
\end{equation*}
In Theorem~\ref{thm:d<n+1} the same approach will be extended to the case of Kodaira negative projective hypersurfaces. Then $f$-mirror symmetry and the LG/Hypersurface correspondence defined in \S\ref{sssez:LG/Hyp} turn out to give a unified procedure for constructing mirror partners of toric hypersurfaces. The same approach seems to be completely extendable to toric complete intersections: for the details, the interested reader is referred to the forthcoming paper \cite{R-fpCI}.
\end{remark}

\subsection{A-side of the topological mirror web}

   The following result characterizes which framing $D_{\aa}$ of $\P^n$, among those satisfying conditions (\ref{convenzione}), (a), (b) in Theorem~\ref{thm:dualita}, give rise to  $f$-mirror partners, of the generic $Y_d\subset\P^n$, sharing an A-side mirror behaviour.

   \begin{proposition}\label{prop:m*=k in Pn}
      Let $D_{\aa}$ be a framing of $\P^n$ satisfying conditions (\ref{convenzione}), (a) and (b) in Theorem~\ref{thm:dualita} and assume $n\geq 4$. Then, the following facts are equivalent:
      \begin{enumerate}
        \item Theorem~\ref{thm:m*=k} holds for the ftv $(\P^n,D_{\aa})$, that is,
        \begin{equation*}
          m_{Y^\vee}=k_d^n=1
        \end{equation*}
        and $Y^\vee$ is an $A$-mirror of $Y$,
        \item the number of lattice points in $\D_{\bb}$ equals the number of lattice points in $\NN=\conv(V)$\,, i.e.
            \begin{equation*}
              l(\D_{\bb})=l(\NN)=n+2
            \end{equation*}
        \item $\D(\XX_{\aa},\bb):=[\D_{\bb}]=\NN=:\conv(V)$\,,
        \item $\left[a_n /a_1\right]=1$\,.
      \end{enumerate}
    \end{proposition}

    \begin{proof} Recall that a fan matrix $\L_{\aa}$ of $\XX_{\aa}$ is given by (\ref{Lambda_a}). Then the anti-canonical polytope $\D_{-K_{\XX_{\aa}}}$ is given by
   \begin{eqnarray}\label{-K_a pol}
   \nonumber
     \D_{-K_{\XX_{\aa}}}&=&\conv\left(\left\{-\left(\left(\L_{\aa}^{\{i\}}\right)^T\right)^{-1}\cdot\1\,|\,1\leq i\leq n+1\right\}\right)\\
     &=&\left(
                          \begin{array}{ccccc}
                            1/a_1 & 0 & \ldots & 0&-1/a_{n+1} \\
                            0&1/a_2&&\vdots&\vdots\\
                            \vdots &  & \ddots & 0&\vdots \\
                            0 & \cdots & 0 & 1/a_n&-1/a_{n+1} \\
                          \end{array}
                        \right)
   \end{eqnarray}
   In particular, every facet of $\D_{-K_{\XX_{\aa}}}$ does not contain any interior point. Then, recalling (\ref{mYY}), one find that
   \begin{equation*}
     m_{Y^\vee}= \dim\P\left(H^0(\XX_{\aa},\cO_{\XX_{\aa}}(D'_{\bb}))\right)- \dim\Aut(\XX_{\aa})=l(\D_{\bb})-1-n
   \end{equation*}
   Then clearly $m_{Y^\vee}=1$ if and only if $l(\D_{\bb})=n+2=l(\NN)$\,, so giving the equivalence between items (1) and (2) in the statement.
   Moreover, notice that relation (\ref{[Deltab]}) gives
   \begin{equation*}
     [\D_{\bb}]=\conv\left(\left\{\e_1,\left[{a_n\over a_{n-1}}\right]\e_2,\ldots,\left[{a_n\over a_{1}}\right]\e_n,-\mathbf{1}\right\}\right)\supseteq \NN
   \end{equation*}
  Therefore
   \begin{equation*}
     l(\D_{\bb})= l(\NN)\ \Longleftrightarrow\ [\D_{\bb}]=\NN
   \end{equation*}
   so proving the equivalence between items (2) and (3).
   Finally, notice that, relation (\ref{[Deltab]}), again, recalling convention (\ref{convenzione}), ensures the equivalence between items (3) and (4).

      \end{proof}

   \begin{remark}
     Condition (4) in the statement of Proposition~\ref{prop:m*=k in Pn}, together with convention (\ref{convenzione}), implies that a framing $\aa$ of $\P^n$, satisfying one of the equivalent conditions in Proposition~\ref{prop:m*=k in Pn}, presents necessarily in the following shape
     \begin{equation*}
       \aa=(\underbrace{a,\ldots,a}_{\text{$n$ times}},\d:=d-na)=(a\1_n,\d)\quad\text{with}\quad 1\leq a\leq \d
     \end{equation*}
     Notice that condition (b) in Theorem~\ref{thm:dualita} gives $(a,\d)=1$, then $\aa$ is reduced if and only if $\aa=\aa_0=(\1_n,\d)$ and $a=1$. Assume $a\geq 2$: then the reduced weight vector $\q$ of $\aa$ is just given by $\q=(\1_n,\d)$. Therefore
     $$(\XX_\aa,\bb)=\left(\P(\1_n,\d)/G_\aa\ ,\ (\underbrace{\d,\ldots,\d}_{\text{$n$ times}},1)\right)$$
     with $|G_\aa|=d^{n-1}$. Then, the same argument used in Remark~\ref{rem:s-ample} shows that $dD'_\bb$ is base point free, $D'_\bb$ is semi-ample and $\pi^*(D'_\bb)$ is a very ample divisor generating $\Pic(\P(\1_n,\d))\cong\Z$.
   \end{remark}

   \begin{remark}
     The last condition (4) in Proposition~\ref{prop:m*=k in Pn} implies that a framing $D_{\aa}\neq D_{\aa_0}$ can give rise to an $A$-mirror partner of $Y_d\subset\P^n$ only if $d\geq 2n+3$. In particular, recalling considerations given in \S\ref{ssez:MWeb}, for $n=4$, the minimum value of the degree $d$ realizing an effective $A$-mirror web, that is, giving rise to multiple $A$-mirrors, is $d=11$, with the two framing $\aa_0=(1,1,1,1,7)$ and $\aa_1=(2,2,2,2,3)$\,. Here, the two mirror partners of the generic $Y^4_{11}\subset\P^4$ are given by (a suitable desingularization of) the generic hypersurfaces
     \begin{eqnarray*}
        Y^\vee_{35}&=&\left\{\left(\sum_{i=1}^4 x_i^{11}\cdot \prod_{j=1}^4 x_j^{6}\right)+\,x_{5}^{5}+\psi\,\left(\prod_{j=1}^4 x_j^{7}\right)\cdot x_{5}=0\right\}\subset\P(\1_4,7)/(\Z/11\Z)^3 \\
        Y^\vee_{15}&=&\left\{\left(\sum_{i=1}^4 x_i^{11}\cdot \prod_{j=1}^4 x_j\right)+x_{5}^{5}+\psi\,\left(\prod_{j=1}^4 x_j^{3}\right)\cdot x_{5}=0\right\}\subset\P(\1_4,3)/(\Z/11\Z)^3
     \end{eqnarray*}
     These two mirror models are not isomorphic as they have singular loci of different dimension:
     \begin{eqnarray*}
       \Sing\left(Y^\vee_{35}\right)&=&\bigcup_{i=1}^4\{x_i=x_5=0\}\ \Longrightarrow\ \dim\left(\Sing\left({Y^\vee_{35}}\right)\right)=2 \\
       \Sing\left(Y^\vee_{15}\right)&=&\left(\bigcup_{1\leq i < j\leq 4}\{x_i=x_j=x_5=0\}\right)\\
       &&\cup\left(\bigcup_{i=1}^4\left\{x_i=x_5=\sum_{\substack{1\leq j\leq 4\\j\neq i}}x_j^{11}=0\right\}\right)\\
       &\Longrightarrow&\dim\left(\Sing\left({Y^\vee_{15}}\right)\right)=1
     \end{eqnarray*}
   \end{remark}

\begin{remark}[About Hodge diamond A-symmetry]\label{rem:Hodge-A}
  By means of methods like those employed by Batyrev and Borisov \cite{BB96}, one can check that, calling $Y^\vee$ the $f$-mirror partner assigned to $Y_d\subset\P^n$ by the choice of the framing $\aa_0=(\1_n,\d)$, and assuming $n\geq 4$ and $d=n+\d\geq n+2$, then
\begin{equation}\label{h21dual}
  h^{n-2,1}(\widehat{Y}^\vee)= l^*(2\D_{\bb_0})-n-2+r
\end{equation}
where $\bb_0=(\d\cdot\1_n,1)$ and $\widehat{Y}^\vee$ is a resolution of singularities of $Y^\vee$ with $r:=\rk(\Cl(\widehat{Y}^\vee))$. If $\d\ge 2$ then $l^*(2\D_{\bb_0})\ge n+2$ and $r>1$, so giving
\begin{equation*}
  h^{n-2,1}(\widehat{Y}^\vee)>1= m_{Y^\vee}
\end{equation*}
Consequently, with that framing, there is no hope of getting any Hodge diamond $A$-symmetry, beyond the \cy setup.

Relation (\ref{h21dual}) is a consequence of a more general computation we developed in the broader context of toric complete intersections \cite{R-fpCI}.
\end{remark}

   \subsection{B-side mirroring}
    Assuming $n\geq 4$, the other side of the mirroring process, so called \emph{B-side}, is that of comparing the \ka moduli $k_{\widehat{Y}^\vee}$ with either the complex moduli $m^n_d$, as computed in (\ref{m_d^n}) (see also Remark~\ref{rem:combinatorica})
    or the Hodge number $h^{n-2,1}(Y_d)$, for a generic hypersurface $Y_d\subset\P^n$ and a generic hypersurface $\widehat{Y}^\vee\in|\widehat{D}'_\bb|$ in $\widehat{\XX}_\aa$, where $$(\P^n,D_\aa)\leftrightsquigarrow(\XX_\aa,D'_\bb)$$
    is a calibrated $f$-process, with $Y_d\sim D_\aa$, and $(\widehat{\XX}_\aa,\widehat{D}'_\bb)\longrightarrow(\XX_\aa,D'_\bb)$ is a sufficiently good resolution.
     The Hodge number $h^{n-2,1}(Y_d)$ can be computed, e.g., by means of the Griffiths' theory on Poincar\'{e} residues \cite{Griffiths}. A comparison with (\ref{m_d^n}) immediately shows that, for $d\geq n+2$,
    \begin{equation*}
      m_d^n={n+d\choose d}-(n+1)^2\neq {2d-1\choose n}-(n+1){d\choose n}=h^{n-2,1}(Y_d)
    \end{equation*}
    meaning that, also in this case, we cannot hope in a symmetry into the Hodge diamond, as in the A-side for the framing $\aa_0=(\1,d-n)$ (recall Remark~\ref{rem:Hodge-A}).

To perform the topological $B$-side check, recall that
\begin{equation*}
  m_{Y_d}=\dim \P\left(H^0(\cO_{\P^n}(d)\right) - \dim \P\GL(n+1) ={n+d\choose d}-(n+1)^2
\end{equation*}
Then, one has to exhibit a suitable (partial) resolution $\phi:\widehat{Y}^\vee\longrightarrow Y^\vee$ to compute
 $k_{\widehat{Y}^\vee}$ and make the required comparison with $m_{Y_d}$. Both the construction of $\phi$ and the computation of Hodge numbers, in particular of $h^1(\widehat{\Omega}_{\widehat{Y}^\vee})=k_{\widehat{Y}^\vee}$, are quite tricky. For $d=n+1$ this is a particular case of Batyrev's results on toric \cy hypersurfaces \cite{Batyrev94}. But this is not the case for $d=n+\d$ with $\d\ge 2$. For this reason, we just anticipate here a statement resuming results whose proof will be deferred to the forthcoming paper \cite{R-fpCI}, where this kind of construction and calculation will be performed in the more general context of complete intersections in toric varieties.

\begin{theorem}\label{thm:B-mirror}
Let $Y^\vee$ be the generic hypersurface of $\XX_{\aa_0}$ in the linear system $|D'_{\bb_0}|$ whose defining polynomial is $f^\vee\in\Cox(\XX_{\aa_0})$.
  There exists a possibly partial (depending on $\d\ge 1$) resolution $\phi: \widehat{\XX}_{\aa_0}\longrightarrow\XX_{\aa_0}$ such that the transformed hypersurface
\begin{equation}\label{trasformata}
  \widehat{Y}^\vee:=\phi^{-1}(Y^\vee)
\end{equation}
 defined as the zero-locus of $\phi^*(f^\vee)\in\Cox(\widehat{\XX}_{\aa_0})$, is either quasi-smooth or smooth (depending on $\d\ge 1$) and
\begin{equation*}
  h^{1,1}\left(\widehat{Y}^\vee\right)=k_{\widehat{Y}^\vee}=m_d^n
\end{equation*}
 That is, recalling Definition~\ref{def:A,B-mirror}, the generic $Y^\vee\subset\XX_{\aa_0}$ is a $B$-mirror partner of the generic hypersurface $Y_d\subset \P^n$. Recalling Theorem~\ref{thm:m*=k}, this means that $(Y_d,Y^\vee)$ is a pair of topologically mirror partners.
\end{theorem}

\section{Extending the duality to complete intersections in toric varieties}\label{sez:CI}

The present section is devoted to extending $f$-duality to families of complete intersection varieties in a fixed toric variety $X$, keeping in mind \S\ref{sez:dualita-hyp} and how Borisov  generalized the Batyrev duality \cite{Borisov}, \cite{BB96}.

\begin{definition}
  Let $(X,D_\aa=\sum_{j=1}^m a_jD_j)$ be a ftv and $V=\left(
                                    \begin{array}{ccc}
                                      \v_1 & \cdots & \v_m \\
                                    \end{array}
                                  \right)
  $ be a fan matrix of $X$, where $m=n+r$, recalling notation \S\ref{sssez:notazione}. A \emph{partition} of the framing $D_\aa$ is the datum of a partition
  $$\exists\,l\in \N:\quad I_1\cup\cdots\cup I_l=\{1,\ldots,m\}\ ,\quad\forall\,i\ I_i\ne\emptyset\ ,\quad\forall\,i\neq j\quad I_i\cap I_j =\emptyset$$
   and divisors $D_{\aa_1},\ldots,D_{\aa_l}$ such that
  $$\forall\,k=1,\ldots,l\quad D_{\aa_k}:=\sum_{i\in I_k}a_iD_i$$
  Clearly $D_\aa=\sum_{k=1}^l D_{\aa_k}$, that is, $\aa=\sum_{k=1}^l \aa_k$\,.\\
   The ftv $(X,D_\aa)$ with a framing partition $\aa=\sum_{k=1}^l\aa_k$ is called a \emph{partitioned ftv} and denoted by $(X,\aa=\sum_{k=1}^l\aa_k)$\,.
\end{definition}

\subsection{$f$-process for complete intersections}
Given a partitioned ftv
$$(X,D_\aa=\sum_{k=1}^l D_{\aa_k})$$
consider the following algorithm (proofs of details are deferred to the forthcoming paper~\cite{R-fpCI}).
\subsubsection{The partitioned $f$-process algorithm}\label{algoritmoDnef}
\begin{enumerate}
  \item Let $\D_\aa$ and $\D_{\aa_1},\ldots,\D_{\aa_l}$ be the polytopes associated with divisors $D_\aa$ and $D_{\aa_1},\ldots,D_{\aa_l}$, respectively, that is
  \begin{eqnarray*}
    \D_\aa&=&\{\m\in M_\R\,|\,V^T\cdot\m \geq -\aa\}\\
    \forall\,k=1,\ldots,l\quad\D_{\aa_k}&=&\{\m\in M_\R\,|\,V^T\cdot\m \geq -\aa_k\}
\end{eqnarray*}
In particular, it turns out that
\begin{equation}\label{sommaintersezione}
  \bigcap_{k=1}^l\D_{\aa_k}=\{\0\}\quad \text{and}\quad \D_\aa=\sum_{k=1}^l\D_{\aa_k}
\end{equation}
where the sum denotes the Minkowski sum of polytopes.
  \item Define
\begin{equation*}
  \cv{\D}_\aa:=\conv(\D_{\aa_1},\ldots,\D_{\aa_l})\subset M_\R
\end{equation*}
Clearly $\cv{\D}_\aa\subseteq\D_\aa$ and relations (\ref{sommaintersezione}) suffices to show that $\0\in\Int(\cv{\D}_\aa)$, \cite{R-fpCI}. Recalling Definition~\ref{def:Deltapolytope}, relations (\ref{sommaintersezione}) still hold for multiple polytopes $k_0\D_\aa$ and $k_0\D_{\aa_1},\ldots,k_0\D_{\aa_l}$, so giving that
\begin{equation*}
  \bigcap_{k=1}^l[k_0\D_{\aa_k}]=\{\0\}\quad \text{and}\quad \0\in\Int(\D(X,\aa))
\end{equation*}
since $\D(X,\aa)=[\sum_{k=1}^lk_0\D_{\aa_k}]$.
Then $\0\in\Int(\cv{\D}(X,\aa))$, being $\cv{\D}(X,\aa):=[k_0\cv{\D}_\aa]$\,, \cite{R-fpCI}.
  \item Set
\begin{equation*}
  \cv{\XX}_\aa:=\XX_{\cv{\Si}_\aa}\quad\text{where}\quad\cv{\Si}_\aa:=\Si_{\cv{\D}(X,\aa)}
\end{equation*}
and let $\cv{\L}_\aa\in\mathbf{M}(n\times \cv{m};\Z)$ be a fan matrix of $\cv{\XX}_\aa$\,, where $\cv{m}=|\cv{\Si}(1)|$. Notice that $\cv{\XX}_\aa$ is a complete toric variety, by Proposition~\ref{prop:Fmatricedipolitopo}.
  \item For every $k=1,\ldots,l$, set $m_k:=|I_k|$ and consider the matrix
\begin{equation*}
  \cv{M}_{\aa_k}:= (V_{I_k})^T\cdot\cv{\L}_\aa\in\mathbf{M}(m_k\times \cv{m};\Z)
\end{equation*}
and let $\bb_k=(b_{jk})_{j=1}^{\cv{m}}$ be \emph{the minimum non-negative column vector} such that
\begin{equation*}
  \cv{M}_{\aa_k}^T+B_k\geq \0\quad\text{where}\quad B_k:=\underbrace{\left(\,\bb_k\ \cdots\ \bb_k\,\right)}_{m_k\ \text{times}}\,\in \mathbf{M}(\cv{m}\times m_k,\N)
\end{equation*}
Then, define $\cv{\bb}:=\sum_{k=1}^l\bb_k$\,. Calling $\cv{D}_1,\ldots,\cv{D}_{\cv{m}}$ the torus invariant ge\-ne\-ra\-tors of $\Weil(\cv{\XX}_{\aa})$, there is a unique induced partition
\begin{equation*}
  J_1\cup\cdots\cup J_l=\{1,\ldots,\cv{m}\}
\end{equation*}
such that $\left(\cv{\XX}_\aa,\cv{D}_{\cv{\bb}}=\sum_{k=1}^l\cv{D}_{\bb_k}\right)$, with $\cv{D}_{\bb_k}:=\sum_{j\in J_k}b_{jk}\cv{D}_j$, is a partitioned ftv, \cite{R-fpCI}.
\item Analogously to step (1), let ${\D}_{\cv{\bb}}$ and $\cv{\D}_{\bb_1},\ldots,\cv{\D}_{\bb_l}$ be the polytopes associated with divisors $\cv{D}_{\cv{\bb}}$ and $\cv{D}_{\bb_1},\ldots,\cv{D}_{\bb_l}$, respectively, that is
  \begin{eqnarray*}
    {\D}_{\cv{\bb}}&=&\{\n\in N_\R\,|\,\cv{\L}_\aa^T\cdot\n \geq -\cv{\bb}\}\\
    \forall\,k=1,\ldots,l\quad\cv{\D}_{\bb_k}&=&\{\n\in N_\R\,|\,\cv{\L}_\aa^T\cdot\n \geq -\bb_k\}
\end{eqnarray*}
Then
\begin{equation}\label{sommaintersezione-b}
  \bigcap_{k=1}^l\cv{\D}_{\bb_k}=\{\0\}\quad \text{and}\quad {\D}_{\cv{\bb}}=\sum_{k=1}^l\cv{\D}_{\bb_k}
\end{equation}
\item Analogously to step (2), define
\begin{equation*}
  \cv{\D}_{\cv{\bb}}:=\conv(\cv{\D}_{\bb_1},\ldots,\cv{\D}_{\bb_l})\subset N_\R
\end{equation*}
Clearly $\cv{\D}_{\bb}\subseteq{\D}_{\cv{\bb}}$ and relations (\ref{sommaintersezione-b}) ensure that $\0\in\Int(\cv{\D}_{\cv{\bb}})$. Then, (\ref{sommaintersezione-b}) still holds for multiple polytopes $h_1{\D}_{\cv{\bb}}$ and $h_1\cv{\D}_{\bb_1},\ldots,h_1\cv{\D}_{\bb_l}$, so giving that
\begin{equation*}
  \bigcap_{k=1}^l[h_1\cv{\D}_{\bb_k}]=\{\0\}\quad \text{and}\quad \0\in\Int(\D(\cv{\XX}_{\aa},\cv{\bb}))
\end{equation*}
since $\D(\cv{\XX}_{\aa},\cv{\bb})=[\sum_{k=1}^lh_1\cv{\D}_{\bb_k}]$, where $h_1$ is defined as the minimum positive integer such that $\0\in\Int([h_1\D_{\cv{\bb}}])$.
Then $\0\in\Int(\cv{\D}(\cv{\XX}_{\aa},\cv{\bb}))$,
being $$\cv{\D}(\cv{\XX}_{\aa},\cv{\bb}):=[h_1\cv{\D}_{\cv{\bb}}]$$
    \item  Analogously to step (3), set
\begin{equation*}
  \cv{\XX}_{\cv{\bb}}:=\XX_{\cv{\Si}_{\cv{\bb}}}\quad\text{where}\quad\cv{\Si}_{\cv{\bb}}:=
  \Si_{\cv{\D}(\cv{\XX}_{\aa},\cv{\bb})}
\end{equation*}
and let $\cv{\L}_{\cv{\bb}}\in\mathbf{M}(n\times \widetilde{m};\Z)$ be a fan matrix of $\cv{\XX}_{\cv{\bb}}$\,, for some $\widetilde{m}\in\N$\,. As above, $\cv{\XX}_{\cv{\bb}}$ is a complete toric variety, by Proposition~\ref{prop:Fmatricedipolitopo}.
    \item Analogously to step (4), for every $k=1,\ldots,l$, set $\cv{m}_k:=|J_k|$ and consider the matrix
\begin{equation*}
  \cv{M}_{\aa_k,\cv{\bb}}:=\left((\cv{\L}_\aa)_{J_k}\right)^T\cdot\cv{\L}_{\cv{\bb}}\in\mathbf{M}(\cv{m}_k\times \widetilde{m};\Z)
\end{equation*}
and let $\cc_k=(c_{j,k})_{j=1}^{\widetilde{m}}$ be \emph{the minimum non-negative column vector} such that
\begin{equation*}
  \cv{M}_{\aa_k,\cv{\bb}}^T+C_k\geq \0\quad\text{where}\quad C_k:=\underbrace{\left(\,\cc_k\ \cdots\ \cc_k\,\right)}_{\cv{m}_k\ \text{times}}\,\in \mathbf{M}(\widetilde{m}\times \cv{m}_k,\N)
\end{equation*}
Then, $(\cv{\XX}_{\cv{\bb}}, \cv{\cc}:=\sum_{k=1}^l\cc_k)$ is a partitioned ftv, whose partitioned framing is given by $\widetilde{D}_{\cv{\cc}}= \sum_{j=1}^{\widetilde{m}}c_{jk}\widetilde{D}_j$, calling $\widetilde{D}_1,\ldots,\widetilde{D}_{\widetilde{m}}$ the torus invariant ge\-ne\-ra\-tors of $\Weil(\cv{\XX}_{\cv{\bb}})$.
\end{enumerate}

\begin{definition}[partitioned $f$-process]
Following the previous algorithm~\ref{algoritmoDnef}, the partitioned ftv $(\cv{\XX}_\aa,\cv{\bb}=\sum_{k=1}^l\bb_k)$, is called a \emph{partitioned $f$-dual} of $(X,\aa=\sum_{k=1}^l\aa_k)$.

\noindent A double application of partitioned $f$-duality defines a \emph{partitioned $f$-process}
\begin{equation}\label{nefDprocess}
  \left(X,\aa=\sum_{k=1}^l\aa_k\right)\ {\rightsquigarrow}\ \left(\cv{\XX}_\aa,\cv{\bb}=\sum_{k=1}^l\bb_k\right)\ {\rightsquigarrow}\ \left(\cv{\XX}_{\cv{\bb}},\cv{\cc}=\sum_{k=1}^l\cc_k\right)
\end{equation}
which is called \emph{calibrated} if there exist $\Xi\in\SF(V)$ and $\Xi'\in\SF(\cv{\L}_{\cv{\bb}})$, refining $\Si$ and $\cv{\Si}_{\cv{\bb}}$, respectively, such that
  $$\left(\widehat{X},\vf^*D_\aa\right)\ {\cong}\ \left(\widehat{X}',(\vf')^*\widetilde{D}_{\cv{\cc}}\right)$$
  are isomorphic framed toric varieties, where
  $$\vf:\widehat{X}(\Xi)\longrightarrow X(\Si)\quad\text{and}\quad\vf':\widehat{X}'(\Xi')\longrightarrow \cv{\XX}_{\cv{\bb}}(\cv{\Si}_{\cv{\bb}})$$
  are the small resolutions associated with the choice of $\Xi$ and $\Xi'$, respectively.
\end{definition}

The following characterization of a calibrated partitioned $f$-process is a direct consequence of Theorem~\ref{thm:Deltatriviale}.

\begin{proposition}
  In the above notation, up to identifying lattices $M$ (hence $N$) of $X$ and $\cv{\XX}_{\cv{\bb}}$, the partitioned $f$-process (\ref{nefDprocess}) is calibrated if and only if
  \begin{eqnarray}\label{nefDcalibrato}
\nonumber
  \cv{\L}_{\cv{\bb}} &=& V\quad\text{up to a permutation of columns} \\
  \forall\,k=1,\ldots,l\quad \cc_k&=& \aa_k
\end{eqnarray}
\end{proposition}

\begin{definition}[$f$-mirror of a complete intersection]\label{def:nefpmirror}
  Given the partitioned ftv $(X,\aa=\sum_{k=1}^{l}\aa_k)$, assume that the associated partitioned $f$-process (\ref{nefDprocess}) is calibrated. Consider the complete intersection subvariety
  $$Y:=\bigcap_{k=1}^lY_k\subset X\quad\text{with}\quad Y_k\in|D_{\aa_k}|$$
  The generic complete intersection subvariety
  $$Y^\vee:=\bigcap_{k=1}^l Y^\vee_k\subset \cv{\XX}_\aa\quad\text{with}\quad Y_k\in|\cv{D}_{\bb_k}|$$
  is called a \emph{$f$-mirror partner of} $Y$.
\end{definition}

\begin{remark}
  If $l=1$, that is, the partition is trivial, the $f$-mirror duality defined by the previous Definition~\ref{def:nefpmirror} reduces to give the $f$-mirror duality between hypersurfaces in toric varieties defined in Definition~\ref{def:mirror}.
\end{remark}

\begin{remark}\label{rem:CIequazioni}
  This is the analogue of what described by Remark~\ref{rem:famiglie} when $l=1$. \\
  One can explicitly describe the defining polynomials of both $Y$ and $Y^\vee$ in the Cox rings of $X$ and $\cv{\XX}_\aa$, respectively. Namely:
  \begin{itemize}
    \item[(a)] for every $k=1,\ldots,l$, the lattice polytope $[\D_{\aa_k}]$ is the Newton polytope of $Y_k\in|D_{\aa_k}|$; call $\overline{\L}_{\aa_k}$ a matrix whose columns are given by all the lattice points in $[\D_{\aa_k}]$: it is well defined up to a permutation of columns; setting $l_k:=|\D_{\aa_k}\cap M|$, then $\overline{\L}_{\aa_k}$ is a $n\times l_k$ integer matrix; define
        \begin{equation*}
          \overline{M}_{\aa_k}:= V^T\cdot \overline{\L}_{\aa_k}\quad\text{and}\quad \overline{A}_k:=\underbrace{\left(\,\aa_k\ \cdots\ \aa_k\,\right)}_{l_k\ \text{times}}\,\in \mathbf{M}(m\times l_k;\N)\,;
        \end{equation*}
        then the polynomial of $Y_k$ is given by
        \begin{equation*}
          f_k=\sum_{j=1}^{l_k} c_j\x^{\m_j} \in \Cox(X)\cong\C[x_1,\ldots,x_m]
        \end{equation*}
        where $\m_j=(m_{i,j})$ is the $j$-th column of $\overline{M}_{\aa_k}+\overline{A}_k$ and $\x^{\m_j}:=\prod_{i=1}^m x_i^{m_{i,j}}$;
    \item[(b)] recalling step (5) in the algorithm~\ref{algoritmoDnef}, the lattice polytope $[\cv{\D}_{\bb_k}]$ is the Newton polytope of $Y_k^\vee\in|\cv{D}_{\bb_k}|$; call $\overline{\L}_{\bb_k}$ a matrix whose columns are given by all the lattice points in $[\cv{\D}_{\bb_k}]$; setting $l'_k:=|\cv{\D}_{\bb_k}\cap N|$, then $\overline{\L}_{\bb_k}$ is a $n\times l'_k$ integer matrix; define
        \begin{equation*}
          \overline{M}_{\aa,\bb_k}:= \cv{\L}_\aa^T\cdot \overline{\L}_{\bb_k}\quad\text{and}\quad \overline{B}_k:=\underbrace{\left(\,\bb_k\ \cdots\ \bb_k\,\right)}_{l'_k\ \text{times}}\,\in \mathbf{M}(\cv{m}\times l'_k;\N)\,;
        \end{equation*}
        then the polynomial of $Y^\vee_k$ is given by
        \begin{equation*}
          f^\vee_k=\sum_{j=1}^l c_j\x^{\n_j} \in \Cox(\cv{\XX}_\aa)\cong\C[x_1,\ldots,x_{\cv{m}}]
        \end{equation*}
        where $\n_j=(n_{i,j})$ is the $j$-th column of $\overline{M}_{\aa,\bb_k}+\overline{B}_k$ and $\x^{\n_j}:=\prod_{i=1}^{\cv{m}} x_i^{n_{i,j}}$.
  \end{itemize}
  Notice that, for every $k$, both $f_k$ and $f_k^\vee$ are homogeneous polynomials, with respect to degrees induced by class groups. In fact, columns of both $\overline{M}_{\aa_k}$ and $\overline{M}_{\aa,\bb_k}$ determine trivial divisors, up to linear equivalence. Then
  $$\deg(f_k)=[D_{\aa_k}]\in\Cl(X)\quad\text{and}\quad\deg(f_k^\vee)=[\cv{D}_{\bb_k}]\in\Cl(\cv{\XX}_\aa)$$
\end{remark}

\subsection{Generalizing Batyrev-Borisov duality} Definition~\ref{def:nefpmirror} is clearly motivated by the case when $X$ is a Fano toric variety and $\aa=\1$, that is $D_\aa=-K_X$. In fact, in this case a framing partition $\aa=\sum_{k=1}^{l} \aa_k$ such that $D_{\aa_k}$ is a nef divisor, for every $k=1,\ldots,l$, is precisely a Borisov nef partition of the anti-canonical divisor \cite[Def.~2.5, Rem.~2.6]{Borisov}, \cite[Def.~4.6]{BB96}. In this case, $f$-duality reduces to give the well known Batyrev-Borisov mirror symmetry between \cy complete intersections in Fano toric varieties.

\begin{example}\label{ex:nefD}
  To fix ideas, an easy example is presented here.
  Consider the partitioned ftv $$\left(X,\aa=\sum_{k=1}^l\aa_k\right)=\left(\P^2,(1,1,2)=(1,0,0)+(0,1,2)\right)$$
  where weights of the partition are referred to primitive generators of the 1-skeleton of the fan defining $\P^2$ and given by columns of the fan matrix
  \begin{equation*}
    V=\left(
        \begin{array}{ccc}
          1 & 0 & -1 \\
          0 & 1 & -1 \\
        \end{array}
      \right)
  \end{equation*}
  Notice that the framing partition $(1,1,2)=(1,0,0)+(0,1,2)$ is actually a \emph{nef} partition, as both of the summands give back nef divisors.
  We are then considering a generic complete intersection $Y\subset\P^2$ of a line and a cubic (hence 3 points) whose equations are given by Newton polytopes (step (1) in algorithm~\ref{algoritmoDnef})
  \begin{eqnarray*}
    \D_{\aa_1} &:=& \conv(\L_{\aa_1})\ ,\quad \L_{\aa_1}:=\left(
                                                           \begin{array}{ccc}
                                                             0 & -1 & -1 \\
                                                             0 & 1 & 0 \\
                                                           \end{array}
                                                         \right)
     \\
    \D_{\aa_2} &:=& \conv(\L_{\aa_2})\ ,\quad \L_{\aa_2}:=\left(
                                                           \begin{array}{ccc}
                                                             3 & 0 & 0 \\
                                                             -1 & 2 & -1 \\
                                                           \end{array}
                                                         \right)
  \end{eqnarray*}
  Notice that $\D_\aa=\D_{\aa_1}+\D_{\aa_2}$ (see Fig.~\ref{Fig3}).
\begin{figure}
\begin{center}
\includegraphics[width=12truecm]{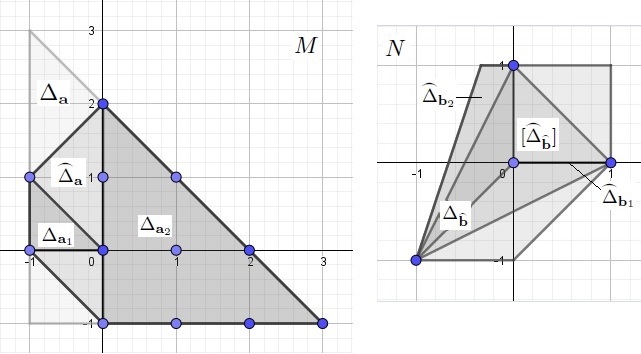}
\caption{\label{Fig3} The calibrated partitioned $f$-process of Example \ref{ex:nefD}.}
\end{center}
\end{figure}

\noindent By part (a) in Remark~\ref{rem:CIequazioni}, polynomials defining $Y$ are then given by
\begin{eqnarray*}
  \overline{M}_{\aa_1}+\overline{A}_1 &=&  V^T\cdot\L_{\aa_1}+\left(
                                                                \begin{array}{ccc}
                                                                  \aa_1 & \aa_1 & \aa_1 \\
                                                                \end{array}
                                                              \right) = \left(
                                                                \begin{array}{ccc}
                                                                  1 & 0 & 0 \\
                                                                  0 & 1 & 0 \\
                                                                  0 & 0 & 1 \\
                                                                \end{array}
                                                              \right)
  \\
  &\Longrightarrow& f_1=a_1x_1+a_2 x_2+a_3 x_3\\
  \overline{M}_{\aa_2}+\overline{A}_2 &=&  V^T\cdot\overline{\L}_{\aa_2}+\underbrace{\left(
                                                                \begin{array}{ccc}
                                                                  \aa_2 & \cdots & \aa_2 \\
                                                                \end{array}
                                                              \right) }_{\text{10 times}}
                                                              \\
                                                              &=& \left(
                                                                \begin {array}{cccccccccc} 0&1&0&2&1&0&3&2&1&0\\ 3&2&2&1&1&1&0&0&0&0\\ 0&0&1&0&1&2&0&1&2&3\end {array}
                                                              \right)
                                                              \\
                                                              &\Longrightarrow& f_2=b_1x_2^3+b_2x_1x_2^2+b_3x_2^2x_3+b_4x_1^2x_2+b_5x_1x_2x_3\\
                                                              &&\quad\quad\ +b_6x_2x_3^2+b_7x_1^3+b_8x_1^2x_3+b_9x_1x_3^2+b_{10}x_3^3
\end{eqnarray*}
Then, generically, $Y$ is given by 3 distinct aligned points.

\noindent A mirror partner $Y^\vee$ of $Y$ is determined by part (b) in Remark~\ref{rem:CIequazioni}. Namely, by step (2) in algorithm~\ref{algoritmoDnef}, one has
\begin{equation*}
  \cv{\D}_\aa:=\conv\left(\D_{\aa_1},\D_{\aa_2}\right)=\conv\left(
                                                                    \begin{array}{ccccc}
                                                                      -1 & -1 & 3 & 0 & 0 \\
                                                                      1 & 0 & -1 & 2 & -1 \\
                                                                    \end{array}
  \right)
\end{equation*}
Then, passing to step (3), one has
\begin{equation*}
  \cv{\L}_\aa=\left(
                                                                    \begin{array}{ccccc}
                                                                      -1 & -1 & 3 & 0 & 0 \\
                                                                      1 & 0 & -1 & 1 & -1 \\
                                                                    \end{array}
  \right)
\end{equation*}
and this is enough to determine the fan $\cv{\Si}_\aa$ of $\cv{\XX}_\aa$. In particular, an easy check gives that $\cv{\XX}_\aa$ is the blow up of $\P(1,2,1)$ in two distinct points.

\noindent Step (4) in algorithm~\ref{algoritmoDnef} allows us to compute the partitioned framing $\cv{\bb}=\bb_1+\bb_2$ over $\cv{\XX}_\aa$. Namely
\begin{eqnarray*}
  \cv{M}_{\aa_1}^T &=& \cv{\L}_\aa^T\cdot\left(
                                           \begin{array}{c}
                                             1  \\
                                             0  \\
                                           \end{array}
                                         \right)= \left(
                                                    \begin{array}{c}
                                                      -1  \\
                                                      -1  \\
                                                      3  \\
                                                      0  \\
                                                      0  \\
                                                    \end{array}
                                                  \right)\ \Longrightarrow\ \bb_1=(1,1,0,0,0)
   \\
  \cv{M}_{\aa_2}^T &=& \cv{\L}_\aa^T\cdot\left(
                                           \begin{array}{cc}
                                             0 & -1  \\
                                             1 & -1  \\
                                           \end{array}
                                         \right)= \left(
                                                    \begin{array}{cc}
                                                      1 & 0  \\
                                                      0 & 1  \\
                                                      -1 & -2  \\
                                                      1 & -1  \\
                                                      -1 & 1  \\
                                                    \end{array}
                                                  \right)\ \Longrightarrow\ \bb_2=(0,0,2,1,1)
\end{eqnarray*}
Then, step (5) gives polytopes associated with divisors $\cv{D}_{\cv{\bb}}, \cv{D}_{\bb_1}, \cv{D}_{\bb_2}$, namely
\begin{eqnarray*}
  \cv{\D}_{\bb_1} &=& \conv\left(
                             \begin{array}{cc}
                               1 & 0 \\
                               0 & 0 \\
                             \end{array}
                           \right)
   \\
  \cv{\D}_{\bb_2} &=& \conv\left(
                             \begin{array}{cccc}
                               0 & 0 & -1 & -1/3 \\
                               0 & 1 & -1 & 1 \\
                             \end{array}
                           \right)
  \\
  \D_{\cv{\bb}} &=& \conv\left(
                             \begin{array}{ccccc}
                               1 & 1 & 0 & -1/3 & -1 \\
                               0 & 1 & -1 & 1 & -1 \\
                             \end{array}
                           \right)
\end{eqnarray*}
Notice that $\D_{\cv{\bb}}=\cv{\D}_{\bb_1}+\cv{\D}_{\bb_2}$. By a direct check (use e.g. \cite[Thm.~3]{RT-Ample}), divisors $\cv{D}_{\cv{\bb}}$, $\cv{D}_{\bb_1}$ and $\cv{D}_{\bb_2}$ turn out to be semi-ample and line bundles $\cO_{\cv{\XX}_{\aa}}(3\cv{D}_{\cv{\bb}})$, $\cO_{\cv{\XX}_{\aa}}(\cv{D}_{\bb_1})$ and $\cO_{\cv{\XX}_{\aa}}(3\cv{D}_{\bb_2})$ be globally generated.

\noindent Passing to step (6) one gets
\begin{equation*}
  \cv{\D}_{\cv{\bb}}:=\conv(\cv{\D}_{\bb_1},\cv{\D}_{\bb_2})=\conv\left(
                             \begin{array}{cccc}
                               1 & 0 & -1 & -1/3 \\
                               0 & 1 & -1 & 1 \\
                             \end{array}
                           \right)
\end{equation*}
Therefore $[\cv{\D}_{\cv{\bb}}]=\conv(V)$, so giving $\cv{\L}_{\cv{\bb}}=V$, up to a permutation on columns, which is the first condition in (\ref{nefDcalibrato}). To check the second one, by step (8) one gets
\begin{eqnarray*}
  \cv{M}_{\aa_1,\cv{\bb}}^T &=& V^T\cdot\L_{\aa_1}=\left(
                                                     \begin{array}{ccc}
                                                       0 & -1 & -1 \\
                                                       0 & 1 & 0 \\
                                                       0 & 0 & 1 \\
                                                     \end{array}
                                                   \right)\ \Longrightarrow\ \cc_1=(1,0,0)=\aa_1
   \\
  \cv{M}_{\aa_2,\cv{\bb}}^T &=& V^T\cdot\L_{\aa_2}=\left(
                                                     \begin{array}{ccc}
                                                       3 & 0 & 0 \\
                                                       -1 & 2 & -1 \\
                                                       -2 & -2 & 1 \\
                                                     \end{array}
                                                   \right)\ \Longrightarrow\ \cc_2=(0,1,2)=\aa_2
\end{eqnarray*}
and the partitioned $f$-process associated with $(\P^3,(1,0,0)+(0,1,2))$ turns out to be calibrated. Then, recalling part (b) of Remark~\ref{rem:CIequazioni}, polynomials defining the mirror partner $Y^\vee$ are given, in the Cox ring $\Cox(\cv{\XX}_\aa)\cong\C[x_1,\ldots,x_5]$, by
\begin{eqnarray*}
  \overline{M}_{\aa,\bb_1}+\overline{B}_1 &=&  \cv{\L}_\aa^T\cdot\left(
                                                                   \begin{array}{cc}
                                                                     1 & 0 \\
                                                                     0 & 0 \\
                                                                   \end{array}
                                                                 \right)
  +\left(
                                                                \begin{array}{cc}
                                                                  \bb_1 & \bb_1 \\
                                                                \end{array}
                                                              \right) = \left(
                                                                \begin{array}{cc}
                                                                  0 & 1 \\
                                                                  0 & 1 \\
                                                                  3 & 0 \\
                                                                  0 & 0 \\
                                                                  0 & 0 \\
                                                                \end{array}
                                                              \right)
  \\
  &\Longrightarrow& f^\vee_1=a_1x_3^3+a_2 x_1x_2\\
  \overline{M}_{\aa,\bb_2}+\overline{B}_2 &=&  \cv{\L}_\aa^T\cdot\left(
                                                                   \begin{array}{ccc}
                                                                    0 & -1 & 0 \\
                                                                    1 & -1 & 0 \\
                                                                   \end{array}
                                                                 \right)
  +\left(
                                                                \begin{array}{ccc}
                                                                  \bb_2 & \bb_2 &\bb_2 \\
                                                                \end{array}
                                                              \right) =
                                                              \\
                                                              &=& \left(
                                                                \begin {array}{ccc} 1&0&0\\ 0&1&0\\
                                                                1&0&2\\
                                                                2&0&1\\
                                                                0&2&1\end {array}
                                                              \right)
                                                              \\
                                                              &\Longrightarrow& f^\vee_2=b_1x_1x_3x_4^2+b_2x_2x_5^2+b_3x_3^2x_4x_5
\end{eqnarray*}
Therefore $Y^\vee=Y^\vee_1\cap Y^\vee_2\subset\cv{\XX}_\aa$, where $Y^\vee_1$ is an hypersurface of degree $(3,3,0)\in \Cl(\cv{\XX}_\aa)$ and $Y^\vee_2$ is an hypersurface of degree $(2,3,2)\in \Cl(\cv{\XX}_\aa)$.
\end{example}

\subsection{Mirroring projective complete intersections} After Theorem~\ref{thm:dualita} and Corollary~\ref{cor:ipersuperfici}, one may expect that analogous statements could still hold for suitable partitioned framed projective spaces and projective complete intersections. This is actually the case, holding the following

\begin{theorem}\label{thm:CI}
  Let $Y_d=\bigcap_kY_{d_k}\subseteq\P^n$ be a complete intersection of $l$ generic projective hypersurface, of degree $(d_1,\ldots,d_l)$ such that $\sum_kd_k\geq n+1$. Then there always exists a partitioned framing ${\aa}=\sum_k\aa_k$ of $\P^n$ such that $Y_{d_k}\sim D_{\aa_k}$, for every $k$, and the associated partitioned $f$-process is calibrated.
\end{theorem}

For the proof and any further detail, the interested reader is referred to \cite{R-fpCI}.

\section{Weakly framed toric varieties and the work of Givental}\label{sez:wftv}

In Definition~\ref{def:ftv} we asked for a \emph{framing} to be a strictly effective divisor. This is motivated by the will to get an involutive duality between pairs of framed toric varieties (recall Remark~\ref{rem:negKod}). On the other hand, Givental's approach \cite{Givental-ICM} produces mirror partners of complete varieties admitting non-negative first Chern class, by means of LG models, so introducing a strong asymmetry in the mirror correspondence, \wrt the \cy case. Actually he proved a Mirror Theorem in the case of toric complete intersections \cite{Givental96}. Consequently, we are led to relax the definition of a framing, just requiring it is nothing more than an \emph{effective divisor}. In this way, one gets a bridge between Givental's LG mirror models and those proposed by Hori-Vafa \cite{Hori-Vafa} for varieties admitting negative first Chern class, represented by the Krawitz duality of Laurent LG models described \S~\ref{ssez:K-dualita}. To get a complete vision of this unifying construction the reader should compare what follows with \S~\ref{ssez:K-dualita} and \S~\ref{ssez:HoriVafa}.

\begin{definition}[Weakly framed toric variety (wftv)]\label{def:wftv} A \emph{weakly framed toric variety} is a couple $(X,D_\aa)$ (also denoted $(X,\aa)$) where:
   \begin{itemize}
     \item $X$ is a complete toric variety, with $\dim(X)=n$ and $\rk(\Pic(X))=r$,
     \item $D=\sum_{\rho\in\Si(1)}a_\rho D_\rho=\sum_{i=1}^m a_i D_i\in\Weil(X)$, with $m=n+r$, is an effective torus invariant Weil divisor, called a \emph{weak framing} of $X$.
   \end{itemize}
   \end{definition}
By Proposition~\ref{prop:gg}, the associated polyhedron $\D_\aa$ is still a polytope, but in general $\0\in M$ is no more a relative interior point of $\D_\aa$, but just a lattice point of $\D_\aa$. Recalling Definition~\ref{def:Deltapolytope}, define the $f$-polytope associated with a wftv $(X,\aa)$ to be the following
\begin{equation*}
  \D(X,\aa):=[\D_\aa]\ \Longrightarrow\ \XX_\aa:=\XX_{\D(X,\aa)}=\XX_{[\D_\aa]}
\end{equation*}
If $\aa$ is a weak framing and not a framing, that is $a_i=0$ for some $i$, then $\0\in\partial[\D_\aa]$ and the toric variety $\XX_\aa$ over the polytope $\D(X,\aa)$ is no more complete. We cannot hope to reconstructing a mirror wftv of $(X,\aa)$. By the way, we can adopt the Givental's asymmetry and thinking of $\XX_\aa$, endowed with a sort of framing we are going to define in a moment, in terms of LG mirror model. More precisely, calling $\L_\aa$ the fan matrix of $\XX_\aa$, obtained by the $f$-polytope $\D(X,\aa)$ as in Proposition~\ref{prop:Fmatricedipolitopo}, and recalling (\ref{b}), let us define the \emph{mirror framing} as the \emph{minimum non-negative column vector} $\bb=(b_j)_{j=1}^{m'}$ (i.e. effective divisor $D'_\bb\in \Weil(\XX_\aa)$) such that
\begin{equation}\label{wb}
  M_\aa^T+B\geq \0\quad\text{where}\quad B:=\underbrace{\left(\,\bb\ \cdots\ \bb\,\right)}_{m\ \text{times}}\,\in \mathbf{M}(m'\times m;\N)
\end{equation}
being $V$ a fan matrix of $X$ and $M_\aa:=V^T\cdot\L_\aa\in\mathbf{M}(m\times m';\Z)$, as usual. One can now go on as in \S\ref{sssez:LG/Hyp}, by setting:
\begin{itemize}
  \item $\T_\aa\cong(\C^*)^n$ be the maximal acting torus on $\XX_\aa$,
  \item $f^\vee_\bb:=f^\vee/\x^\bb\in\C[\x,\x^{-1}]$, where $f^\vee$ is the generic polynomial given in (\ref{fdual}), gene\-ra\-ted by the columns of the matrix $\overline{M}_{\aa,\bb}+\overline{B}$, with
      \begin{equation}\label{Mab(wf)}
        M_{\aa,\bb}:=\L_\aa^T\cdot V
      \end{equation}
      Then $f^\vee_\bb$ is the generic polynomial gene\-ra\-ted by the columns of the matrix $\overline{M}_{\aa,\bb}$.
\end{itemize}

\begin{remark}
  Notice that definition (\ref{Mab(wf)}) of $M_{\aa,\bb}$ differs from that given in display (\ref{c}) in \S\ref{ssez:Deltaprocesso}, namely $M_{\aa,\bb}:=\L^T_\aa\cdot\L_\bb$. In fact, if $\aa$ is not strictly effective then $\XX_\aa$ is not complete and $\D_\bb$ is a polyhedron and not a polytope. Then $\XX_\bb$ and $\L_\bb$ are not well defined. Nevertheless, if $(X,\aa)$ is a ftv  admitting a calibrated $f$-process then $\L_\bb=V$, up to a permutation of columns, and the two definitions coincide.
\end{remark}

\begin{definition}\label{def:LGmirror}
  Given a wftv $(X,\aa)$, let $f$ be the generic polynomial constructed as in (\ref{f-WT}) and generated by the columns of $\overline{M}_\aa+\overline{A}$, and let $Y\in|D_\aa|$ be the hypersurface defined by $f$. Then:
  \begin{enumerate}
    \item the hypersurface $Y^\vee\subset\XX_\aa$ defined by the generic polynomial $f^\vee$ is called an \emph{$f$-mirror partner} of $Y$,
    \item the LG model given by $(\T_\aa,f^\vee_\bb)$ is called a \emph{$f$-mirror LG model} of $Y$.
  \end{enumerate}

\end{definition}

\subsection{Mirror partners of hypersurfaces of degree $d\le n$ in $\P^n$}\label{ssez:NegKod}
 Passing from a framing to a weak framing, that is,  dropping the word ``strictly'', explains why one cannot expect a complete mirror model for toric hypersurfaces (and complete intersections, adapting to a weak framing what described in \S\ref{sez:CI}) associated with a weak framing. The LG mirror model given in Definition~\ref{def:LGmirror} shares interesting similarities to the LG mirror model proposed by Givental. Here we consider, e.g., the case of hypersurfaces $Y_d\subset\P^n$ of degree $d\leq n$, for sake of completeness \wrt what analyzed in \S\ref{sez:ipersuperfici}. The same result can be extended to projective complete intersections: for further details, the interested reader is referred to the forthcoming paper \cite{R-fpCI}.

\begin{theorem}\label{thm:d<n+1}
  For every $d=1,\ldots,n$, set $\aa_d:=(\1_d,\0_{n+1-d})$ and consider the wftv $(\P^n,\aa_d)$. Then
\begin{equation}\label{wf-politopo}
  \D(\P^n,\aa_d)=[\D_{\aa_d}]=\D_{\aa_d}
\end{equation}
is the Newton polytope of the generic degree $d$ homogeneous polynomial in $\C[\x]$, whose zero-locus defines the generic hypersurface $Y_d\subset\P^n$. Moreover, $\XX_{\aa_d}$ is a non-complete toric variety and
\begin{enumerate}
  \item an $f$-mirror partner of $Y_d$ is given by the hypersurface $Y^\vee_d$ of $\XX_{\aa_d}$ defined, up to a variables rescaling, as the zero-locus of
      \begin{eqnarray*}
        f^\vee_1 &=& \prod_{i=1}^nx_i\cdot\left(\psi+\sum_{i=1}^n x_i\right)+1\in\Cox(\XX_{\aa_1})\cong\C[x_1,\dots,x_n]\quad\begin{array}{c}
                                                             \text{if $d=1$\,,} \\
                                                             \text{being $\XX_{\aa_1}\cong\C^n$}
                                                           \end{array}\\
        f^\vee_d &=& \prod_{k=1}^{n+1}x_k\cdot\left(\psi+\sum_{j=d+1}^{n+1}x_j^{d}\right)+\sum_{i=1}^d x_i^d \in\Cox(\XX_{\aa_d})\cong\C[x_1,\dots,x_{n+1}] \quad\text{if $2\le d\le n$}
      \end{eqnarray*}
where $\psi\in\C$ is the unique complex modulus of the mirror family, so giving an $A$-side mirror check, as $h^{1,1}(Y_d)=1$,
  \item an $f$-mirror LG model of $Y_d$ is given by $(\T_{\aa_d},f^\vee_{\bb_d})$, where $\T_{\aa_d}\cong(\C^*)^n$ is the maximal acting torus on $\XX_{\aa_d}$ and
      \begin{equation*}
        f^\vee_{\bb_d}=\left\{\begin{array}{cc}
                               \psi+\sum_{i=1}^n x_i+1/\prod_{i=1}^nx_i & \text{if $d=1$} \\
                               \psi+\sum_{j=d+1}^{n+1}x_j^{d}+\left(\sum_{i=1}^d x_i^d\right)/ \prod_{k=1}^{n+1}x_k & \text{if $2\le d\le n$}
                             \end{array}
        \right.
      \end{equation*}
\end{enumerate}
In particular, the LG mirror model proposed in (2) is a re-parameterization of the Givental's one, so giving a complete set of solutions of the Picard-Fuchs differential system controlling the quantum variation of Hodge structure on $Y_d$\,.
\end{theorem}

\begin{proof}
  Equalities (\ref{wf-politopo}) are immediately obtained by definitions.

  Assume now $d=1$. Then
\begin{equation*}
  \L_{\aa_1}=\left(
               \begin{array}{ccccc}
                 -1 & -1& \cdots&-1 & -1 \\
                 1 &0  &\cdots& 0&0 \\
                 0&1&0&\cdots&0\\
                 \vdots&\ddots&\ddots&\ddots&\vdots\\
                 0&\cdots&0&1&0
               \end{array}
             \right)=\left(
                       \begin{array}{ccc}
                         &-\1_n& \\
                         \hline
                         I_{n-1}&\vline &\0_{n-1}^T \\
                       \end{array}
                     \right)\in\M(n,n;\Z)
\end{equation*}
Then $\XX_{\aa_1}\cong\C^n$ is affine and $\Cox(\XX_{\aa_1})\cong\C[x_1,\ldots,x_n]$.
Moreover, (\ref{wb}) gives
\begin{equation*}
  M_{\aa_1}^T=\L_{\aa_1}^T\cdot V= \left(
                                     \begin{array}{cc}
                                       -\1_n^T & I_n \\
                                     \end{array}
                                   \right)\ \Longrightarrow\ \bb_1=\1_n\ \Longrightarrow\ \x^{\bb_1}=\prod_{i=1}^nx_i
\end{equation*}
\begin{equation*}
\overline{M}_{\aa_1,\bb_1}+\overline{B}=\L_{\aa_1}^T\cdot \overline{V}+\overline{B}=\left(
                                                                                      \begin{array}{cccccc}
                                                                                        1 & 0 & 2 & 1 & \cdots & 1 \\
                                                                                         \vdots& \vdots & 1 & 2 & \cdots & 1 \\
                                                                                         \vdots& \vdots & \vdots & \ddots & \ddots & \vdots \\
                                                                                         1& 0 & 1 & \cdots & 1 & 2 \\
                                                                                      \end{array}
                                                                                    \right)\in\M(n,n+2;\Z)
\end{equation*}
\begin{eqnarray*}
  &\Longrightarrow&f_1^\vee=\prod_{i=1}^nx_i\cdot\left(\psi+\sum_{i=1}^n x_i\right)+1\quad(\text{up to a variables rescaling})\\
   &\Longrightarrow& f^\vee_{\bb_1}=\psi+\sum_{i=1}^n x_i+{1\over\prod_{i=1}^nx_i}
\end{eqnarray*}
Assume now $2\le d\le n$. Then
\begin{equation*}
  \L_{\aa_d}=\left(
                       \begin{array}{ccccc}
                         &-\1_{d,n+1}&+&\left(dI_d\,\vline\,\0_{d,n+1-d}\right)& \\
                         \hline
                         \0_{n-d,d}&\vline &dI_{n-d}&\vline&\0_{n-d}^T \\
                       \end{array}
                     \right)\in\M(n,n+1;\Z)
\end{equation*}
and $\Cox(\XX_{\aa_d})\cong\C[x_1,\ldots,x_{n+1}]$. Moreover, (\ref{wb}) gives
\begin{equation*}
  M_{\aa_d}^T=\L_{\aa_d}^T\cdot V= \left(
                                     \begin{array}{c|c}
                                       -\1_{d,d}+dI_d & \0_{d,n+1-d} \\ \hline
                                       -\1_{n+1-d,d}& dI_{n+1-d}
                                     \end{array}
                                   \right)\ \Longrightarrow\ \bb_d=\1_{n+1}\ \Longrightarrow\ \x^{\bb_d}=\prod_{k=1}^{n+1}x_k
\end{equation*}
\begin{equation*}
\overline{M}_{\aa_d,\bb_d}+\overline{B}=\L_{\aa_d}^T\cdot \overline{V}+\overline{B}=\left(
                                                                                      \begin{array}{c|c|c}
                                                                                      \1_{d}^T&dI_{d}&\1_{d,n+1-d}\\
                                                                                      \hline                                                                                       \1_{n+1-d}^T & \0_{n+1-d,d} & (d+1)I_{n+1-d}
                                                                                      \end{array}
                                                                                    \right)
\end{equation*}
\begin{eqnarray*}
  &\Longrightarrow&f_d^\vee=\prod_{k=1}^{n+1}x_k\cdot\left(\psi+\sum_{j=d+1}^{n+1}x_j^{d}\right)+\sum_{i=1}^d x_i^d\quad(\text{up to a variables rescaling})\\
  &\Longrightarrow& f^\vee_{\bb_d}=\psi+\sum_{j=d+1}^{n+1}x_j^{d}+{\sum_{i=1}^d x_i^d\over \prod_{k=1}^{n+1}x_k}
\end{eqnarray*}
To prove the last sentence of the statement, recall Givental's notation in \cite[Thm.~5]{Givental-ICM} and consider the following re-parameterization of Givental variables $u_1,\ldots,u_N, q$, by setting $N=n+1$ and
\begin{eqnarray}\label{riparametrizzazione}
    \text{for}\ d=1,\quad u_i&=&\begin{cases}\begin{array}{cc}
                                    x_i/\psi &\text{for}\ 1\leq i\leq N-1\\
                                     1/\left(\psi\prod_{j=1}^n x_j\right)&\text{for}\ i=N
                                   \end{array}
    \end{cases}\\
\nonumber
    \text{for}\ 2\leq d\leq n,\quad u_i&:=&\begin{cases}\begin{array}{cc}
x_i^d/\left( \psi\prod_{j=1}^{n+1} x_j\right)& \text{for}\ 1\leq i\leq d\\
                                     x_i^d/\psi& \text{for}\ d+1\leq i\leq N
                                   \end{array}
    \end{cases}\\
\nonumber
    q&:=&\psi^{-N}
\end{eqnarray}
The Givental LG model is given by the superpotential
\begin{equation*}
    F: \C^N\longrightarrow \C\ ,\quad F(\uu):=\sum_{i=1}^N u_i = (1/\psi)f^\vee_{\bb_d}-1
\end{equation*}
restricted to the torus fibration
\begin{equation*}
  \pi:\C^N\longrightarrow \C\ ,\quad \pi(\uu):=\prod_{i=1}^N u_i=q
\end{equation*}
Setting, in Givental's notation,
\begin{equation*}
  \omega_q:={du_1\wedge\cdots\wedge du_N\over dq}
\end{equation*}
a complete set of solutions of the Picard-Fuchs differential system controlling the quantum VHS on $Y_d$ is described by the \emph{oscillating integrals}
\begin{equation*}
    I(\log q)=\int_{\Ga\subset\pi^{-1}(q)} \omega_q e^{F(\uu)/\hbar}
  \end{equation*}
  $\Ga$ being cycles related to critical points of either $F|_{\pi^{-1}(q)}$. Then, we get the following re-parameterized solutions
\begin{equation*}
  I=e^{-1/\hbar}\int_{\Ga\subset\T_{\aa_d}} \omega_{\psi} e^{f^\vee_{\bb_d}/\hbar\psi}
\end{equation*}
where $\omega_\psi$ is the re-parameterization of $\omega_q$ by means of relations (\ref{riparametrizzazione}), and $\Ga$ are now cycles related with critical points of $f^\vee_{\bb_d}$\,.
\end{proof}

 \begin{remark}
   Notice that for $d\geq 2$, $\XX_{\aa_d}$ is covered by at least two open affine subsets. In particular, $f^\vee_d$ restricted to one such open affine subset becomes of the same shape as $f^\vee_1$, that is, setting e.g. $x_1=1$, one gets
   \begin{equation*}
     f^\vee_d|_{\{x_1=1\}}= \prod_{k=2}^{n+1}x_k\cdot\left(1+\sum_{j=d+1}^{n+1}x_j^{d}\right)+1+\sum_{i=2}^d x_i^d=  \prod_{i=1}^{n}y_i\cdot\left(1+\sum_{i=1}^{n}y_i^d\right)+1+ \sum_{i=1}^{d-1}y_i^d
   \end{equation*}
   by setting $y_i=x_{i+1}$. In particular, imposing $d=1$, the right hand side gives $f_1^\vee(\y)$.

   Moreover, for $d\geq n+1$, the construction above is precisely the one already analyzed in \S\ref{sez:ipersuperfici}.
 \end{remark}

 \begin{remark}\label{rem:LGparagone}
   If $d=1$ then $Y_d\cong\P^{n-1}$ embedded in $\P^n$ by setting $x_1=0$. One can then check the relation between the LG mirror model $((\C^*)^n, f_{\bb_1}^\vee)$ given in Theorem~\ref{thm:d<n+1} and the Givental's LG mirror model as given, e.g., in the Introduction of \cite{GKR} and in Ex.~2.2 of \cite{KKP}. In particular, the LG model here presented turns out to be the section $x_{n+1}=1$ of the LG model presented in \cite{GKR}, after the dimensional correction needed to comparing the two constructions. Moreover, the LG model presented in \cite{KKP} is, up to a translation by $1+\psi$, the section $x_n=1$ of the one here presented.
 \end{remark}

\section{Further examples, remarks and open problems}\label{sez:open}

This final section is devoted to collect a series of suggestions and perspectives coming from the pre\-vious treatment of $f$-duality and the induced mirror web, which will be the main object of forthcoming works. Let us first of all recall, in order of appearance, main problems that arose earlier.
\begin{enumerate}
  \item Understanding the generalized Krawitz duality and LG/Hypersurfaces correspondence as sketched in \S\ref{sssez:LG/Hyp} and in particular Remark~\ref{rem:K-dualita}.
  \item By \S\ref{ssez:KKP-compct}, understanding relations between $f$-duality, log geometry and  Intrinsic Mirror Symmetry in the sense of the Gross-Siebert program.
  \item Conjecture~\ref{conj:LGmirror} and, more in general, the HMS implications of $f$-duality, taking into account the previous items, as observed in Remark~\ref{rem:LGmirrors}. As a starting point in proving equivalence of multiple mirror models here and elsewhere proposed, one could consider the work by Doran, Favero and Kelly \cite{DFK}, \cite{Favero-Kelly}.
  \item Check several MS instances, among those listed in \S\ref{ssez:mirrortest}, for some further examples of hypersurfaces and complete intersections in toric varieties.
\end{enumerate}
In the following, we present some further interesting remark and related problems.

 \subsection{What happens when the $f$-process is not calibrated?}\label{ssez:noncalibrato} Recalling Definition~\ref{def:Deltaproc banale}, assume that the $f$-process
 \begin{equation*}
    (X,\aa)\stackrel{f-dual}{\rightsquigarrow}(\XX_\aa,\bb)\stackrel{f-dual}{\rightsquigarrow}(\XX_\bb,\cc)
  \end{equation*}
  is not calibrated. This fact means that $f$-duality cannot be involutive or, in other words, that it is asymmetric: this is not a new situation, as for instance the case of the Givental's Fano/LG model correspondence and, more in general, as for $f$-duality on a weakly framed toric variety just considered in the previous \S\ref{ssez:NegKod}.

  Le us then assume, by definition as done in Definition~\ref{def:dual-ftv}, that:
   \begin{itemize}
     \item $(\XX_\aa,\bb)$ \emph{is the $f$-dual ftv of $(X,\aa)$ and $(\XX_\bb,\cc)$ is the $f$-dual ftv of $(\XX_\aa,\bb)$}.
   \end{itemize}
   Calling $Y,Y',Y''$ the generic hypersurfaces in $|D_\aa|,|D'_\bb|,|D''_\cc|$, respectively, many questions are naturally arising.
   \begin{enumerate}
    \item Is there a relation between $(X,\aa)$ and $(\XX_\bb,\cc)$\,? For instance, is there a birational map $f:\XX_\bb\dashrightarrow X$ such that $D''_\cc=f^{-1}(D_\aa)$\,? If not, may a similar birational transformation relate $(X,\aa)$ with the final ftv obtained after a finite and even number of $f$-dual passages?
     \item Recalling \S\ref{ssez:mirrortest}, which mirror symmetric aspects relate hypersurfaces in the ordered pairs $(Y,Y')$ and $(Y',Y'')$\,?
     \item Is there a relation linking $Y$ and $Y''$? For instance, is it true that $h^{p,q}(\widehat{Y})=h^{p,q}(\widehat{Y}'')$ for suitable resolutions $\widehat{Y}\longrightarrow Y$ and $\widehat{Y}''\longrightarrow Y''$\,? Are they equivalent from the HMS point of view?
   \end{enumerate}
   The present paper is already too long to start analyzing these and related problems, but they are interesting questions to be settled in future works.

\subsection{Generalized complete intersections, BH-transpolarity and $f$-duality}\label{ssez:Hubsch}
Recently, physicists Anderson \emph{et al.} \cite{AAGGL} described a method to produce examples of new \cy varieties which are not compete intersections. The basic idea is taking an hypersurface (or complete intersection) $Y$ in an ambient variety $P$ and then considering hypersurfaces (or complete intersections) $X$ in $Y$ for which there need not exist sections of two (or $r+s$, resp.) line bundles on $P$ whose common zero locus is $X$. The \cy condition is resumed by a constraint on involved degrees of $Y$ and $X$: hence it is not an essential tool of the geometric construction of these varieties, called \emph{generalized} complete intersections (gCI). This method has been further studied by Berglund and H\"ubsch \cite{BH-CYgCI} and rigorously (and nicely) explained in cohomological terms, in the basic case $r=s=1$, by mathematicians Garbagnati and Van Geemen \cite{G-vG_gCI}, who presented $X$ as the zero locus of a global section $\xi$ of a suitable negatively twisted line bundle on $P$, restricted to $Y$.\\
In their preprint \cite{BH-gCY&MS}, Berglund and H\"ubsch conjecturally describe a method to extending Batyrev-Borisov mirror duality on \cy complete intersections to that kind of generalized \cy complete intersections, by means of a, so called, \emph{trans-polarity between VEX polytopes}, that is, a sort of a finite patching of Batyrev-Borisov dualities on convex pieces composing a not necessarily convex polytope, arising as the Newton polytope associated with the global section $\xi$ (\cite[\S3]{BH-gCY&MS}). Very recently, T.~H\"ubsch pointed me out (private communication) that, dropping \cy condition in the above mentioned transpolarity may correspond to replacing BB-duality on the convex pieces by $f$-duality. This observation opens interesting, although possibly intricate, perspectives to extending $f$-duality to generalized complete intersections in a toric ambient variety $P$.

   \subsection{Toric degeneration: extending $f$-duality via geometric transitions}\label{ssez:Tdegenarazione} Following Batyrev's ideas  given in \cite{Batyrev02} (see also \cite[\S6.3]{Rossi-JGP}), since $f$-mirror partners come in families, one can easily extend the $f$-mirror definition to complete algebraic varieties admitting a toric degeneration.

   \begin{definition}\label{def:degenerazione}
     Let $Y$ be a smooth and complete algebraic variety isomorphic to the generic fiber of a flat family $y:\cY\longrightarrow B$, endowed with a special point $0\in B$ such that $Y_0:=y^{-1}(0)$ is isomorphic to a complete intersection subvariety of a complete toric variety $X(\Si)$, determined by a nef-partitioned framing $D_\aa=\sum_kD_{\aa_k}$ of $X$: $Y_0$ is called a \emph{toric degeneration} of $Y$. Assume that the nef-partitioned process associated to $(X,\aa)$ is calibrated. Then the generic complete intersection $Y_0^\vee$, giving a $f$-mirror partner of $Y_0$, is also an \emph{$f$-mirror partner of} $Y$.
   \end{definition}

   \begin{conjecture}\label{conj:topmirror}
     In the same notation of the previous Definition~\ref{def:degenerazione}, there exists a partitioned ftv $(X,\aa=\sum_k\aa_k)$ and a suitable resolution $\widehat{Y}_0^\vee\longrightarrow Y_0^\vee$ such that the $f$-mirror partner $Y_0^\vee$ of $Y$ is a topological mirror partner of $Y$, that is,
     \begin{equation*}
       k_{\widehat{Y}_0^\vee}=m_Y\quad\text{and}\quad k_Y=m_{Y_0^\vee}
     \end{equation*}
   \end{conjecture}

   Notice that, calling $\widehat{Y}_0\longrightarrow Y_0$ a resolution of singularities, the process
   \begin{equation}\label{g.t.}
           \xymatrix{\widehat{Y}_0\ar[r]&Y_0\ar@{<~>}[r]&Y}
         \end{equation}
         is a \emph{geometric transition} (see \cite[Def.~1.4]{Rossi-JGP} for a definition, here considered in a broader sense, beyond the \cy setup).
    Recalling Morrison's argumentation  given in \cite{Morrison} (see also \cite[\S6.2]{Rossi-JGP}), the extension of Batyrev's mirror duality, given by $f$-mirror duality, allows one to formulate, beyond the \cy setup, the following

   \begin{conjecture}[of reverse transition]\label{conj:reverset.}
     Under notation given in Definition~\ref{def:degenerazione} and Conjecture~\ref{conj:topmirror}, and given the geometric transition (\ref{g.t.}), there should exist a \emph{reverse} geometric transition
     \begin{equation*}
       \xymatrix{\widehat{Y}_0^\vee\ar[r]&Y^\vee_0\ar@{<~>}[r]& Y^\vee}
     \end{equation*}
     such that $Y_0$ is a topological mirror partner of $Y^\vee$, that is,
     \begin{equation*}
       k_{\widehat{Y}_0}=m_{Y^\vee}\quad\text{and}\quad k_{Y^\vee}=m_{Y_0}
     \end{equation*}
     In particular, $Y_0^\vee$ is a toric degeneration of $Y^\vee$, meaning $Y^\vee$ is isomorphic to the generic fiber of a flat family $y^\vee:\mathcal{Y}^\vee\longrightarrow B^\vee$, endowed with a special point $0^\vee\in B^\vee$ such that $(y^{\vee})^{-1}(0^\vee)\cong Y_0^\vee$.
   \end{conjecture}

   \begin{remark}
     Following the lines given in \cite{Batyrev02}, in the \cy setup, examples satisfying  Conjecture~\ref{conj:reverset.} are constructed by means of the monomial-divisor correspondence \cite{AGM}. More or less, the same argumentation may be extended beyond the \cy setup. In fact, the meaning of the monomial-divisor correspondence is that of the differential of the mirror map. Assume that there exist well defined isomorphisms (actually differentials of mirror maps)
     \begin{equation*}
       \xymatrix{\mu_A:K_Y\ar[r]^\cong &M_{Y_0^\vee}}\ ,\quad \xymatrix{\mu'_B:K_{\widehat{Y}_0}\ar[r]^\cong &M_{Y^\vee}}
     \end{equation*}
     where $K_Y,K_{\widehat{Y}_0}$ are the tangent spaces to the \ka moduli spaces of $Y$ and $\widehat{Y}_0$, respectively, and analogously $M_{Y_0^\vee},M_{Y^\vee}$ are the tangent spaces to the complex moduli speces of $Y_0^\vee$ and $Y^\vee$, respectively: assume all of them are well defined! The isomorphism $\mu_A$ comes from the $A$-side topological mirror symmetry of $(Y,Y_0^\vee)$ and the isomorphism $\mu'_B$ comes from the $B$-side topological mirror symmetry of $(Y^\vee,Y_0)$ (recall Definition~\ref{def:A,B-mirror}). The geometric transition (\ref{g.t.}) induces an inclusion $K_Y\hookrightarrow K_{\widehat{Y}_0}$, via the inclusion of the associated Picard groups. Therefore, the subspace $\mu'_B(K_Y)\subset M_{Y^\vee}$ defines a first-order deformation of $Y^\vee$ which should give rise to the toric degeneration to $Y_0^\vee$.
   \end{remark}

\subsection{General hyperelliptic curve}\label{ssez:iperellittica}
The only examples of toric subvarieties considered throughout the present paper are hypersurfaces and complete intersections in some projective space. The present subsection is devoted to consider a general hyperelliptic curve of genus $g\geq 2$ as presented in \cite[\S4.1]{KKOY}, that is, a divisor $Y$, of bi-degree $(2,g+1)$, in the Hirzebruch surface $\FF_0=\P(\cO_{\P^1}\oplus\cO_{\P^1})$. The latter is a toric variety of Picard number $r=2$, hence a substantially different example from the case of $\P^n$. The reader is warmly invited to compare the $f$-mirror (complete) model here proposed with Landau-Ginzburg mirror models proposed in \cite{KKOY} and, for $g=2$, in \cite{Seidel}, and, moreover, for the case of a general curve of genus $g\geq 2$ in \cite{Efimov}, generalizing Seidel's approach.

A fan matrix of $\FF_0$ is given by
\begin{equation*}
  V=\left(
      \begin{array}{cccc}
        1 & -1 & 0 & 0 \\
        0 & 0 & 1 & -1 \\
      \end{array}
    \right)
\end{equation*}
and a framing $D_\aa$ of bi-degree $(2,g+1)$ is given, e.g., by $\aa=(1,1,1,g)$. Then
\begin{equation*}
  \L_\aa=\left(
           \begin{array}{cccc}
             1 & 1 & -1 & -1 \\
             g & -1 & g & -1 \\
           \end{array}
         \right)\ \Longrightarrow\ \D_\aa=\conv(\L_\aa)
\end{equation*}
In particular $D_\aa$ is an ample divisor of $\FF_0$. Recalling (\ref{b}),
\begin{equation*}
  M_\aa^T=\L_\aa^T\cdot V= \left(
                            \begin{array}{cccc}
                              1&-1&g&-g\\
                              1&-1&-1&1\\
                              -1&1&g&-g\\
                              -1&1&-1&1
                            \end{array}
                          \right)\ \Longrightarrow\ \bb=(g,1,g,1)
\end{equation*}
Then
\begin{equation*}
  \D_{\bb}=\conv\left(
             \begin{array}{cccc}
               2g\over g+1 & 0 & 0 & -{2g\over g+1} \\
               -{g-1\over g+1} & 1 & -1 & g-1\over g+1 \\
             \end{array}
           \right)\ \Longrightarrow\ \left[\D_\bb\right]=\NN=\conv(V)
\end{equation*}
so giving $\L_\bb=V$, up to a permutation of columns. Moreover, (\ref{c}) gives
\begin{equation*}
  M_{\aa,\bb}^T=V^T\cdot\L_\aa=M_\aa\ \Longrightarrow\ \cc=(1,1,1,g)=\aa
\end{equation*}
implying that the $f$-process is calibrated, by Theorem~\ref{thm:Deltatriviale}.

\noindent Part (b) of Remark~\ref{rem:famiglie} gives the polynomial $f^\vee\in\Cox(\XX_\aa)\cong\C[x_1,\ldots,x_4]$, defining the generic element $Y^\vee$ of the mirror family,
\begin{equation*}
  f^\vee=c_1x_1^{2g}x_3^{2g}+c_2x_1^{g+1}x_2^2x_3^{g-1}+c_3x_1^{g}x_2x_3^{g}x_4+c_4x_1^{g-1}x_3^{g+1}
  x_4^2+c_5x_2^2x_4^2
\end{equation*}
Up to a variables rescaling, the generic $f^\vee$ can be reduced to the following shape
\begin{equation}\label{dual f}
  f^\vee=x_1^{2g}x_3^{2g}+x_1^{g+1}x_2^2x_3^{g-1}+x_1^{g}x_2x_3^{g}x_4+\psi\, x_1^{g-1}x_3^{g+1}x_4^2+\vf\, x_2^2x_4^2
\end{equation}
As a Cox quotient, $\XX_\aa\cong \left(\C^4\setminus Z\right)/H_\aa$, where the irrelevant locus $Z$ is the union of two plains meeting in the origin of $\C^4$, namely $Z=\{x_1=x_2=0\}\cup\{x_3=x_4=0\}$, and
\begin{equation*}
  H_\aa\cong\left\{\begin{array}{cc}
                     (\C^*)^2 & \text{if $g=2h$ is even} \\
                     (\C^*)^2\times\boldsymbol{\mu}_2 & \text{if $g=2h+1$ is odd}
                   \end{array}\right.\quad h\in\N\setminus\{0\}
\end{equation*}
In particular the weight matrix defining the action of $H_\aa$ over $\C^4\setminus Z$ is given by
\begin{eqnarray*}
  Q_g&=&Q_{2h}=\left(
               \begin{array}{cccc}
                 1 & g & 1 & g \\
                 0 & g+1 & 2 & g-1 \\
               \end{array}
             \right)=\left(
               \begin{array}{cccc}
                 1 & 2h & 1 & 2h \\
                 0 & 2h+1 & 2 & 2h-1 \\
               \end{array}
             \right)\ \text{if $g=2h$}\\
  Q_g&=&Q_{2h+1}\times \boldsymbol{\tau}=\left(
               \begin{array}{cccc}
                 1 & h & 0 & h+1 \\
                 0 & h+1 & 1 & h \\
               \end{array}
             \right)\times\left(
                            \begin{array}{cccc}
                              1 & 1 & 1 & 1 \\
                            \end{array}
                          \right)\ \text{if $g=2h+1$}
\end{eqnarray*}
meaning that the action is given by
\begin{eqnarray}\label{azione_g}
  &\xymatrix{((\l,\mu),\x)\ar@{|->}[r]&(\l x_1,\l^g\mu^{g+1}x_2,\l\mu^2x_3,\l^g\mu^{g-1}x_4)}& \quad\text{if $g=2h$} \\
  \nonumber
  &\xymatrix{((\l,\mu,\pm 1),\x)\ar@{|->}[r]&(\pm\l x_1,\pm\l^h\mu^{h+1}x_2,\pm\mu x_3,\pm\l^{h+1}\mu^hx_4)} &\quad\text{if $g=2h+1$}
\end{eqnarray}
Notice that $f^\vee$ is equivariant \wrt to both these actions. In particular, as an element of $\Cox(\XX_\aa)$, which is graded on
\begin{equation*}
  \Cl(\XX_\aa)\cong\left\{\begin{array}{cc}
                     \Z^2 & \text{if $g=2h$ is even} \\
                     \Z^2\oplus\Z/2\Z & \text{if $g=2h+1$ is odd}
                   \end{array}\right.\quad h\in\N\setminus\{0\}
\end{equation*}
$f^\vee$ turns out to be homogeneous of degree either $(4g,4g)$ or $(2g,2g,\overline{0})$, respectively. In particular, it turns out that, if $g$ is even then $(g+1)D_\bb$ is ample and, analogously, if $g=2h+1$ is odd then $(h+1)D_\bb$ is ample (apply e.g. \cite[Thm.~3]{RT-Ample}).

\subsubsection{Hori-Vafa type LG mirror models}\label{sssez:HViperellittica} In \cite[\S4.2]{KKOY} a LG mirror model of the general hyperelliptic curve of genus $g$ is proposed, by adopting the Hori-Vafa recipe \cite{Hori-Vafa} for an hypersurface of bi-degree $(2,g+1)$ in $\FF_0$. By a different approach, Seidel proposed a further LG mirror model for the case $g=2$ \cite{Seidel}, obtained as an unramified quotient of the Hori-Vafa LG mirror model of a plane quintic curve. Seidel's methods have been generalized by Efimov \cite{Efimov} for every $g\geq 2$. In particular we get a double proposals of LG mirror models for the generic hyperelliptic curve of genus $g\geq 2$. In all these cases, authors checked one direction of HMS.

Recalling what observed in \S\ref{ssez:HoriVafa}, and in particular in \S\ref{sssez:HVinvariante}, we can obtain a further proposal of LG mirror model for the general hyperelliptic curve of genus $g\geq 2$, by considering the LG model $(\L_{g,\vf,\psi},w_{g,\vf,\psi})$ so defined:
\begin{itemize}
        \item $\L_{g,\vf,\psi}\cong(\C^*)^4$ is an irreducible component of the reducible torus complete intersection
        $$\L_{g,\vf,\psi}:=\left\{\tau_1\,y_1^2=
                             x_1^{2g}x_2^2x_3^{2g}x_4^2=
                             \tau_2\,y_2^2\right\}\subset (\C^*)^4\times (\C^*)^2$$
        where $4\vf^{2}=\tau_1=e^{t_1}$ and $4\psi^{2}=\tau_2=e^{t_2}$, being $t_1,t_2\in\C^*$ \ka parameters related with volumes of the two rulings on $\FF_0$;
        \item $w_{g,\vf,\psi}$ is the restriction to $\L_{g,\vf,\psi}$ of the regular function $\widetilde{w}_{g,\vf,\psi}:\C^6\longrightarrow\C$ defined by
            \begin{equation*}
              \widetilde{w}_{g,\vf,\psi}(\x,\y):=x_1^{2g}x_3^{2g}+x_1^{g+1}x_2^2x_3^{g-1}+\psi (x_1^{g-1}x_3^{g+1}x_4^2+y_2)+\vf (x_2^2x_4^2+y_1)
            \end{equation*}
      \end{itemize}
When restricted to $\L_{g,\vf,\psi}$, the superpotential $\widetilde{w}_{g,\vf,\psi}$ can then be rewritten as $f^\vee$ in (\ref{dual f}), that is,
\begin{equation*}
  w_{g,\vf,\psi}(\x)= x_1^{2g}x_3^{2g}+x_1^{g+1}x_2^2x_3^{g-1}+x_1^{g}x_2x_3^{g}x_4+\psi\, x_1^{g-1}x_3^{g+1}x_4^2+\vf\, x_2^2x_4^2 =f^\vee
\end{equation*}
This gives the following global picture, analogous to (\ref{LGquoziente}),
\begin{equation*}
   \xymatrix{ \{\0\}\ar@{^(->}[r]&\C&\L_{g,\vf,\psi}\cong(\C^*)^{4}\ar@{^(->}[r]\ar[l]_-{w_{g,\vf,\psi}}
   \ar@{->>}[d]_-{/H_\aa}&\C^4\setminus Z\ar@{->>}[d]_-{/H_\aa}\\
   w_{g,\vf,\psi}^{-1}(0)/H_\aa\ar[u]\ar@{^(->}[rr]&&\T_\aa\ar@{^(->}[r]&\XX_{\aa}}
 \end{equation*}
 where $\T_\aa$ is the acting torus on $\XX_{\aa}$. Then the $f$-mirror $Y^\vee$ of $Y$, as proposed in Definition~\ref{def:mirror} and defined by $f^\vee\in\Cox(\XX_\aa)$, is precisely the closure
 \begin{equation*}
   Y^\vee=\overline{w_{g,\vf,\psi}^{-1}(0)/H_\aa}\subset\overline{\T}=\XX_{\aa}
 \end{equation*}
 induced by the open embedding $\T_\aa\hookrightarrow \XX_{\aa}$.
  Evidences seem enough to motivating the following
  \begin{conjecture}
    A LG mirror model of the general hyperelliptic curve of genus $g\geq2$ of \ka parameters $t_1,t_2$, is given by $((\C^*)^4,w_{g,\vf,\psi})$, with $2(\ln2+\ln\vf)=t_1$ and $2(\ln2+\ln\psi)=t_2$.
  \end{conjecture}
  Accordingly with Hori-Vafa terminology \cite{Hori-Vafa}, a similar LG model admits an associated gauged linear sigma model whose gauge action is given by the $(\C^*)^2$-action described in (\ref{azione_g}). Quotienting by such a gauge action gives back, up to a possibly further quotient by $\Z/2\Z$ depending on the parity of the genus $g$, the LG model $(\T_\aa,f^\vee_\bb)$ described in \S\ref{sssez:LG/Hyp}, and admitting the (generalized) KKP-compactification $(\XX_\aa,\overline{f}^\vee_\bb)$, described in \S\ref{ssez:KKP-compct}. In particular:
  \begin{equation*}
    (\overline{f}^\vee_\bb)^{-1}(0)=Y^\vee\ ,\quad(\overline{f}^\vee_\bb)^{-1}(\infty)=gD'_1+D'_2+gD'_3+D'_4=D'_\bb
  \end{equation*}
  Also in the present case, the ftv $(\XX_\aa,D'_\bb)$ can be thought of a log (no Calabi-Yau) pair, opening the door to an intrinsic mirror symmetric interpretation, in the sense of Gross-Siebert \cite{GS-IMS}.

Moreover, the Laurent superpotential $f_\bb^\vee$ of the LG model $(\T_\aa,f^\vee_\bb)$ admits the following explicit expression
\begin{equation*}
  f^\vee_\bb={f^\vee\over \x^\bb}= 1+{x_1^g x_3^g\over x_2 x_4}+\vf{x_2 x_4 \over x_1^g x_3^g}+{x_1 x_2\over x_3 x_4}+\psi{x_3 x_4\over x_1 x_1}
\end{equation*}
meaning that $(\T_\aa,f^\vee_\bb)$ is a re-parameterization of a Givental type LG model, the latter defined as follows
\begin{eqnarray}\label{riparametrizzazione2}
   u_1 &:=& {x_1^g x_3^g\over x_2 x_4} \\
\nonumber
 u_2 &:=& \vf{x_2 x_4 \over x_1^g x_3^g} \\
\nonumber
  u_3 &:=& {x_1 x_2\over x_3 x_4} \\
\nonumber
  u_4 &:=& \psi{x_3 x_4\over x_1 x_1} \\
\nonumber
  q_1 &:=& \vf \\
\nonumber
  q_2 &:=& \psi
\end{eqnarray}
The associated Givental LG model is then given by the superpotential
\begin{equation*}
  \xymatrix{F:\C^4\ar[r]&\C}\ :\quad F(\mathbf{u}):=\sum_{i=1}^4u_i=f^\vee_\bb-1
\end{equation*}
restricted to the torus fibration
\begin{equation*}
  \xymatrix{\pi:\C^4\ar[r]&\C^2}\ :\quad \pi(\uu):=(u_1u_2,u_3u_4)=\mathbf{q}:=(q_1,q_2)
\end{equation*}
This is analogous to what done in (\ref{riparametrizzazione1}) and (\ref{riparametrizzazione}) for projective hypersurfaces.

  Checking if the LG mirror model here presented and those proposed in \cite{KKOY}, \cite{Seidel} and \cite{Efimov} are actually each other equivalent from the HMS point of view, is a completely open task! Anyway, let us observe that the Givental type LG mirror model defined in (\ref{riparametrizzazione2}) turns out to be the section $x_3=1$ of the LG mirror model proposed in \cite[\S4.2]{KKOY}, analogously to what observed in Remark~\ref{rem:LGparagone} for the LG mirror models of the hyperplane in $\P^n$.

\bibliography{MILEA}
\bibliographystyle{acm}

\end{document}